\documentclass[fleqn,reqno,11pt,a4paper,final]{amsart}
\usepackage[a4paper,left=20mm,right=20mm,top=20mm,bottom=20mm,marginpar=20mm]{geometry}
\usepackage{amsaddr}
\usepackage{amsmath}
\usepackage{amssymb}
\usepackage{amsthm}
\usepackage{mathtools}
\usepackage[ansinew]{inputenc}
\usepackage{cite}
\usepackage{bbm}
\usepackage{multirow}
\usepackage{xcolor}
\usepackage[english=american]{csquotes}
\usepackage[final]{graphicx}
\usepackage{hyperref}
\usepackage{calc}
\usepackage{tikz}
\usepackage{pgfplots}
\usepackage{graphicx}
\usepackage{stix}
\usepackage{enumitem}
\usepackage{multirow}
\makeatletter
\def\namedlabel#1#2{\begingroup
    #2%
    \def\@currentlabel{#2}%
    \phantomsection\label{#1}\endgroup
}
\makeatother

\graphicspath{{../Pictures/}}

\numberwithin{equation}{section}

\newtheoremstyle{thmlemcorr}{10pt}{10pt}{\itshape}{}{\bfseries}{.}{10pt}{{\thmname{#1}\thmnumber{ #2}\thmnote{ (#3)}}}
\newtheoremstyle{thmlemcorr*}{10pt}{10pt}{\itshape}{}{\bfseries}{.}\newline{{\thmname{#1}\thmnumber{ #2}\thmnote{ (#3)}}}
\newtheoremstyle{remexample}{10pt}{10pt}{}{}{\bfseries}{.}{10pt}{{\thmname{#1}\thmnumber{ #2}\thmnote{ (#3)}}}
\newtheoremstyle{ass}{10pt}{10pt}{}{}{\bfseries}{.}{10pt}{{\thmname{#1}\thmnumber{ A#2}\thmnote{ (#3)}}}

\theoremstyle{thmlemcorr}
\newtheorem{theorem}{Theorem}
\numberwithin{theorem}{section}
\newtheorem{lemma}[theorem]{Lemma}
\newtheorem{corollary}[theorem]{Corollary}
\newtheorem{proposition}[theorem]{Proposition}

\newtheorem{definition}[theorem]{Definition}

\theoremstyle{thmlemcorr*}
\newtheorem{theorem*}{Theorem}
\newtheorem{lemma*}[theorem]{Lemma}
\newtheorem{corollary*}[theorem]{Corollary}
\newtheorem{proposition*}[theorem]{Proposition}
\newtheorem{problem*}[theorem]{Problem}
\newtheorem{conjecture*}[theorem]{Conjecture}
\newtheorem{definition*}[theorem]{Definition}
\newtheorem{assumption*}[theorem]{Assumption}

\theoremstyle{remexample}
\newtheorem{remark}[theorem]{Remark}

\newtheorem*{remark*}{Remark}

\theoremstyle{ass}

\newcommand{\Dcal}{\mathcal{D}}

\newcommand{\T}{\mathbb{T}}

\newcommand{\A}{\mathscr{A}}
\newcommand{\B}{\mathscr{B}}
\newcommand{\D}{\mathscr{D}}

\newcommand{\Newt}{\mathscr{N}}
\newcommand{\Haus}{\mathscr{H}}
\newcommand{\Pow}{\mathscr{P}}
\newcommand{\Mon}{\mathscr{M}}
\DeclareMathOperator{\Lin}{Lin}
\newcommand{\tigma}{\tilde{\sigma}}
\newcommand{\funct}{\mathscr{f}}
\newcommand{\gunct}{\mathscr{g}}
\newcommand{\QA}{\mathscr{Q}_{\mathscr{A}}}
\newcommand{\test}{\mathscr{T}_{\mathscr{A}}}
\newcommand{\topoeq}{\mathscr{T}_{\textup{eq}}}
\newcommand{\topobd}{\mathscr{T}_{\textup{bd}}}
\newcommand{\Ccal}{\mathscr{C}}
\definecolor{Gump}{rgb}{0,0.6,0.4}
\definecolor{Hanks}{rgb}{0.7,0.3,0.1}

\DeclareMathOperator{\id}{id}

\DeclareMathOperator{\diverg}{div}
\DeclareMathOperator{\Div}{div}
\DeclareMathOperator{\curl}{curl}
\DeclareMathOperator{\dist}{dist}

\DeclareMathOperator{\Tr}{tr}

\DeclareMathOperator{\const}{const.}
\DeclareMathOperator{\lin}{lin}
\DeclareMathOperator{\nl}{nl}

\newcommand{\norm}[1]{\|#1\|}

\newcommand{\abs}[1]{|#1|}

\newcommand{\dd}{\;\mathrm{d}}

\newcommand{\N}{\mathbb{N}}
\newcommand{\R}{\mathbb{R}}

\newcommand{\Z}{\mathbb{Z}}

\newcommand{\sym}{\mathrm{sym}}
\newcommand{\skw}{\mathrm{skew}}

\newcommand{\eps}{\epsilon}
	
\newcommand{\limdata}{\overset{bd}{\longrightarrow}}
\newcommand{\limeq}{\overset{eq}{\longrightarrow}}
\newcommand\bra[1]{\left(#1\right)}

\def\XXint#1#2#3{{\setbox0=\hbox{$#1{#2#3}{\int}$}
\vcenter{\hbox{$#2#3$}}\kern-.5\wd0}}

\definecolor{luh-dark-blue}{rgb}{0, 0.313, 0.608}
\definecolor{luh-light-blue}{rgb}{0.6, 0.725, 0.847}
\definecolor{luh-green}{rgb}{0.784, 0.827, 0.09}

\begin{document}


\title[Data-driven Fluid Mechanics]{A data-driven approach to viscous fluid mechanics -- the stationary case}

 \author{Christina Lienstromberg, Stefan Schiffer$^\ast$, Richard Schubert$^\ast$}
\address{Institute of Analysis, Dynamics and Modeling, University of Stuttgart, Pfaffenwaldring~57, 70569 Stuttgart, Germany}
\email{christina.lienstromberg@iadm.uni-stuttgart.de}
\email{schiffer@iam.uni-bonn.de}\address{$^\ast$ Institute of Applied Mathematics, University of Bonn, Endenicher Allee~60, 53115 Bonn, Germany}
\email{schubert@iam.uni-bonn.de}

\begin{abstract}
We introduce a data-driven approach to the modelling and analysis of viscous fluid mechanics. Instead of including constitutive laws for the fluid's viscosity in the mathematical model, we suggest to directly use experimental data. Only a set of differential constraints, derived from first principles, and boundary conditions are kept of the classical PDE model and are combined with a data set. 
The mathematical framework builds on the recently introduced data-driven approach to solid-mechanics \cite{KO,CMO}.
We construct optimal data-driven solutions that are \emph{material model free} in the sense that no assumptions on the rheological behaviour of the fluid are made or extrapolated from the data.
The differential constraints of fluid mechanics are recast in the language of constant rank differential operators. Adapting abstract results on lower-semicontinuity and $\mathscr{A}$-quasiconvexity, we show a $\Gamma$-convergence result for the functionals arising in the data-driven fluid mechanical problem. The theory is extended to compact nonlinear perturbations, whence our results apply to both inertialess fluids and flows with finite Reynolds number.  
Data-driven solutions provide a new \emph{relaxed} solution concept.
We prove that the constructed data-driven solutions are consistent with solutions to the classical PDEs of fluid mechanics if the data sets have the form of a monotone constitutive relation.
 
\end{abstract}



\maketitle
\bigskip


\bigskip

\noindent\textsc{Acknowledgement. } We thank Michael Ortiz for insightful discussions. Moreover, we acknowledge support by the Hausdorff Center for Mathematics (GZ 2047/1, Project-ID 390685813). C.~L. has been supported by the Deutsche Forschungsgemeinschaft (DFG, German Research Foundation) through the collaborative research centre `The mathematics of emergent effects' (CRC 1060, Project-ID  211504053).




\section{Introduction}

In this article, a new approach to the modelling and analysis of viscous fluid mechanics is introduced. The hydrostatic behaviour of an incompressible fluid at any instant $t$ in time may be described by its velocity field $u\colon x\mapsto u(x)\in\R^d$ which induces a strain(-rate) $\epsilon\colon x\mapsto \epsilon(x)\in \R^{d\times d}_{\sym}$
\begin{align}\label{eq:strain}
  \epsilon=\frac 12\left(\nabla u+\nabla u^T\right),  
\end{align} 
the symmetric gradient of the velocity field. Moreover the fluid generates a stress field $\sigma\colon x\mapsto \sigma(x)\in \R^{d\times d}_{\sym}$ which, in the case of an inertialess fluid, satisfies 
\begin{align}\label{eq:stress}
    -\diverg \sigma=f,    
\end{align} with an external force density $f\colon x\mapsto f(x)\in \R^d$. Both \eqref{eq:strain} and \eqref{eq:stress} are prescribed \emph{differential constraints} and are also called \emph{compatibility conditions}. The strain $\epsilon$ and the stress $\sigma$ cannot be \emph{any} field -- they have to be a symmetric gradient of another field in the first, and admit a predefined divergence in the second case. For fluids with finite Reynolds number the force balance \eqref{eq:stress} has to be complemented by the inertial forces proportional to $\partial_t u+(u\cdot \nabla)u$. This results (after suitable non-dimensionalisation) in the equation
\begin{align*}
    \partial_t u+(u\cdot \nabla)u-\diverg \sigma=f.
\end{align*}
However, in this paper we restrict our analysis to the stationary case $\partial_t u = 0$, i.e. we study the problem
\begin{equation*}
    (u\cdot \nabla)u-\diverg \sigma
    =
    f.
\end{equation*}
Since our analysis is mainly based on variational arguments suited for stationary problems, we postpone the time-dependent case to a separate work.

\bigskip
\subsection{The PDE-Based Approach -- Constitutive Laws for Viscous Fluids.}

Hitherto, the modelling and analysis of a rich set of phenomena in viscous fluid mechanics relies on \textit{constitutive laws} describing the relation between the strain field $\eps$ and the stress field $\sigma$. A commonly used relation is 
\begin{equation*}
    \sigma
    =
    -\pi\id 
    +
    2\mu(\abs{\epsilon}) \epsilon,
\end{equation*}
which relies on the assumption that the stress comprises two components -- the hydrostatic stress $\pi \id$ and the viscous stress $2\mu(\abs{\epsilon})\epsilon$. Here, $\mu\colon s\mapsto \mu(s)\in \R_+$ denotes the viscosity of the fluid. It depends on the strain rate and measures the resistance of the fluid to deformation. Mathematically, the hydrostatic pressure $\pi\colon x\mapsto \pi(x)\in\R$ is the Lagrange multiplier corresponding to the incompressibility condition $\Div u=0$. In the simplest model of a viscous fluid, the viscosity $\mu$ is assumed to be constant $\mu\equiv\const$ and the corresponding fluid is called \textit{Newtonian}. In other words, the relation between the viscous forces and the local strain rate is perfectly linear, the constant viscosity being the factor of proportionality. In the case of an \emph{inertialess} incompressible Newtonian fluid one obtains the well-known Stokes equations
\begin{equation} \label{intro:Stokes}
    \begin{cases}
        -\mu\Delta u+\nabla \pi = f &
         \\
        \Div u = 0. &
    \end{cases}
\end{equation}
For incompressible Newtonian fluids \emph{with inertia}  one obtains the stationary Navier--Stokes equations
\begin{equation} \label{intro:Navier-Stokes}
    \begin{cases}
        (u\cdot \nabla)u - \mu\Delta u+\nabla \pi = f &
         \\
        \Div u = 0. &
    \end{cases}
\end{equation}
Although it is reasonable in many practical applications to assume a fluid being Newtonian, real fluids that account for viscosity are in fact non-Newtonian, i.e. they feature a nonlinear relation between the stresses $\sigma$ and the rate of strain $\eps$. A widely-used constitutive relation is given by 
\begin{equation} \label{eq:power-law}
    \mu(|\epsilon|)= \mu_0 \abs{\epsilon}^{\alpha-1}, \quad \alpha>0,
\end{equation}
and the corresponding fluid's are called \textit{power-law fluids} or \textit{Ostwald--de Waele fluids}. The exponent $\alpha > 0$ denotes the so-called \textit{flow-behaviour exponent} and $\mu_0 > 0$ is the \emph{flow consistency index}. In the case $0 < \alpha < 1$ the fluid exhibits a \textit{shear-thinning} behaviour as its viscosity decreases with increasing shear-rate, while the fluid is called \textit{shear-thickening} in the case $\alpha > 1$. In this case the viscosity is an increasing function of the shear rate. The corresponding stationary non-Newtonian Navier--Stokes system reads
\begin{equation} \label{intro:non-Newtonian_NS}
    \begin{cases}
        (u\cdot \nabla)u - \Div\bigl(2\mu(|\epsilon(u)|) \epsilon(u)\bigr) +\nabla \pi = f &
         \\
        \Div u = 0. &
    \end{cases}
\end{equation}

For $\alpha=1$ we recover a Newtonian behaviour. In practice, constitutive laws for the viscosity are derived from experimental measurements. This is done by determining the parameters inside a prescribed class of laws, for instance $\mu_0$ and $\alpha$ in the case of power-law fluids \eqref{eq:power-law}, to best approximate the measured data.
A large part of the mathematical knowledge in the mechanics of viscous fluids comes from the theoretical and numerical analysis of partial differential equations such as Stokes equation and Navier--Stokes equation, that are derived using constitutive laws. Here, a lot of progress has been made by allowing for increasingly general classes of (nonlinear) viscosity laws (see for example \cite{Lad67, MNR:1993, MRS05, MPS06}).


\medskip
\subsection{A Data-Driven Approach.}

Nowadays, the availability of big  data and the possibility to mine them is increasing drastically. 
In the present work, instead of including constitutive laws in the mathematical models, we suggest to directly use experimental data in order to find the strain rate $\eps$ and the stress $\sigma$ that satisfy the respective differential constraints and, at the same time, approximate the experimental data best.
In order to realise this mathematically, we are inspired by the articles \cite{KO,CMO}, where a similar approach has first been  introduced in the context of solid mechanics. 

In the present paper, \emph{data sets} consist of strain-stress pairs $(\epsilon,\sigma)$, which we think of as being extracted from an experiment. These data might be obtained by preprocessing the information coming from actual measurements of other physical quantities. We emphasise that the step of preprocessing is also necessary when deriving constitutive laws from measurements. 

\medskip
The motivation for replacing the classical PDE-based approach by the data-driven approach is the following. Once one accepts the fundamental assumptions (first principles) about the nature of the fluid leading to the differential constraints, the PDE-based approach generates two errors with respect to modelling the real world: First, the experimental equipment is imperfect, leading to \emph{measurement errors}. Second, the fitting of a material law to the experimental data introduces a \emph{modelling error}. The \textit{data-driven approach} entirely skips this second step. \\
Turning to the remaining source of errors, with perfect equipment and infinitely many measurements, we expect to recover the viscosity law of the fluid (if it exists). In reality, measurements are however restricted by 
\begin{itemize}
    \item the inaccuracy of the equipment leading to a measurement error;
    \item a limited number of data points. This comprises both `density of measurements' (i.e. given a strain $\epsilon$, how many data points lie in a neighbourhood of $\epsilon$?), as well as `range of measurement' (how large is the range of values of $\epsilon$ that can be measured in the experiment?).
\end{itemize} 
Nevertheless, if over the course of several consecutive measurement series the measurement error decreases or the density and range of data points increases, we expect the experimental data to converge to the material law. Mathematically, we give consideration to this behaviour by
introducing different notions of \textit{data convergence}. In this paper, we restrict ourselves to the study of the following two settings:
\begin{itemize}
    \item data with increasing quality and an unbounded range of measurements;
    \item data with increasing quality and     a bounded but increasing range of    measurements.
\end{itemize}
An overview of the possible settings and where they are discussed in this paper is given in Table \ref{tab}. \\

\bigskip

\begin{center}
\begin{table}[h!]
\begin{tabular}{|c|c|c|c|}
    \hline
    \multicolumn{2}{|c|}{} & \multicolumn{2}{c|}{\textbf{Range of measurement }} \\
    \cline{3-4}
    \multicolumn{2}{|c|}{} & \textbf{Constant (unbounded)} & \textbf{Increasing} \\
    \hline
    \multirow{2}{*}{\textbf{Error}} & \textbf{Constant (no improvement)} & Need to deal with ''bad`` data & Need to deal with ''bad`` data
    \\
    \cline{2-4}
    & \textbf{Decreasing} & Section \ref{sec:dataconvbd} & Section \ref{sec:dataconveq} \\
    \hline
\end{tabular}
\vspace{0.3cm}
\caption{\label{tab} Measurement error and range of measurement.}
\end{table}
\end{center}


In the case of non-increasing accuracy, measurements for a given strain rate $\epsilon$ might be located in a neighbourhood of the exact value with a certain \emph{likelihood}. In this case, the set of data converges in a weak sense to some distribution, see \cite{CHO}. See also \cite{RS} for the analysis of single outliers in measurements.

\medskip
\subsection{Mathematical Approach for the Data-Driven Problem and Main Results.}
We follow the mathematical approach proposed in \cite{CMO} in a solid mechanical context. To this end, we first split the stress $\sigma = -\pi \id + \tigma$ into its hydrostatic part $\pi \id = -\tfrac{1}{d} \Tr (\sigma) \id$ and its viscous part $\tigma$.

Throughout the paper we assume that the \emph{data set} $\D$ comprises pairs  $(\epsilon,\tigma)$ of strain and viscous stress only. The hydrostatic pressure $\pi$ (i.e. the trace of $\sigma$) is not included in the data set, since we allow $\pi$ to attain arbitrary values. This is due to the fact that the pressure does not play a role in the constitutive law for the viscosity but arises as a Lagrange multiplier corresponding to the incompressibility constraint.

Given a \emph{data set} $\D_n=\{(\epsilon_\beta,\tigma_\beta)\}_{\beta\in B_n}$, consisting of pairs $(\epsilon_\beta,\tigma_\beta)$ of symmetric and trace-free matrices in $\R^{d\times d}$, we 
consider the functional 
\begin{equation}\label{intro:I}
    I_n(\epsilon,\tigma) = 
    \begin{cases} 
        \int_{\Omega} \dist\left(\left(\epsilon(x),\tigma(x)\right), \D_n\right) \dd x, &  (\epsilon,\tigma) \in \Ccal 
        \\ 
        \infty, & \text{else},
    \end{cases}
\end{equation}
as a measure for the distance of functions $(\epsilon,\tigma)$, defined on a simply connected and bounded $C^1$-domain $\Omega\subset \R^d$, to the data set.
Here, $\Ccal$ is the \textit{constraint set} of fields $\epsilon,\tigma$ satisfying the prescribed differential constraints and suitable boundary conditions, and $\dist(\cdot,\cdot)$ is a suitable distance function.

In the present paper, the set of differential constraints is given by \eqref{eq:strain} in combination with either the inertialess force balance or the stationary Navier--Stokes force balance. That is, we study both the \emph{linear constraint set}
\begin{equation} \label{intro:linear_C}
    \begin{cases}
        \epsilon = \frac{1}{2}\left(\nabla u + \nabla u^T\right) 
        &
        \\
        \Div u = 0 &
         \\
        -\Div \tigma = f - \nabla \pi, &
    \end{cases}
\end{equation}
as well as the \emph{nonlinear constraint set}
\begin{equation} \label{eq:nonlinear_C}
    \begin{cases}
        \epsilon 
        = 
        \frac{1}{2}\left(\nabla u + \nabla u^T\right) 
        &
        \\
        \Div u = 0 &
         \\
        -\Div \tigma = f - (u\cdot \nabla)u - \nabla \pi. 
        &
    \end{cases}
\end{equation}
The set of constraints is complemented by suitable boundary conditions. Typical boundary conditions in fluid mechanics are the \emph{no-slip condition}
\begin{align}\label{eq:BC_Dirich}
    u=0\quad \text{on }\partial \Omega
\end{align}
and the \emph{Navier-slip condition} 
\begin{align}\label{eq:BC_Navier}
   \begin{cases}
        \tau\cdot\bra{\sigma \nu+\lambda u}=0,
        &
        \tau\in T\partial\Omega
        \\
        u\cdot \nu=0 
        &
        \text{on }\partial \Omega.
    \end{cases}
\end{align}
Here, $\lambda\ge 0$ is the inverse of the so-called slip length and $\nu$ denotes the outer normal to $\partial\Omega$. Moreover, $T\partial\Omega$ denotes the tangential bundle of $\partial\Omega$. The case of free slip $\tau\cdot\sigma \nu=0$ for $\tau\in T\partial\Omega$ is included via $\lambda=0$. The second condition in \eqref{eq:BC_Navier} expresses the non-permeability of the boundary.

Less natural is the \emph{Neumann type boundary condition}
\begin{align}\label{eq:BC_Neum}
    \sigma \nu=0
    \quad 
    \text{on }\partial \Omega.
\end{align}
In the linear case \eqref{intro:linear_C}, we are able to handle all three types of boundary conditions \eqref{eq:BC_Dirich}, \eqref{eq:BC_Navier}, and \eqref{eq:BC_Neum}. In the nonlinear case \eqref{eq:nonlinear_C}, we are able to handle the physical boundary conditions \eqref{eq:BC_Dirich} and \eqref{eq:BC_Navier}. In some cases we allow for \emph{inhomogeneous} boundary conditions, i.e. non-zero right-hand sides.

Coming back to \eqref{intro:I}, a minimiser (or a minimising sequence) of the functional $I_n$ always satisfies the compatibility conditions for $\epsilon$ and $\tigma$ and is as close to the experimental data $\D_n$ as possible. 

In the case in which a sequence $\D_n$ of data sets approximates a limiting set $\D$, corresponding to a constitutive law, it is expected that the minimisers $v_n=(\epsilon_n,\tigma_n)$ of the functional $I_n$ converge to a solution $v$ of the PDE corresponding to the constitutive law. One main contribution of the present article is to specify conditions under which this is true.
We use the following notion for convergence of data sets.
\begin{definition} \label{intro:defi}
We say that a sequence of closed sets $\D_n$ converges to $\D$, $\D_n \to \D$, if the following is satisfied. 
\begin{enumerate}[label=(\roman*)]
    \item\label{item:fine} \textbf{Fine approximation on bounded sets:} There are sequences $a_n\to 0$ and $R_n\to\infty$ such that for all $n\in\N$ and for all $z \in \D$ with $|z| < R_n$, it holds that
    \[
        \dist(z, \D_n) 
        \leq 
        a_n (1 + |z|).
    \]
    \item\label{item:uniform} \textbf{Uniform approximation on bounded sets:} There are sequences $b_n\to 0$ and $S_n\to\infty$ such that for all $n\in\N$ and for all $z_n \in \D_n$ with $|z_n| < S_n$, it holds that 
    \[
        \dist(z_n,\D) 
        \leq 
        b_n (1 + |z_n|).
    \]
\end{enumerate}
Here, $|\cdot| = \dist(\cdot,0)$ defines a pseudo-norm. 
\end{definition}

The sequences $a_n$ and $b_n$ represent the relative error, while $S_n$ and $R_n$ describe the measurement range. Note that condition \ref{item:fine} ensures that every point in the limiting set is approximated by data points in $\D_n$ while condition \ref{item:uniform} ensures that the $\D_n$ approximates $\D$ uniformly.

Moreover, the notion of convergence introduced in Definition \ref{intro:defi} (ii) is justified from an experimental point of view. Indeed, for a given experimental setup we expect the measurements to be precise only within a certain range, $|z| \leq S_n$. For instance, in the experiment conducted by \textsc{Couette} \cite{Couette}, the aim of which was to measure the viscosity of a fluid, the range $S_n$ is linked to the aspect ratio of the rotating cylinders. 
In the setting of this article, the absolute error is allowed to grow with the range of measurements, which extends the setting studied in \cite{CMO}, where the absolute errors are required to converge to zero. 

From a mathematical point of view, the above notion of convergence is justified by the observation that we may restrict the analysis to $p$-equi-integrable recovery sequences in the $\Gamma$-convergence result below. Indeed, the first main result of this article is the following.
\begin{itemize}
    \item \textbf{$\Gamma$-convergence} (Theorem \ref{thm:gammaconvlin} and Theorem \ref{thm:gammaconvsemilin}): If $\D_n \to \D$ and the $\D_n$ satisfy a certain growth condition, then $I_n$ $\Gamma$-converges to
        \begin{equation*}
            I^\ast(\epsilon,\tigma)
            =
            \begin{cases}
                \int_\Omega \QA \dist\bigl(\left(\epsilon(x),\tigma(x)\right),\D\bigr) \dd x, & (\epsilon,\tigma) \in \Ccal
                \\
                \infty, 
                & \text{else},
            \end{cases}
        \end{equation*}
        where $\QA$ is a suitable convex envelope of the distance function corresponding to the differential operators defining the compatibility conditions \eqref{eq:strain} and \eqref{eq:stress}.
\end{itemize}
There are two main challenges in the proof of this result. One difficulty is the suitable modification of sequences of functions while preserving differential constraints and given boundary conditions. 
To overcome this challenge we prove the following result, which might be of independent interest.
\begin{itemize}
    \item \textbf{$p$-equi-integrability and boundary conditions} (Theorem \ref{lemma:equiintboundary}): If a weakly convergent sequence $u_n$ of $L_p$-functions on $\Omega \subset \R^d$ satisfies some differential constraint $\A u_n=0$ for a constant coefficient (and constant rank) differential operator $\A$, we can modify $u_n$ slightly in the sense of closeness in $L_r$ for $r < p$. The modified sequence still satisfies the differential constraint and the same boundary conditions, but is $p$-equi-integrable (i.e. no concentrations of mass occur).
\end{itemize}
This modification result, together with slight adaptations of existing theory on relaxation subject to a \emph{linear} differential constraint, yields Theorem \ref{thm:gammaconvlin}. Moreover, overcoming the second challenge, we show that the relaxation result continues to hold when we include \emph{compact nonlinear perturbations} in the constraint set $\Ccal$, see Theorem \ref{thm:relaxationnonlinear}. This includes in particular the inertia term $[u \mapsto (u\cdot \nabla)u]: W^1_p(\Omega;\R^d) \to W^{-1}_q(\Omega;\R^d)$, whenever $p> 3d/(d+2)$ and $1/p + 1/q = 1$.

In the case of a data set given by a constitutive law, data-driven solutions provide a \emph{new solution concept}. Another main result of this article proves that, in the case of \emph{monotone} constitutive laws, this solution concept is compatible with the concept of weak solutions to PDEs:
\begin{itemize}
     \item \textbf{Consistency} (Section \ref{sec:exact}): If the data set $\D$ corresponds to a monotone constitutive law, e.g. $\D=\{(\epsilon,\vert \epsilon \vert^{\alpha-1}\epsilon)\}$ in the case of power-law fluids, and if the corresponding PDE admits a solution, then for a map $v=(\epsilon,\tigma)$ the following three statements are equivalent:
     \begin{enumerate}[label=(\roman*)]
        \item $v$ is a minimiser of $I^\ast$, i.e. a solution to the \emph{relaxed data-driven problem};
        \item $I^\ast(v) = 0$, i.e. there exists a sequence $v_n \rightharpoonup v$ with $I(v_n) \to 0$;
        \item $v$ is a solution to the corresponding PDE (i.e. to \eqref{intro:non-Newtonian_NS} in the nonlinear case) in the classical weak sense.
    \end{enumerate}
\end{itemize}
In the case of \emph{non-monotone} constitutive laws, the requirement $I^\ast(v)=0$ amounts to a relaxed solution concept that might be useful for instance in order to deal with viscoelastic fluids.

\medskip

\subsection{Outline of the Paper}
 Section \ref{sec:prelim} shows how the fluid mechanical problems fit into the general theory of constant rank operators. In Section~\ref{subsec:notation} we introduce relevant notation and recall the notion of $\Gamma$-convergence with respect to the weak topology of $L_p$-spaces. In Subsection \ref{sec:Aqc} we recall the generalised form of Problem \eqref{intro:I}, where the differential constraint $(\epsilon,\tigma) \in \Ccal$ is written abstractly as $\A v=0$ and the distance function is replaced by some function $\funct(x,v)$. In Section \ref{sec:diff:fluids} it is demonstrated that the fluid mechanical setting fits into this abstract framework.

An abstract theory for lower-semicontinuity of functionals under linear differential constraints has been developed by \textsc{Fonseca} \& \textsc{M\"uller} (\cite{FM}, see also \cite{BFL}) and we recall these results at the beginning of Section \ref{sec:wlsc}. The remainder of Section \ref{sec:wlsc} is devoted to the modification of the corresponding arguments to fit the fluid mechanical setting of the present paper. In particular, we show the crucial Theorem \ref{lemma:equiintboundary}, which allows us to modify sequences to be equi-integrable, while still respecting both the differential constraints and the boundary conditions. This result is used to extend relaxation results, previously obtained in \cite{BFL}, to the situation of a semilinear differential constraint in Theorem \ref{thm:relaxationnonlinear}.

\smallskip

For Sections \ref{sec:dataconv}--\ref{sec:exact} we return to the fluid mechanical setting and apply the abstract results of Section \ref{sec:wlsc}. 

In Section \ref{sec:dataconv} we discuss two different notions of \emph{data convergence} on a purely set-theoretic level; in particular these notions of convergence are not directly connected to the differential constraints. First, in Subsection \ref{sec:dataconvbd} we introduce a form of data convergence which corresponds to fixed range of measurement (lower-left entry of Table~\ref{tab}) and show that this is equivalent to a suitable notion of convergence for the unconstrained functionals
\begin{equation}\label{intro:J}
    J_n(\epsilon,\tigma) = \int_{\Omega} \dist\left(\left(\epsilon(x),\tigma(x)\right), \D_n\right) \dd x.
\end{equation}
For results about $\Gamma$-convergence of \emph{constrained} functionals of type \eqref{intro:I}, however, we can weaken the notion of convergence to Definition \ref{intro:defi}. This type of convergence is examined in Section \ref{sec:dataconveq}. The reason for this convergence being of interest for $\Gamma$-convergence, is discussed already at the beginning of Section \ref{sec:wlsc} in Theorem \ref{theorem:equiint}.

The abstract results of Section \ref{sec:wlsc} and results about distance functions to data sets $\D_n$ of Section \ref{sec:dataconv} are combined in Sections \ref{sec:diffconst}. In Subsection \ref{subsec:inertialess} and Subsection \ref{subsec:semilin} we introduce the data-driven problem both for inertialess fluids and fluids with inertia. We show that, given boundary conditions and a suitable pointwise coercivity condition, the functionals $I_n$ in \eqref{intro:I} are coercive on the phase space $V$. Therefore, we can apply results from Section \ref{sec:wlsc} to get the respective $\Gamma$-convergence result (Theorem \ref{thm:gammaconvlin} and Theorem \ref{thm:gammaconvsemilin}). 

Finally, Section \ref{sec:exact} links the (relaxed) data-driven problem $I^{\ast}(v)=0$ to the partial differential equations obtained by including a constitutive law in the modelling. We show that if the data set $\D$ coincides with the set obtained by a \emph{monotone} constitutive law, i.e. $\D = \{(\epsilon,\tigma) \colon \tigma= 2\mu(\vert \epsilon \vert)\epsilon\}$, then solutions to the relaxed data-driven problem are weak solutions to the classical PDE problem and vice versa.

\bigskip

\section{Functional Analytic Setting of the Fluid Mechanical Problem} \label{sec:prelim}

In this section we introduce an abstract functional analytic framework that offers a convenient way to reformulate the differential constraints. First, in Subsection \ref{subsec:notation}, we recall the notion of $\Gamma$-convergence and the notion of constant rank operators. The latter requires a short reminder on some results from Fourier analysis. In Subsection \ref{sec:diff:fluids} we show how the differential operators appearing in the fluid mechanical applications
fit into the framework of constant rank operators.


\subsection{$\Gamma$-Convergence and Constant Rank Operators}\label{subsec:notation}


\subsubsection{Underlying Function Spaces}
Let $\Omega \subset \R^d$ be a bounded, simply connected set with $C^1$-boundary and let
\begin{equation*}
    Y = \R^{d\times d}_{\sym,0}\coloneqq\left\{A \in \R^{d \times d} \colon A=A^T, \mathrm{tr} (A) = 0 \right\}
\end{equation*}
be the set of symmetric trace-free matrices in $\R^{d\times d}$.
We mainly study functions $v \colon \Omega \to Y \times Y$ and we shall write $v = (\epsilon,\tigma)$ to denote their components and $\sigma= -\pi \id + \tigma$ for a function $\pi:\Omega\to \R$.
One might think of $\epsilon$ as the strain and $\tigma$ the viscous part of the stress. For $1< p,q< \infty$ with $1/p+1/q=1$, we consider the \emph{phase space}
\begin{equation*}
    V 
    = 
    L_p(\Omega;Y) \times L_q(\Omega;Y), 
\end{equation*}
equipped with the norm 
\[
    \Vert v \Vert_{V} = \Vert \epsilon \Vert_{L_p} + \Vert \tigma \Vert_{L_q}.
\]
We call $Y \times Y $ the \emph{local phase space}.
Recall that we assume throughout the paper that the pressure $\pi$ (i.e. the trace of $\sigma$) is not considered as part of the data. Consequently, each data set $\Dcal_n$ is a subset of $Y\times Y$. In order to introduce a distance on $Y \times Y$, for pairs $(\epsilon_i,\tigma_i)\in Y \times Y$, $i=1,2$, we define
\begin{equation*}
    \dist\!\left((\epsilon_1,\tigma_1),(\epsilon_2,\tigma_2)\right)
    =
    \tfrac{1}{p} \abs{\epsilon_1-\epsilon_2}^p
    +
    \tfrac{1}{q} \abs{\tigma_1-\tigma_2}^q
\end{equation*}
and therewith
\begin{equation} \label{eq:def_metric}
    d\!\left((\epsilon_1,\tigma_1),(\epsilon_2,\tigma_2)\right)=\left( \dist\left((\epsilon_1,\tigma_1),(\epsilon_2,\tigma_2)\right)\right)^{\frac{1}{\max\{p,q\}}}.
\end{equation}
The  function $d(\cdot,\cdot)$ is defined by taking the $p$-th, respectively the $q$-th root of $\dist(\cdot,\cdot)$, in order to guarantee that the triangle inequality is satisfied . Thus,  $d(\cdot,\cdot)$ defines a metric on $Y\times Y$.

Accordingly, we define the distance on the phase space $V$ by 
\begin{equation*}
    \dist(v_1,v_2)=\int_\Omega \dist\left(v_1(x),v_2(x)\right)\dd x,
    \quad
    v_1,v_2\in V.
 \end{equation*}
 
We start by proving that the distance function $d(\cdot,\cdot)$, introduced in \eqref{eq:def_metric}, defines a metric.







\begin{lemma} The map $d \colon (Y \times Y) \times (Y \times Y) \to \R$ is a metric.
\end{lemma}


\begin{proof}
Positivity, definiteness and symmetry are clear.
The triangle inequality follows from the elementary inequality
\begin{equation}\label{eq:triangle1}
    \left( (a_1+a_2)^p + (b_1+b_2)^q \right)^\frac{1}{\max\{p,q\}} 
    \leq 
    \left(a_1^p+b_1^q\right)^\frac{1}{\max\{p,q\}}
    +
    \left(a_2^p+b_2^q\right)^\frac{1}{\max\{p,q\}},
\end{equation}
being valid for all $a_i,b_i \in [0,\infty)$, $i=1,2$, and $p \geq q$. 
Indeed, assume withput loss of generality that $p \geq q$. Then, since the functions $s \mapsto s^{q/p},s\mapsto s^{1/p}\, s\in \R$, are concave, we obtain
\begin{align*}
     \left[(a_1+a_2)^p + (b_1+b_2)^q\right]^{1/p} 
     & \leq  
     \left[(a_1+a_2)^p + \bigl(b_1^{q/p}+b_2^{q/p}\bigr)^p   \right]^{1/p} \\
     & \leq 
     \left[a_1^p+\bigl(b_1^{q/p}\bigr)^p\right]^{1/p}
     + 
     \left[a_2^p+\bigl(b_2^{q/p}\bigr)^p\right]^{1/p} \\
     &= 
     \bigl(a_1^p+b_1^q\bigr)^{1/p} +
     \bigl(a_2^p+b_2^q\bigr)^{1/p}.
\end{align*}
\end{proof}

In the following we embed $\Omega$ into the $d$-dimensional torus $\T_d$ when it is convenient. Without loss of generality we therefore assume that $\Omega$ is compactly contained in $(0,1)^d$. In general we use $C$ as a generic constant. However, we use specific constants whenever it is convenient. 


\subsubsection{$\Gamma$-convergence}

In this subsection we recall some well-known results on $\Gamma$-convergence that are frequently used throughout the paper. We use this notion of convergence to consider the behaviour of functionals of type \eqref{intro:I} and \eqref{intro:J} under convergence of the data.

\begin{definition} \label{def:Gammalim}
Let $(X,d)$ be a metric space. A sequence of functionals $I_n \colon X \to [-\infty,\infty]$, $\Gamma$-converges to $I \colon X \to [-\infty,\infty]$, in symbols $I= \Gamma-\lim_{n \to \infty} I_n$, whenever the following is satisfied: 
\begin{enumerate}[label=(\roman*)]
    \item \textbf{liminf-inequality:} For all $x \in X$ and for all sequences $x_n  \to x$ we have
    \[
    I(x) \leq \liminf_{n \to \infty} I_n(x_n).
    \]
    \item  \textbf{limsup-inequality:}  For all $x \in X$ there exists a sequence $x_n \to x$ (called the \textit{recovery sequence}) such that \[
    I(x) \geq \limsup_{n \to \infty} I_n(x_n).
    \]
\end{enumerate}
\end{definition}


\begin{remark}
\begin{enumerate}[label=(\roman*)]
    \item In metric spaces the constant sequence $I_n =I$ possesses a $\Gamma$-limit $I^{\ast}$, namely the lower-semicontinuous hull of $I$, given by 
    \begin{equation} \label{eq:Gamma-lim}
        I^{\ast}(x) = \inf_{x_n \to x } \liminf_{n \to \infty} I(x_n).
    \end{equation}
    $I^\ast$ is called the \textit{relaxation} of $I$.
    \item If each $x_n$ is a minimiser of $I_n$ and $x_n \to x$, then $x$ is a minimiser of $I$.
    \item One may define $\Gamma$-convergence on topological spaces, cf. \cite{DalMaso}. 
    This reproduces the definition on metric spaces when equipped with the standard topology.
    Weak convergence is not metrisable on Banach spaces. 
    However, it is metrisable on bounded sets of reflexive, separable Banach spaces.
    Hence, if a functional $I$ satisfies a certain growth condition; i.e.
    \begin{equation} \label{eq:babycoercive}
        \alpha(\Vert x \Vert) \leq I(x)
    \end{equation}
    for a function $\alpha \colon [0,\infty) \to \R$ with $\alpha(t) \to \infty$ as $t \to \infty$, we may use the metric for weak convergence defined on bounded sets of the Banach space and treat the Banach space together with the weak topology as a metric space.
    \item In topological spaces, especially in Banach spaces equipped with the weak topology, the constant sequence $I_n=I$ does in general not possess a sequential $\Gamma$-limit, as the infimum in \eqref{eq:Gamma-lim} does not need to be a minimum.
    \item If $I$ does not satisfy the growth condition \eqref{eq:babycoercive}, it is possible to consider the sequential $\Gamma$-limit, given as in Definition \ref{def:Gammalim}. However, this might not exist, even if the topological $\Gamma$-limit of a sequence of functionals exists. In particular, the constant sequence might not have a sequential $\Gamma$-limit.
\end{enumerate}
\end{remark}


In the following we only consider the sequential $\Gamma$-limit of sequences in the weak topology of some Banach space (usually $L_p \times L_q$). If the functional $I$ is coercive in the sense of \eqref{eq:babycoercive}, then the sequential $\Gamma$-limit coincides with the topological $\Gamma$-limit.

The following lemma links $\Gamma$-convergence to uniform convergence of functionals.


\begin{lemma}[Uniform convergence and $\Gamma$-convergence] \label{gammaconv:1}
Let $V$ be a reflexive, separable Banach space equipped with the weak topology. Suppose that $I_n, I \colon V \to [-\infty,\infty]$, such that $I_n \to I$ uniformly on bounded sets of $V$. If the sequential $\Gamma$-limit of the constant sequence $I$ exists, then also $I_n$ possesses a $\Gamma$-limit and 
\[
    \Gamma-\lim_{n \to \infty} I_n 
    = 
    \Gamma-\lim_{n \to \infty} I 
    = 
    I^{\ast}.
\]
\end{lemma}

\medskip

Note that the sequential $\Gamma$-limit of the constant sequence $I$ exists if the functional is coercive.

  
\begin{proof}
If $v_n \rightharpoonup v$ is a bounded sequence in $V$, we have 
\[
    \limsup_{m \to \infty} \sup_{n \in \N} \vert I_m(v_n) - I(v_n) \vert =0.
\]
Therefore, 
\[
    \limsup_{n \to \infty} I_n(v_n) 
    = 
    \limsup_{n \to \infty} I(v_n)
    \leq 
    I^\ast(v)
    \quad \text{and} \quad 
    \liminf_{n \to \infty} I_n(v_n) 
    = 
    \liminf_{n \to \infty} I(v_n)
    \geq 
    I^\ast(v),
\]
which establishes both the $\limsup$-inequality and the $\liminf$-inequality.
\end{proof}  

\subsubsection{Korn--Poincar\'e inequality}
 
In this subsection, we revisit a combination of Korn's inequality (i.e. the full gradient is controlled by its symmetric part) and Poincare's inequality to obtain an estimate of the form 
\[
    \Vert u \Vert_{W^1_p} 
    \leq 
    C \Vert \epsilon \Vert_{L_p}, 
    \quad \text{where} \quad
    1 < p< \infty
    \quad \text{and} \quad
    \epsilon = \tfrac{1}{2} \left(\nabla u + \nabla u^T\right).
\]
This estimate is a straightforward consequence of the $p$-Korn inequality and the Poincar\'e inequality, cf. for instance \cite{CiarletKorn}. For the convenience of the reader we provide the proof. In the following we use the notation
\begin{equation}
    \R^{d\times d}_{\skw}\coloneqq \{A\in \R^{d\times d}: A=-A^T\}.
\end{equation}
\begin{lemma}[Abstract Korn--Poincar\'e inequality] \label{kopo} 
Let $1<p<\infty$ and $\Omega \subset \R^d$ be open, connected, and bounded with $C^1$-boundary. Then the following is true.
\begin{enumerate} [label=(\roman*)]
    \item \label{abstract:1} There is a constant $C=C(p,\Omega)$, such that for any $u \in W^1_p(\Omega;\R^d)$ we have
    \[
        \Vert u - (A_u x +b_u) \Vert_{W^1_p} \leq C \Vert \nabla u + \nabla u^T \Vert_{L_p},
    \]
    where  $A_u = \tfrac{1}{2} \fint_{\Omega} \nabla u - \nabla u^T \dd x $ and $b_u = \fint_{\Omega} u \dd x$.
    \item \label{abstract:2} Let $X \subset W^1_p(\Omega;\R^d)$ be a closed subspace, such that 
    \[
        X \cap \left\{A x + b \colon A \in \R^{d \times d}_{\skw}, b \in \R^d\right\} = \{0\}.
    \]
    Then there is a constant $C=C(p,\Omega,X)$, such that for any $u \in X$ we have
    \[
        \Vert u \Vert_{W^{1}_p} 
        \leq 
        C \Vert \nabla u + \nabla u^T \Vert_{L_p}.
    \]

\end{enumerate}
\end{lemma}
\begin{proof}
\textbf{\ref{abstract:1}} Recall that there is a first-order  differential operator $\tilde{\A}$ with constant coefficients, such that 
\[
    \nabla^2 u = \tilde{\A} \left(\tfrac{1}{2} \left(\nabla u + \nabla u^T\right) \right).
\]
Therefore, we can bound 
\begin{align} \label{eq:hessian}
    \Vert \nabla^2 u \Vert_{W^{-1}_p} 
    \leq 
    C \Vert \nabla u + \nabla u^T \Vert_{L_p}.
\end{align}
Using Ne\v{c}as' lemma \cite{Nevcas,Czech} for functions with zero mean twice and writing $A'_u = \fint_{\Omega} \nabla u \dd x$, we get 
\begin{align}\label{eq:double_P}
    \Vert u - (A'_ux +b_u) \Vert_{W^1_p} \leq C \Vert \nabla^2 u \Vert_{W^{-1}_p}.
\end{align}
To obtain an inequality featuring only the skew-symmetric part $A_u = \tfrac 12\bra{A'_u - (A'_u)^T}$ note that by the triangle inequality 
\[
\Vert u - (A_u x + b_u) \Vert_{W^1_p} \leq \Vert u - (A'_u x + b_u) \Vert_{W^1_p} + \Vert (A_u-A'_u)x \Vert_{W^1_p}.
\]
The statement follows by estimating each term on the right-hand side by $C \Vert \nabla u + \nabla u^T \Vert_{L_p}$. For the first term we combine \eqref{eq:double_P} and \eqref{eq:hessian} to obtain
\[
\Vert u - (A'_u x + b_u) \Vert_{W^1_p} \leq  C \Vert \nabla u + \nabla u^T \Vert_{L_p}.
\]
Using Poincar\'e's and Jensen's inequalities, the second term can be estimated by
\begin{equation*}
    \Vert (A_u - A'_u)x \Vert_{W^1_p} 
    \leq 
    C\Vert A_u - A'_u \Vert_{L_p} 
    = 
    |\Omega|^{1/p} 
    \fint_{\Omega}\tfrac{1}{2}
    \left(\nabla u + \nabla u^T\right) \dd x 
    \leq 
    C\Vert\nabla u + \nabla u^T \Vert_{L_p}. 
\end{equation*}

\noindent\textbf{\ref{abstract:2}} Note that the space 
\[
    \tilde{X} = \left\{Ax +b \colon A \in \R^{d \times d}_{\skw}, b \in \R^d\right\}
\]
is finite-dimensional. As a consequence, if $\tilde{P} \colon W^{1,p}(\Omega;\R^d) \to \tilde{X}$ is a projection, then there is a constant $C(X)$, such that \begin{equation} \label{proj:finite}
\Vert u \Vert_{W^1_p} \leq C \Vert u - P u \Vert_{W^1_p}, \quad  u \in X.
\end{equation}
Indeed, if \eqref{proj:finite} were false, then there  would exist a sequence $u_n \subset X$ with $\Vert u_n \Vert_{W^1_p} =1$ and $\Vert u_n - P u_n \Vert_{W^1_p} \to 0$ as $n \to \infty$. As $P u_n \in \tilde{X}$ is bounded and $\tilde X$ is finite dimensional, there is a subsequence $P u_{n_j}$ converging strongly to some $y \in \tilde{X}$. Since $\Vert u_n - P u_n \Vert_{W^1_p} \to 0$, this implies $u_{n_j} \to y$ in $W^1_p(\Omega;\R^d)$. But this is a contradiction, as $X$ is closed, $\Vert u_{n_j} \Vert_{W^1_p} =1$ and $X \cap \tilde{X} = \{0\}$.
Part \ref{abstract:1} in combination with \eqref{proj:finite} yields \ref{abstract:2}, since $Pu=A_ux+b_u$.
\end{proof}


\subsubsection{Constant Rank Operators} \label{sec:Aqc}
In this subsection we introduce the version of constant rank operators used in this paper. To this end, we slightly adapt the notion of homogeneous constant rank operators \cite{Murat} since the differential operator $\A (\epsilon,\tigma)= (\curl \curl^T \epsilon, \diverg \tigma)$ appearing in the fluid mechanical application is only componentwise homogeneous.

We consider a differential operator $\A$ defined on functions $ v \colon \Omega\to \R^{m_1} \times \R^{m_2}$ defined via \[
\A(v_1,v_2) = (\A_1 v_1, \A_2 v_2)
\]
where $\A_1$ and $\A_2$ are homogeneous constant coefficient differential operators of order $k_i$, $i=1,2$, i.e. 
\begin{align}
    \A_i 
    \colon C^{\infty}(\Omega;\R^{m_i}) \to C^{\infty}(\Omega;\R^{l_i}),
    \quad
    \A_i v_i 
    = 
    \sum_{\vert \alpha \vert=k_i} A_{\alpha}^i \partial_{\alpha} v_i.
    \label{eq:A_i}
\end{align}
Recall that the Fourier symbols corresponding to the operators defined in \eqref{eq:A_i} are given by
\[
    \A_i[\xi] := \sum_{\vert \alpha \vert =k_i} A_{\alpha}^i \xi^{\alpha} \in \Lin (\R^{m_i};\R^{l_i}),
    \quad 
    i=1,2.
\]


\begin{definition}
$\A=(\A_1,\A_1)$ satisfies the \emph{constant rank property} if both $\A_1$ and $\A_2$ satisfy the constant rank property; that is, if
\begin{equation*}
    \dim\ker \A_i[\xi] 
    =
    r_i 
    \quad \text{for some fixed } 
    r_i \in \N 
    \text{ and for all } \xi \in \R^d \setminus\{0\}.
\end{equation*}
The \emph{characteristic cone} of $\A$ is defined as \begin{equation*}
    \Lambda_{\A}:=  \bigcup_{\xi \in \R^d \setminus \{0\}} \ker \A_1[\xi] \times \ker \A_2[\xi] \subset \R^{m_1} \times \R^{m_2}.
\end{equation*}
The operator $\A$ satisfies the \emph{spanning property} whenever 
\begin{equation*}
    \mathrm{span} \Lambda_{\A} =\R^{m_1} \times \R^{m_2}.
\end{equation*}
\end{definition}



\begin{remark}
If $u_i \in W^{k_i}_{p_i}(\T_d;\R^{m_i})$ can be written as 
\[
    u_i = \sum_{\xi \in \Z^d} \hat{u}_i(\xi) e^{-2 \pi i \xi \cdot x},
    \quad
   i=1, 2,
\]
then $u_i \in \ker \A_i$ if and only if for all $\xi \in \Z^d \setminus\{0\}$ we have $\hat{u}_i(\xi) \in \ker \A_i[\xi]$. If, in addition, the operators $\A_i$  satisfy the constant rank property, then $\Z^d \setminus \{0\}$ can be replaced by $\R^d \setminus \{0\}$.
\end{remark}




\subsubsection{Fourier Symbols and Fourier Multipliers}
In this subsection, we recall some important facts about constant rank differential operators that are connected to the Fourier transform on the $d$-torus $\T_d$. 
As we can consider the constraint operators $\A_1$ and $\A_2$ separately, we assume $\A' \colon C^{\infty}(\R^d;\R^m) \to C^{\infty}(\R^d;\R^l)$ to be a constant coefficient differential operator of order $k_{\A'}$, i.e. 
\begin{equation} \label{def:A}
    \A' v= \sum_{\vert \alpha \vert = k_{\A'}} A_{\alpha} \partial_{\alpha} v.
\end{equation}
Analogously, we consider $\B'\colon C^{\infty}(\R^d;\R^h) \to C^{\infty}(\R^d;\R^m)$ a constant coefficient differential operator of order $k_{\B'}$.
We call $\B'$ a \emph{potential of $\A'$}, whenever the corresponding Fourier symbols satisfy \begin{equation}\label{def:potential}
    \textup{Im} \B'[\xi] = \ker \A'[\xi],
    \quad 
    \text{ for all }\xi \in \R^d\setminus\{0\}.
\end{equation}
If $v \in C^{\infty}(\T_d;\R^m) \cap L_p(\T_d;\R^m)$, $1 \leq p < \infty$, we may write 
\[
    v(x) = \sum_{\xi \in \Z^d} \hat{v}(\xi) e^{-2\pi i \xi \cdot x}
    \quad 
    \text{and}
    \quad
    \hat{v}(\xi) := \int_{\T_d} v(x) e^{-2 \pi i \xi \cdot x } \dd x.
\]
For such $v$ and $\mathbb{W} \colon \R^d \setminus \{0\} \to \Lin (\R^m;\R^l)$, we may define a linear operator $W$ on $C^{\infty}(\T_d;\R^m) \cap L_p(\T_d;\R^m),\ 1 \leq p < \infty$,
by 
\[
    W (v) (x) = \sum_{\xi \in \Z^d} \mathbb{W}(\xi)(\hat{v}(\xi))  e^{-2\pi i \xi \cdot x},
\]
such that $W(v) \colon \T_d \to \R^l$. If $W$ maps boundedly into some function space, $W(v)$ can be defined for general $v \in L_p(\T_d;\R^m)$, $1 \leq p < \infty$, by using density. Such an operator $W$ is called \textit{Fourier multiplier}. The algebraic identity \eqref{def:potential} in combination with standard Fourier multiplier theory leads to the following statements. 

\begin{proposition}[\cite{Raita}] \label{proposition:potential}
Let $\A' \colon C^{\infty}(\R^d;\R^m) \to C^{\infty}(\R^d;\R^l)$ be a differential operator as in \eqref{def:A}. Then the following holds true:
\begin{enumerate}[label=(\roman*)]
    \item $\A'$ satisfies the constant rank property if and only if there exists a potential $\B' \colon C^{\infty}(\R^d;\R^h) \to C^{\infty}(\R^d;\R^m)$ of $\A'$.
    \item If $\B'$ is a potential of $\A'$, there exists a Fourier multiplier operator $\B'^{-1} \colon L_q(\T_d;\R^m) \to W^{k_{\B'}}_q(\T_d;\R^h)$ of order $-k_{\B'}$, such that for any $1<q< \infty $ we have
    \[
        \Vert \B' \circ \B'^{-1} v - v \Vert_{L_q} \leq C_q \Vert \A' v \Vert_{W^{k_{\A'}}_q} + \hat{v}(0),
    \]
    for some positive constant $C_q > 0$ that does only depend on $q$.
\end{enumerate}
\end{proposition}

For weakly, but not strongly, convergent sequences on bounded sets, there are essentially two possible effects. There can be oscillations and concentrations. For weak lower-semicontinuity results, oscillations are much easier to handle than concentrations. The notion of $p$-equi-integrability prevents concentration:


\begin{definition} A set $X \subset L_p(\T_d;\R^m)$ is called \emph{$p$-equi-integrable} if 
\[
    \lim_{\delta \to 0} \sup_{v \in X} \sup_{\vert E \vert <\delta} \int_E \vert v \vert^p \dd x =0.
\]
\end{definition}


\begin{lemma} \label{lemma:multiplier} 
Let $W \colon C^{\infty}(\T_d;\R^m) \to C^{\infty}(\T_d;\R^m)$ be a $0$-homogeneous Fourier multiplier. Then, for any $1<p<\infty$ the following holds true.
\begin{enumerate} [label=(\roman*)]
    \item \label{lemma:multiplier1} $W \colon L_p(\T_d;\R^m) \to L_p(\T_d;\R^m)$ is bounded;
    \item \label{lemma:multiplier2} $W$ is continuous from $L_p(\T_d;\R^m)$ to $L_p(\T_d;\R^m)$ with respect to the weak topology of $L_p(\T_d;\R^m)$;
    \item \label{lemma:multiplier3}If $X \subset L_p(\T_d;\R^m)$ is a $p$-equi-integrable and bounded set, then $W(X)$ is also $p$-equi-integrable.
\end{enumerate}
\end{lemma}


\begin{proof}
\textbf{\ref{lemma:multiplier1}} Part (i) follows from the Mikhlin--H\"ormander multiplier theorem (e.g.\cite{FM,Grafakos}). 

\noindent\textbf{\ref{lemma:multiplier2}} This follows from the fact that the adjoint operator $W^{\ast}$ is bounded from $L_{p'}(\T_d;\R^m)$ to $L_{p'}(\T_d;\R^m)$.

\noindent\textbf{\ref{lemma:multiplier3}} In order to verify the  $p$-equi-integrability of $W(X)$, we follow the the lines of the proof of \cite[Lemma 2.17]{FM}.

\noindent\textbf{Step 1: Construction of a truncated sequence. } There exists $R>0$ and for all $\varepsilon >0$ there exists a $\delta>0$, such that we have
\[
    \sup_{v \in X} \Vert v \Vert_{L_p}^p < R
    \quad 
    \text{and}
    \quad
    \sup_{v \in X} \sup_{\vert E \vert < \delta} \int_E \vert v \vert^p \dd x < \varepsilon.
\]
For $a > 0$ consider the function  $\tau_a \colon \R^m \to \R^m$, defined by
\[
    \tau_{a}(z) = 
    \begin{cases} 
    z, & \vert z \vert < a 
    \\
    0, & \vert z \vert \geq a.
\end{cases}
\]
Then, for fixed $a >0$ and  $u \in X$, the set $\{\tau_a\circ u:u\in X\}$ is bounded in $L_{\infty}(\T_d;\R^m)$.
Therefore, by \ref{lemma:multiplier1}, the set $\{W(\tau_a\circ u):u \in X\}$ is bounded in $L_r(\T_d;\R^m)$ for $r \geq p$. \\

\noindent \textbf{Step 2: $p$-equi-integrability of the truncated sequence. } We show that, for fixed $a \in \N$, the set $\{ W(\tau_a \circ u ) \}_{u \in X}$ is $p$-equi-integrable. \\
Taking Step 1 into account, this follows from the fact that any bounded set $X' \subset L_{2p}(\T_d;\R^m)$ is already $p$-equi-integrable. To prove this, assume for contradiction that there exists a bounded set $X'\subset L_{2p}(\T_d;\R^m)$ that is not $p$-equi-integrable. Then there exist $u_n \subset X'$, $E_n \subset \T_d$ with $\vert E_n \vert \to 0$, as $n \to \infty$, and an $\varepsilon >0$, such that 
\[
    \int_{E_n} \vert u_n \vert^p \dd x > \varepsilon,
    \quad
    n \in \N.
\]
By Jensen's inequality this implies 
\[
    \vert E_n \vert \int_{E_n} \vert u_n \vert^{2p} \dd x  > \varepsilon^2,
\]
which contradicts the assumption that $u_n$ is bounded in $L_{2p}(\T_d;\R^m)$ and that $\vert E_n \vert \to 0$. \\

We conclude that for any $\varepsilon>0$ there is $\delta_a(\varepsilon)>0$, such that, for all $u \in X$ we have the implication
\begin{equation} \label{eq:equiint}
    \vert E \vert < \delta_a(\varepsilon) 
    \quad \Longrightarrow \quad 
    \int_{E} \vert W (\tau_a\circ u) \vert^p \dd x < \varepsilon.
\end{equation}

\noindent \textbf{Step 3: $p$-equi-integrability of $W(X)$. } We show that Step 2 together with $p$-equi-integrability of $X$ implies that $W(X)$ is $p$-equi-integrable. \\
Using $p$-equi-integrability and boundedness, we may estimate
\begin{align}\label{eq:Lp_app}
    \lim_{a \to \infty} \sup_{u \in X} \Vert u - \tau_{a}\circ u \Vert_{L_p}^p  
    &\leq 
    \lim_{a \to \infty} \int_{\{ u \geq a \}} \vert u \vert^p \dd x
    \leq 
    \lim_{a \to \infty} \sup_{E \colon \vert E \vert < R a^{-p}} \int_E \vert u \vert^p
    \dd x
    =0.
\end{align} 
Therefore, we find that
\begin{align}\label{eq:W_tau}
    \lim_{a \to \infty} \sup_{u \in X} \Vert W (u-\tau_a\circ u) \Vert_{L_p} \leq C \lim_{a \to \infty} \sup_{u \in X} \Vert u - \tau_a\circ u \Vert_{L_p} =0.
\end{align}
Let now $\varepsilon>0$. By \eqref{eq:W_tau}, there exists $a(\varepsilon) \in \R_{+}$, such that 
\[
    \sup_{u\in X}  \Vert W (u-\tau_{a(\eps)}\circ u) \Vert_{L_p} < \tfrac{\varepsilon}{2}.
\]
In combination with \eqref{eq:equiint}, for all sets $E$ with measure smaller than $\delta_{a(\varepsilon)}(\varepsilon/2)$ , this yields
\[
    \int_{E} \vert W u \vert^p \dd x 
    \leq 
    \int_{E} \vert W (\tau_{a(\varepsilon)}\circ u )\vert^p \dd x 
    + 
    \int_{E \cap \{ \vert  u \vert \geq a\}} \vert W(u - \tau_{a(\varepsilon)}\circ u) \vert^p \dd x
    < 
    \tfrac{\varepsilon}{2} + \tfrac{\varepsilon}{2} = \varepsilon.
\]
Therefore, the set $W(X)$ is $p$-equi-integrable.
\end{proof}


\subsection{The differential operator $\A$ for problems in fluid mechanics} \label{sec:diff:fluids} 
In this section, we discuss how the fluid mechanical constraints \eqref{intro:linear_C} and \eqref{eq:nonlinear_C} fit into the previously outlined abstract setting. We consider the two differential operators 
\begin{equation*}
    \begin{cases}
        \A_1 =\curl \curl^T \colon C^{\infty}(\T_d;Y) \to C^{\infty}(\T_d;(\R^d)^{\otimes 4}) &
        \\
        \A_2 = \diverg \colon C^{\infty}(\T_d; Y)\times C^{\infty}(\T_d; \R) \to C^{\infty}(\T_d;\R^d) &
    \end{cases}
\end{equation*} 
as follows
\begin{equation*}
    \begin{cases}
        \left(\curl \curl^T(\epsilon)\right)_{ijkl} 
        = \partial_{ij} \epsilon_{kl} + \partial_{kl} \epsilon_{ij} - \partial_{il} \epsilon_{kj} - \partial_{kj} \epsilon_{il}, & 
        i,j,k,l=1,...,d 
        \\
        \left(\diverg(\tigma,\pi)\right)_i 
        = 
        (\diverg(\tigma - \pi \id))_i = \sum_{j=1}^d \partial_j (\tigma - \pi \id)_{ij}, 
        &
        i=1,...,d.
    \end{cases}
\end{equation*}
The Fourier symbol of the differential operator $\A_1$ is given by 
\[
  \left(\A_1[\xi](\epsilon)\right)_{ijkl} 
  = 
  \xi_{i} \xi_j \epsilon_{kl} +\xi_{k} \xi_l \epsilon_{ij} - \xi_i \xi_l \epsilon_{kj} -\xi_k \xi_j \epsilon_{il}, \quad 
  \xi \in \R^d\setminus \{0\},~\epsilon \in Y,~i,j,k,l=1,\dots,d.
\]
For $\A_2$ the Fourier symbol reads  
\[
    \left(\A_2[\xi](\tigma,\pi)\right)_i 
    = 
    \sum_{j=1}^d \xi_j \tigma_{ij} - \xi_i \pi, 
    \quad 
    \xi \in \R^d\setminus \{0\},~(\tigma,\pi) \in Y \times \R,~i=1,\dots,d.
\]
For a fixed $\xi \in \R^d \setminus \{0\}$, the set $\ker \A_1[\xi] \times \ker \A_2[\xi]$  is given as follows. Let $Y_{\xi} \subset Y$ be defined as 
\[
    Y_{\xi} = \left \{ a \odot \xi \colon  a \in \R^d,~ a \perp \xi \right\},
\]
where $a \odot \xi = \frac 12\left(a \otimes \xi + \xi \otimes a\right) $ is the symmetric tensor product. Note that $Y_{\xi}$ is a $(d-1)$-dimensional subspace of $Y$. Then 
\[
    \ker \A_1[\xi] = Y_{\xi},
\] 
meaning that the space dimension of $\ker \A_1[\xi]$ is $(d-1)$ and 
\[
    \ker \A_2[\xi] = \left\{ (\tigma,\pi_{\tigma}) \colon \tigma \in Y_{\xi}^{\perp} \right\},
\]
where $\pi_{\tigma}$ is defined as the unique $\pi \in \R$, such that $\A_2[\xi](\tigma,\pi) =0$, i.e. \[
    \pi_{\tigma} = \frac{\xi^T \tigma  \xi}{\vert \xi \vert^2}.
\]
The differential condition $\curl \curl^T \epsilon=0$ for $\epsilon \in L_p(\T_d;Y)$ with $\int_{T_d} \epsilon \dd x=0$ encodes that $\epsilon$ is a symmetric gradient, i.e. there is $u \in W^1_p(\T_d;\R^d)$ satisfying 
\[
    \Vert u \Vert_{W^1_p} 
    \leq 
    C \Vert \epsilon \Vert_{L_p},
    \quad 
    \epsilon = \tfrac 12\bra{\nabla u + \nabla u^T} 
    \quad 
    \text{and}
    \quad
    \diverg u=0.
\]
The differential operator 
\begin{equation*}
    \B_1 \colon C^{\infty}(\T_d;\R^d) \cap \ker \diverg \longrightarrow C^{\infty}(\T_d;Y) \colon u\longmapsto \tfrac 12\bra{\nabla u + \nabla u^T}
\end{equation*}
can be treated as if it was a potential of $\A_1$.

\begin{remark}
    Due to the additional constraint $\diverg u=0$, $\B_1$ is \emph{not} a potential to $\A_1$ in the sense of \eqref{def:potential}. In particular, Proposition \ref{proposition:potential} cannot be applied directly. Note, however that a function $u \in W^1_p(\T_d;\R^d)$ with zero average satisfies the differential constraint $\diverg u=0$ if and only if \[
    u = \curl^{\ast} U
    \]
    for a suitable function $U \in W^{2,p}\left(\T_d;\R^{d \times d}_{\mathrm{skew}}\right)$, where $\curl^{\ast}$ is the adjoint of $\curl$; in other words $\curl^{\ast}$ is a potential of $\diverg$. In particular, this also means that if $\epsilon = \tfrac 12\bra{\nabla u + \nabla u^T}$, then there exists $U \in W^2_p\left(\T_d;\R^{d \times d}_{\mathrm{skew}}\right)$ such that 
    \[
    \epsilon = \tfrac 12\left(\nabla+\nabla^T\right) \circ \curl^{\ast} U.
    \]
    Consequently, $\tilde \B_1 =  \tfrac 12\left(\nabla+\nabla^T\right) \circ \curl^{\ast}$ is a potential of $\A_1$.
    
    \medskip
    
    For the purpose of applying Fourier methods, we can use the symmetric gradient $\B_1$ on divergence-free matrices instead of the true potential. The suitable inverse of $\B_1$ in the Fourier space is \[
    \B_1^{-1} = \curl^{\ast} \circ \tilde \B_1,
    \]
    which is a Fourier multiplier of order $1+(-2)=-1$.
\end{remark}

The potential to the differential operator $\A_2$ is not relevant in this setting. Let us remark that the condition \[
    -\diverg \tigma + \nabla \pi = f,
\]
for $(\tigma,\pi) \in L_q(\T_d;Y \times \R)$ and $f \in W^{-1,p}(\T_d;\R^d)$, can be rewritten in terms of $\tigma$ only, as 
\[
    -\curl \circ \diverg \tigma = \curl f.
\]
Another strategy to tackle the linear problem from a ''purely`` Fourier analytic perspective would be to ''forget`` about the pressure $\pi$ by using the operator $\tilde{\A}_2(\tigma) = \curl \circ \diverg \tigma$. Note that in this approach the operator $\curl \circ \diverg$ acting on $\tigma$ is the adjoint operator of $\tfrac 12 \left(\nabla+\nabla^T\right) \circ \curl^{\ast}$ which acts on $U$. For the non-linear problem, cf. Subsection \ref{subsec:semilin}, this approach yields the equation 
\begin{equation}
    -\curl \diverg \tigma = \curl f - \curl( u \cdot \nabla) u.
\end{equation}
We believe however, that from the fluid dynamical point of view it is more instructive to include the pressure $\pi \in L_q(\Omega)$ by sticking to the more physical equation 
\[
    -\diverg \tigma= f - (u \cdot \nabla)u - \nabla \pi.
\]


\section{Existence of minimisers -- Weak Lower-Semicontinuity and Coercivity} \label{sec:wlsc}


It is the structure of the differential constraints, with constant rank operators of different order, the quasilinear perturbation of the otherwise linear constraints, the boundary conditions, and the natural location of $\epsilon$ and $\tigma$ in different spaces, that necessitates the following Section~\ref{sec:wlsc}, where all these challenges are adressed in an abstract setting.


\subsection{$\A$-Quasiconvexity}
In order to study weak lower-semicontinuity results, we first introduce the notion of $\A$-quasiconvexity for a constant rank operator $\A=(\A_1,\A_2)$ as defined in the previous section.


\begin{definition}
A (measurable and locally bounded) function $\funct \colon \R^{m_1} \times \R^{m_2} \to \R$ is called \emph{$\A$-quasiconvex} if for all $z=(z_1,z_2) \in \R^{m_1} \times \R^{m_2}$ and for all test functions $\psi=(\psi_1,\psi_2) \in \test$ with 
\begin{equation} \label{def:test}
    \test
    = 
    \left \{
    \psi \in C^{\infty}(\T_d;\R^{m_1}\times \R^{m_2}) \colon \A \psi =0 \text{ and } \int_{\T_d} \psi\dd x=0 \right\},
\end{equation}
it holds that
\begin{equation} \label{def:Aqc}
    \funct(z) \leq \int_{\T_d} \funct(z+\psi(x)) \dd x.
\end{equation}
For $\funct \in C(\R^{m_1}\times\R^{m_2})$ we define the \emph{$\A$-quasiconvex envelope $\QA \funct$ of $\funct$} as \begin{equation} \label{def:envelope}
    \QA \funct (z) = \inf_{\psi \in \test} \int_{\T_d} \funct(z+\psi(x)) \dd x.
\end{equation}
$\funct$ is called \emph{$\Lambda_{\A}$-convex} if for all $z \in \R^{m_1}\times \R^{m_2}$ and all $w \in \Lambda_{\A}$ the function 
\[
    t \longmapsto \funct(z+tw) 
\]
is convex.
\end{definition}


Note that the \emph{$\A$-quasiconvex envelope} $\QA \funct$ of a continuous function $\funct$ is the largest $\A$-quasiconvex function smaller than $\funct$ \cite{FM}. Moreover, a function $\funct$ is $\A$-quasiconvex if and only if $\funct= \QA \funct$.


\begin{proposition}[Properties of $\A$-quasiconvex functions] \label{prop:FM}
Let $\A= (\A_1,\A_2)$ be a differential operator satisfying the constant rank property and the spanning property and let $\funct: \R^{m_1} \times \R^{m_2}\to \R$. Then the following holds true. \begin{enumerate}[label=(\roman*)]
    \item \label{prop:FM1} If $\funct$ is locally bounded and $\A$-quasiconvex, then $\funct$ is continuous;
    \item \label{prop:FM2} if $\funct$ is continuous, then $\QA \funct$ is $\A$-quasiconvex and for all $z\in \R^{m_1} \times \R^{m_2}$ it holds that
    \[
        \QA \funct (z) 
        = 
        \sup \{\gunct(z)\colon \gunct \text{ is } \A\text{-quasiconvex and } \gunct \leq \funct\};
    \]
    \item \label{prop:FM3} if $\funct$ is continuous and $\A$-quasiconvex, then $\funct$ is $\Lambda_{\A}$-convex;
    \item \label{prop:FM4}  \label{loc:Lipschitz} if $\funct$ is $\A$-quasiconvex, $1< p,q< \infty$ and for all $z\in \R^{m_1} \times \R^{m_2}$ it holds that
    \[
        \funct(z_1,z_2) \leq C\left(1+\vert z_1 \vert^p + \vert z_2 \vert^q \right),
    \]
    then $\funct$ is locally Lipschitz continuous and \begin{align*}
        \left \vert \funct(z_1,z_2) - \funct(w_1,w_2) \right \vert 
        \leq\ & 
        C\left(1+ \vert z_1 \vert^{p-1} +\vert w_1 \vert^{p-1} +\vert z_2 \vert^{\alpha} + \vert w_2 \vert^{\alpha} \right) \cdot \left \vert z_1-w_1 \right \vert \\
        & +  C\left(1+ \vert z_1 \vert^{\beta} +\vert w_1 \vert^{\beta} +\vert z_2 \vert^{q-1} + \vert w_2 \vert^{q-1} \right) \cdot \left \vert z_2- w_2 \right \vert,
    \end{align*}
    where $\alpha = (p-1)q/p$ and $\beta = (q-1)p/q$.
\end{enumerate}
\end{proposition}


Statements \ref{prop:FM1}--\ref{prop:FM3} are slight adaptions of \cite[Section 3]{FM} for the case of first-order operators to the higher-order case. Statement \ref{prop:FM4} is a $(p,q)$-adaptation of \cite{JP,KK,RG}, where the $L_p$-setting is treated. The proof relies on the fact that any $\A$-quasiconvex function is $\Lambda_{\A}$-convex.

\subsection{Weak lower-semicontinuity under differential constraints.} 

Throughout this paragraph we consider $1<p,q<\infty$, a Carath\'{e}odory function $\funct \colon \Omega \times \R^{m_1} \times \R^{m_2} \to \R$ and functionals $I,J\colon L_p(\Omega;\R^{m_1}) \times L_q(\Omega;\R^{m_2}) \to \R$ defined by
\begin{equation} \label{def:G}
    J(v) = \int_{\Omega} \funct(x,v(x)) \dd x
    \quad \text{and}\quad 
    I(v) = 
    \begin{cases} 
    J(v), & \A v =0
    \\ 
    \infty, &\text{else},
\end{cases}
\end{equation}
The following proposition is a straight-forward adaption of the semi lower-continuity result \cite[Theorem 3.6]{FM} to the $(p,q)$-setting.


\begin{proposition} \label{proposition:wlsc}
Let $1<p,q<\infty$, let $\funct\colon \Omega \times \R^{m_1} \times \R^{m_2} \to \R$ be a  Carath\'{e}odory function, and assume that there exists $C>0$  such that the following growth condition is satisfied.
\begin{equation} \label{eq:pqgrowth}
    0 \leq \funct(x,z_1,z_2) \leq C (1+ \vert z_1 \vert^p + \vert z_2 \vert^q),
    \quad \text{for almost all}\quad
    x\in\Omega \quad \text{and all}\quad (z_1, z_2) \in \R^{m_1} \times \R^{m_2}.
\end{equation}
Moreover, let $\funct(x,\cdot)$ be $\A$-quasiconvex for a.e. $x \in \Omega$, where $\A=(\A_1,\A_2)$ is a constant rank operator with $\A_i$ having rank $k_i$. Then the following holds true.
\begin{enumerate}[label=(\roman*)]
    \item Along all sequences $v_n \rightharpoonup v$ in $L_p(\Omega;\R^{m_1}) \times L_q(\Omega;\R^{m_2})$ with $\A v_n \to \A v$ strongly in $W^{-k_1}_p(\Omega;\R^{m_1}) \times W^{-k_2}_q(\Omega;\R^{m_2})$ the functional $J$ is sequentially weakly lower-semicontinuous, i.e.
    \[
        J(v) \leq \liminf_{n \to \infty} J(v_n);
    \] 
    \item the functional $I$ is sequentially weakly lower-semicontinuous on $L_p(\Omega;\R^{m_1}) \times L_q(\Omega;\R^{m_2})$.
\end{enumerate}
\end{proposition}


We do not provide the proof of Proposition~\ref{proposition:wlsc} here, since it is largely analogous to the proof of \cite[Theorem 3.6]{FM}, which is based on a suitable notion of equi-integrable sequences. In the $(p,q)$-setting, the right notion of equi-integrability is the following.


\begin{definition}
A set $X \subset L_p(\Omega;\R^{m_1}) \times L_q(\Omega;\R^{m_2})$ is called \emph{$(p,q)$-\textit{equi-integrable}}, if for all $\varepsilon >0$ there exists a $\delta >0$, such that \[
    E \text{ measureable },~\vert E \vert < \delta \quad \Longrightarrow \quad  \sup_{v \in X} \int_E \vert v_1 \vert^p + \vert v_2 \vert^q \dd x < \varepsilon,
\]
that is $\{ v_1 \}_{v \in X}$ and $\{ v_2 \}_{v \in X}$ are $p$-equi-integrable and $q$-equi-integrable, respectively.
\end{definition}


The key insight for Proposition \ref{proposition:wlsc} is that it suffices to consider $(p,q)$-equi-integrable sequences. 
This is the content of the following proposition which is again a straightforward adaption of the $p$-setting.

\begin{proposition} \label{proposition:equiintegrablewlsc}
Let $1<p,q < \infty$ and let $\funct \colon \Omega\times \R^{m_1} \times \R^{m_2} \to \R$ be a  Carath\'{e}odory function satisfying the growth condition \eqref{eq:pqgrowth}. Let $v_n \rightharpoonup v$ in $L_p(\Omega;\R^{m_1}) \times L_q(\Omega;\R^{m_2})$ and suppose that there is a $(p,q)$-equi-integrable sequence $w_n \subset L_p(\Omega;\R^{m_1}) \times L_q(\Omega;\R^{m_2})$ such that for some $\theta$ with $\max\left(1/p,1/q\right)<\theta<1$ it holds that
\[
    \Vert v_n -w_n \Vert_{L_{\theta p} \times L_{\theta q}}
    \longrightarrow 
    0.
\]
Then we have
\begin{equation*}
    \liminf_{n \to \infty} J(w_n) \leq \liminf_{n \to \infty} J(v_n).
\end{equation*}
\end{proposition}


The proof of Proposition \ref{proposition:equiintegrablewlsc} is contained in the proof of the following theorem.


\begin{theorem} \label{theorem:equiint}
Let $1<p,q < \infty$ and let $X \subset L_p(\Omega;\R^{m_1}) \times L_q(\Omega;\R^{m_2})$ be weakly closed. Moreover, let $\funct, \funct_n \colon \Omega \times \R^{m_1} \times \R^{m_2}$ be Carath\'eodory functions. We define the functionals $I^X_n, I^X \colon X \to \R$ as 
\[
    I_n^X(v)
    = 
    \begin{cases} 
    \int_{\Omega} \funct_n(x,v) \dd x, & v \in X 
    \\     
    \infty, & \text{else,} 
    \end{cases}
    \quad\text{and}\quad 
    I^X (v) 
    =
    \begin{cases}
    \int_{\Omega} 
    \funct(v) \dd x, & v \in X 
    \\ 
    \infty, & \text{else,}  
\end{cases}
\]
Suppose that $X$ satisfies the following condition: \begin{enumerate}[label=(H\arabic*)]
    \item \label{theorem:equiint:hypo1}
        For all bounded sequences $v_n \subset X$ there exists a $(p,q)$-equi-integrable sequence $w_n \subset X$, such that $w_n-v_n \to 0$ in measure.
\end{enumerate}
Suppose further that $\funct_n,\funct$ satisfy: \begin{enumerate}[resume,label=(H\arabic*)]
    \item \label{theorem:equiint:hypo2} there exists a constant $C>0$, such that for all $(z_1,z_2) \in \R^{m_1} \times \R^{m_2}$ and almost every $x \in \Omega$ we have \[
    0 \leq \funct_n(x,z_1,z_2),\funct(x,z_1,z_2) \leq C(1 + \vert z_1 \vert^p + \vert z_2 \vert^q);
    \]
    \item \label{theorem:equiint:hypo3} $\funct$ and $\funct_n$ are uniformly continuous on bounded sets of $ \R^{m_1} \times \R^{m_2}$, i.e. there exists a monotone function $\nu_R:[0,\infty)\to \R$ with $\nu_R(s)\to 0$ as $s\to 0$, such that for all $n\in \N$, all $z_1,z_2 \in \R^{m_1}\times \R^{m_2}$ with $\vert z_1 \vert, \vert z_2 \vert \leq R$, and for almost every $x \in \Omega$:
    \[
        \vert \funct_n(x,z_1)- \funct_n(x,z_2) \vert + \vert \funct(x,z_1) - \funct(x,z_2) \vert < \nu_R(\vert z_1-z_2 \vert);
    \]
    \item \label{theorem:equiint:hypo4} the functionals with integrands $\funct_n$ converge uniformly on equi-integrable subsets, i.e. for all equi-integrable sets $B \subset L_p(\Omega;\R^{m_1}) \times L_q(\Omega;\R^{m_1})$ and for all $\varepsilon>0$ there exists $n_{\varepsilon} \in \N$, such that for all $v \in B$ and all $n \geq n_{\varepsilon}$ it holds
    \[
        \left \vert \int_{\Omega} \funct_n(x,v(x)) - \funct(x,v(x)) \dd x \right \vert \leq \varepsilon.
    \]
\end{enumerate}
Then the functionals $I^X_n$ and $I^X$ enjoy the following properties:
\begin{enumerate}[label=(\roman*)]
    \item \label{theorem:equiint:1} for all sequences $v_n \rightharpoonup v$ in $X$, there is a sequence $w_n \rightharpoonup v$ in $X$ such that 
    \begin{equation*}
        \limsup_{n \to \infty} I^X_n(w_n) \leq \liminf_{n \to \infty} I^X(v_n);
    \end{equation*}
    \item \label{theorem:equiint:2} for all sequences $v_n \rightharpoonup v$ in $X$, there is a sequence $\bar{w}_n \rightharpoonup v$ in $X$ such that
    \begin{equation*}
        \limsup_{n \to \infty} I^X(\bar{w}_n) \leq\liminf_{n \to \infty} I^X_n(v_n);
    \end{equation*}
    \item \label{theorem:equiint:3} if the sequential $\Gamma$-limit of the constant sequence $I^X$ exists, then the sequential $\Gamma$-limit of $I^X_n$ exists and 
    \[
        \Gamma - \lim_{n \to \infty} I_n^X 
        = 
        \Gamma - \lim_{n \to \infty} I^X.
    \]
\end{enumerate}
\end{theorem}


Note that the constraint set $\Ccal$ in the fluid mechanical application is weakly closed and may thus play the role of the set $X$.


\begin{proof}
\noindent\textbf{\ref{theorem:equiint:1}} 
The main idea of the proof is to show that a suitable version of Proposition \ref{proposition:equiintegrablewlsc} holds, namely that sequences $w_n \subset X$ as in \ref{theorem:equiint:hypo1} already satisfy \ref{theorem:equiint:1}. To this end, let $v_n \subset X$ be bounded, and let $w_n \subset X$ be a $(p,q)$-equi-integrable sequence, such that $w_n - v_n \to 0$ in measure. Then we have
\begin{align*}
    \limsup_{n \to \infty}  I^X_n(w_n)- I^X(v_n) 
    &= 
    \int_{\Omega} \funct_n (x,w_n) -\funct (x,v_n) \dd x \\ 
    &\leq 
    \limsup_{n \to \infty} \int_{\Omega} \funct_n(x,w_n) - \funct(x,w_n) \dd x
    + 
    \limsup_{n \to \infty} \int_{\Omega}  \funct(x,w_n)-\funct(x,v_n) \dd x.
\end{align*}
Due to \ref{theorem:equiint:hypo4} and the $(p,q)$-equi-integrablility of $w_n$ the first term tends to $0$. In order to estimate the second term, let $L > 0$ be a constant such that $\|v_n\|_{L_p}, \|w_n\|_{L_p} \leq L$. Then, using \ref{theorem:equiint:hypo2}, for any $R>0$ we obtain
\begin{align*}
    &
    \int_{\Omega} \funct(x,w_n)-\funct(x,v_n) \dd x 
    \\
    &= 
    \int_{\{\vert w_n\vert, \vert v_n\vert \leq R\}}  \funct(x,w_n)-\funct(x,v_n) \dd x
    + 
    \int_{\{\vert w_n \vert 
    \geq R\} \cup \{\vert v_n \vert \geq R\}} \funct(x,w_n)-  \funct(x,v_n) \dd x 
    \\
    &\leq 
    \int_{\{\vert w_n\vert, \vert v_n \vert \leq R\}}  \nu_R \bigl( \vert w_n - v_n \vert \bigr) \dd x 
    + 
    \sup_{E \colon  \vert E \vert < 2 (L/R)^{\min(p,q)}} \int_{E}  C(1+ \vert w_{n,1} \vert^p + \vert w_{n,2} \vert^q) \dd x. 
\end{align*}
The first integral on the right-hand side of this inequality converges to $0$ as $n \to \infty$, since $w_n - v_n \to 0$ in measure by \ref{theorem:equiint:hypo1}. 
Moreover, since the sequence $w_n$ is $(p,q)$-equi-integrable, the second integral can be bounded by a constant $c_R$ with $c_R \to 0$ as $R \to \infty$. Consequently,
\[
    \limsup_{n \to \infty} \int \funct(x,w_n) - \funct(x,v_n) \dd x \leq 0
\]
and we conclude that
\begin{equation}
    \limsup_{n \to \infty} I^X_n(w_n) \leq \liminf_{n \to \infty}  I^X(v_n).
\end{equation}

\noindent\textbf{\ref{theorem:equiint:2}} The second statement  is obtained in the same way by swapping the roles of $\funct_n$ and $\funct$. Note that we can uniformly estimate
\[
    \int_{\{\vert w_n\vert, \vert v_n \vert \leq R\}}  \funct_n(x,w_n)-\funct_n(x,v_n) \dd x,
\]
as all $\funct_n$ have the same modulus of continuity on bounded sets, cf. \ref{theorem:equiint:hypo3}.
\medskip

\noindent\textbf{\ref{theorem:equiint:3}} If the sequential $\Gamma$-limit of $I^X$ exists (we denote it by $I^{X\ast}$), then for all $v \in X$ the following holds true. 
\begin{enumerate}
    \item[(a)] Every sequence $v_n \subset X$ with $v_n \rightharpoonup v$ in $X$ satisfies $I^{X\ast}(v) \leq \liminf_{n \to \infty} I^X (v_n)$.
    \item[(b)] There exists a sequence $v_n \subset X$ with $v_n \rightharpoonup v$ in $X$, such that $I^{X\ast}(v) \geq \limsup_{n \to \infty} I^X(v_n)$.
\end{enumerate}
The $\liminf$-inequality for $I_n^X$ is ensured by \ref{theorem:equiint:2}, i.e. if $v_n \rightharpoonup v$ in $X$, then 
\[
    \liminf_{n \to \infty} I^X_n(v_n) \geq \limsup_{n \to \infty} I^X(\bar w_n)\ge  \liminf_{n \to \infty} I^X(\bar w_n) \geq I^{X\ast}(v),
\]
as $\bar w_n \rightharpoonup v$ in $X$. On the other hand, the $\limsup$-inequality follows from \ref{theorem:equiint:1}: the recovery sequence $v_n$ (or at least a suitable subsequence) can be modified to an equi-integrable recovery sequence $w_n$. By \ref{theorem:equiint:1}, we find that
\[
    I^{X\ast}(v) \geq \limsup_{n \to \infty} I^X(v_n) \geq \liminf_{n \to \infty} I^X(v_n) \geq \limsup_{n \to \infty} I_n^X(w_n). 
\]
This completes the proof.
\end{proof}

The main challenge in applying Theorem \ref{theorem:equiint} to the case in which $X$ is a set given by differential constraints and boundary conditions is to verify Hypothesis \ref{theorem:equiint:hypo1}. In Section \ref{sec:dataconv} we check the conditions  \ref{theorem:equiint:hypo2}--\ref{theorem:equiint:hypo4} on the integrand $\funct$. To verify \ref{theorem:equiint:hypo1}, for a given sequence $v_n$ we need to construct a suitable $(p,q)$-equi-integrable modification $w_n$ that conserves both the differential constraints \emph{and} the boundary conditions. For this purpose we need the following two auxiliary results.


\begin{lemma}\label{lemma:diagonal}
Let $(X,d_X)$ be a complete metric space. Suppose that $x_n$ is a sequence in $X$, such that $x_n \to x$ and that, for $m \in \N$, we have ${x_{n,m}}$ with 
\[
    \lim_{m \to \infty} \sup_{n \in \N} d_X(x_{n,m},x_n) =0
    \quad 
    \text{and}
    \quad
    \lim_{n \to \infty} d_X(x_{n,m},x) =0 \quad \text{for all } m \in \N.
\]
Then $x_{n,m} \to x$ uniformly in $m$, as $n \to \infty$.
\end{lemma}


\begin{proof}
 Let $\varepsilon >0$. Then there exists $m_{\varepsilon} \in \N$, such that for all $m \geq m_{\varepsilon}$ 
\[
    d_X(x_{n,m},x_n) < \tfrac{\varepsilon}{2}
\]
and an $N_{\varepsilon}$, such that for all $n>N_{\varepsilon}$ we find that
\[
    d_X(x_n,x) < \tfrac{\varepsilon}{2}.
\]
Moreover, there are $N^1,...,N^{m_{\varepsilon}}$, such that for all $m=1,...,m_\varepsilon$ it holds
\[
    n > N^m\quad  \Longrightarrow \quad d_X(x_{n,m},x) < \varepsilon.
\]
Choosing $N = \max\{N_{\varepsilon}, N^1,...,N^{m_{\varepsilon}}\}$ yields that for any $n >N$ and $m \in \N$ we have
\[
    d(x_{n,m},x) < \varepsilon,
\]
which is the required uniform convergence.
\end{proof}


The following result is due to \cite[Lemma 2.15]{FM}. It allows to construct $(p,q)$-equi-integrable modified sequences. However, in general these modified sequences fail to conserve the constraints.


\begin{proposition} \label{proposition:equiint}
Let $v_n$ be a bounded sequence in $L_p(\Omega;\R^m)$. Then there exists a $p$-equi-integrable sequence $\tilde{v}_n$ with the following properties:
\begin{enumerate}[label=(\roman*)]
    \item for almost every $x \in \Omega$ we have $\vert \tilde{v}_n(x) \vert \leq \vert v_n(x) \vert$;
    \item for every $q<p$ we have $\lim_{n \to \infty} \Vert v_n - \tilde{v}_n \Vert_{L_q} =0$.
\end{enumerate}
\end{proposition}


The following theorem allows to obtain modified sequences that continue to satisfy both differential constraints and boundary conditions.

\begin{theorem}[Equi-integrable sequences \& boundary values]  \label{lemma:equiintboundary}
 Suppose that $\A \colon C^{\infty}(\R^d;\R^m) \to C^{\infty}(\R^d;\R^l)$ is a homogeneous differential operator of order $k_{\A}$, satisfying the constant rank property and that $\B$ is a potential of $\A$ in the sense of \eqref{def:potential}. Let $\Omega \subset \R^d$ be an open and bounded set with Lipschitz boundary. Let $v_n \rightharpoonup 0$ in $L_p(\Omega;\R^m)$ and $\A v_n \to 0$ in $W^{-k_{\A}}_p(\Omega;\R^l)$. Then there exists a sequence $w_n \subset W^{k_{\B}}_p(\Omega;\R^h)$ such that the following holds true:
 \begin{enumerate}[label=(\roman*)]
    \item the sequence $\sum_{j=0}^{k_{\B}} \vert \nabla^j w_n \vert$ is $p$-equi-integrable;
    \item $\Vert \B w_n - v_n \Vert_{L_q} \to 0$, as $n \to \infty$ for any $q<p$;
    \item $w_n$ is compactly supported in $\Omega$.
    \end{enumerate}
\end{theorem}
The main difficulty in the proof compared to the statement without boundary values in \cite{FM} is to obtain the compact support. 
\begin{proof}
\textbf{Step 1: Construction of the sequence.} 
 We assume by scaling that $\Omega \subset \subset (0,1)^d$, i.e. it may be viewed as a subset of  the $d$-dimensional torus $\T_d$. We extend $v_n$ by $0$ outside $\Omega$. Let $m \in \N$. We define open sets $V_m$ and $U_m$, such that $V_m \subset \subset U_m \subset \subset \Omega$; in particular \begin{align*}
    \{ x \in \Omega \colon \dist(x,\partial \Omega) >2/m\} &\subset  V_m \subset   \{ x \in \Omega \colon \dist(x,\partial \Omega) >1/m\}, \\
     \{ x \in \Omega \colon \dist(x,\partial \Omega) >4/m\} &\subset U_m \subset   \{ x \in \Omega \colon \dist(x,\partial \Omega) >3/m\}.
\end{align*}
Then there exist $\varphi_m \in C_c^{\infty}(V_m)$ with $\varphi_m \equiv 1$ on  $U_m$ and $\psi_m \in C_c^{\infty}(\Omega)$ with $\psi_m \equiv 1$ on $V_m$, such that for all $k,m \in \mathbb{N}$ 
\[
    \Vert \nabla^k \psi_m \Vert_{L_{\infty}}, \Vert \nabla^k \varphi_m \Vert_{L_{\infty}} \leq C(k) m^k.
\]
By Proposition \ref{proposition:equiint} there exists a $p$-equi-integrable sequence $\tilde{v}_n$, such that $\Vert \tilde{v}_n - v_n \Vert_{L_q} \to 0$ for $q<p$. Therefore, as  $v_n$ converges weakly to $0$, so does $\tilde{v}_n$.
We define 
\begin{equation*}
    \bar{v}_{n,m} = \varphi_m \tilde{v}_n, 
    \quad
    \bar{w}_{n,m} = \B^{-1} \bar{v}_{n,m}
    \quad \text{and} \quad 
    w_{n,m} = \psi_m \bar{w}_{n,m}.
\end{equation*}
We claim that we can take an appropriate diagonal sequence $w_{n,m(n)}$ with $m(n) \to \infty$, as $n \to \infty$, such that $w_{n,m(n)}$ satisfies the requirements of Theorem \ref{lemma:equiintboundary}. The purpose of the following steps is to construct such a sequence $m(n)$.
\medskip

\noindent\textbf{Step 2: Estimates on $\bar{v}_{n,m}$. }
First, we show that 
\begin{equation} \label{Step2:1}
    \lim_{m \to \infty} \sup_{n \in \N} \Vert \tilde{v}_n - \bar{v}_{n,m} \Vert_{L_p} =0.
\end{equation}
To this end, note that there is a constant $C>0$, such that \begin{equation} \label{Step2:2}
    \vert \Omega \setminus V_m \vert  \leq \vert \Omega \setminus U_m \vert \leq \tfrac{C}{m}
\end{equation}
since $\Omega$ has Lipschitz boundary.
Then we deduce that
\begin{align*}
    \sup_{n \in \N} \Vert \tilde{v}_n - \bar{v}_{n,m} \Vert_{L_p} &\leq  \sup_{n \in \N}\Vert \tilde{v}_n \Vert_{L_p(\Omega \setminus U_m)} \\
    &\leq  \sup_{n \in \N} \sup_{\vert E \vert \leq \vert (\Omega \setminus U_m)\vert} \Vert \tilde{v}_n \Vert_{L_p(E)} \\
    & \leq  \sup_{n \in \N} \sup_{\vert E \vert \leq Cm^{-1}} \Vert \tilde{v}_n \Vert_{L_p(E)}.
\end{align*}
As $\tilde{v}_n$ is $p$-equi-integrable, the right-hand side converges to $0$, as $m \to \infty$. Thus, \eqref{Step2:1} is established.

Second, we bound the $W^{-k_{\A}}_q$-norm of $\A \bar{v}_{n,m}$. We claim that there exists a sequence $M_1(n)$ with $M_1(n) \to \infty$, as $n \to \infty$, such that for all $m(n)$ with $m(n) \leq M_1(n)$ and $m(n) \to \infty$, as $n \to \infty$, there exists $1<q<p$ such that
\begin{equation} \label{Step2:3}
    \lim_{n \to \infty} \Vert \A \bar{v}_{n,m(n)} \Vert_{W^{-k_{\A}}_q(\T_d;\R^l)}=0.
\end{equation}
Note that if $\tilde{v}_n$ is in $C^k(\Omega;\R^m)$, then we may write 
\[
    \A{\bar{v}_{n,m}} = \A (\varphi_m \tilde{v}_n) = (\A \tilde{v}_n) \varphi_m + \sum_{\vert \alpha \vert =k_{\A}} \sum_{\beta<\alpha} \binom{\alpha}{\beta} A_{\alpha} \partial_{\beta} \tilde{v}_n \partial_{\alpha -\beta} \varphi_m.
\]
Therefore, by applying the definition of $W^{-k_{\A}}_q(\T_d;\R^l)$, we may estimate
\begin{equation} \label{Step2:estimate:A}
    \Vert \A \bar{v}_{n,m} \Vert_{W^{-k_{\A}}_q(\T_d;\R^l)} \leq \Vert \A \tilde{v}_n \Vert_{W^{-k_{\A}}_q(\Omega;\R^l)} \Vert \varphi_m \Vert_{W^{k_{A}}_\infty(\Omega)} + C \Vert \tilde{v}_n \Vert_{W^{-1,q}(\Omega)} \Vert \varphi_m \Vert_{W^{k_{\A}+1}_\infty(\Omega)}.
\end{equation}
Due to density of $C^k(\Omega;\R^m)$ in $L_p(\Omega;\R^m)$, inequality \eqref{Step2:estimate:A} is still valid even if $\tilde{v}_n$ is merely in $L_p(\Omega;\R^m)$. With the estimates for the derivatives of $\varphi$ we get 
\[
    \Vert \A \bar{v}_{n,m} \Vert_{W^{-k_{\A}}_q(\T_d;\R^l)} \leq 
    C\left( m^{k_{\A}} \Vert \A \tilde{v}_n \Vert_{W^{-k_{\A}}_q(\Omega;\R^l)} 
    + 
    m^{k_{\A}+1} \Vert \tilde{v}_n \Vert_{W^{-1}_q(\Omega;\R^l)}\right).
\]
Note that, on the one hand, $\A \tilde{v}_n \to 0$ in $W^{-{k_{\A}}}_q(\Omega;\R^l)$, as $\A v_n \to 0$ in $W^{-{k_{\A}}}_p(\Omega;\R^l)$ and $\tilde{v}_n - v_n \to 0$ in $L_q(\Omega;\R^m)$ for $q<p$. On the other hand, as $\tilde{v}_n$ is bounded in $L_p(\Omega;\R^m)$ and weakly converging to $0$, $\tilde{v}_n \to 0$ in $W^{-1}_q(\Omega;\R^m)$ strongly, due to the compact embedding of $L_q(\Omega;\R^m)$ into $W^{-1}_q(\Omega;\R^m)$. Therefore, choosing \begin{equation} \label{Step2:defM1}
    M_1(n) := \left( \min \left\{\{ \Vert \A \tilde{v}_n \Vert_{W^{-k}_q},~\Vert \tilde{v}_n \Vert_{W^{-1}_q} \right\}\right)^{\frac{-1}{3 k_{\A}}} 
    \longrightarrow \infty, 
    \quad \text{as } n \to \infty,
\end{equation}
we get 
\begin{equation} \label{Step2:4}
    \lim_{n \to \infty} \sup_{m \leq M_1(n)} \Vert \A \bar{v}_{n,m} \Vert_{W^{-k_{\A}}_p(\T_d;\R^l)} =0.
\end{equation}
Last, let us note that due to equi-integrability of $\tilde{v}_n$ and $\vert \bar{v}_{n,m} \vert \leq \vert \tilde{v}_n \vert$, also the set $\{ \bar{v}_{n,m} \}_{n,m \in \N}$ is equi-integrable.
\medskip

\textbf{Step 3: Upper Bound on $\Vert \B w_{n,m} -v_n \Vert_{L_q}$. }
First, we note that, by definition, $w_{n,m}$ is compactly supported in $\Omega$ for any $m \in \N$, as $\psi_m$ is compactly supported in $\Omega$. Moreover, it holds \begin{align*}
    \Vert \B w_{n,m} -v_n \Vert_{L_q} &\leq \Vert \B w_{n,m} - \B \bar{w}_{n,m} \Vert_{L_q} + \Vert \B \bar{w}_{n,m} - \bar{v}_{n,m} \Vert_{L_q} + \Vert \bar{v}_{n,m} - \tilde{v}_{n} \Vert_{L_q} + \Vert \tilde{v}_n -v_n \Vert_{L_q} 
    \\ 
    &\eqqcolon 
    \mathrm{(I)} + \mathrm{(II)} +  \mathrm{(III)} +  \mathrm{(IV)}.
\end{align*}
We already established by the choice of $\tilde{v}_n$ (c.f. Proposition \ref{proposition:equiint}), that $\mathrm{(IV)} \to 0$, as $n \to \infty$. Furthermore, $\mathrm{(III)} \to 0$, as $n \to \infty$, whenever $m=m(n) \to \infty$, cf. \eqref{Step2:1}. Proposition \ref{proposition:potential} yields
\[
    \mathrm{(II)} 
    \leq 
    \vert \A \bar{v}_{n,m(n)} \vert + \int_{\T_d} \bar{v}_{n,m(n)} \dd x.
\]
 The first term tends to $0$ by \eqref{Step2:4}, whenever $m(n) \leq M_1(n)$ is a sequence diverging to $\infty$ as $n \to \infty$, while the mean of $\tilde{v}_{n,m(n)}$ converges to zero since $\tilde{v}_n\rightharpoonup 0$ and because of \eqref{Step2:1}. It remains to bound $\mathrm{(I)}.$ To this end, note that triangle inequality and then H\"older's inequality imply
 \begin{align*}
     \mathrm{(I)} &\le \Vert (1-\psi_m) \B \bar{w}_{n,m} \Vert_{L_q} + \sum_{\vert \alpha \vert= k_{\B}} \sum_{\beta < \alpha } \Vert B_{\alpha} \partial_{\beta} \bar{w}_{n,m} \partial_{\alpha-\beta} \psi_m \Vert_{L_q}
     \\
     &\leq \Vert (1- \psi_m) \Vert_{L^{qp/(p-q)}} \Vert \B \bar{w}_{n,m} \Vert_{L_p} + m^{k_{\B}}\Vert \bar{w}_{n,m} \Vert_{W^{k_{\B}-1}_q} \\
     &\leq m^{(q-p)/(pq)}  \Vert \B \bar{w}_{n,m} \Vert_{L_p} + m^{k_{\B}}\Vert \bar{w}_{n,m} \Vert_{W^{k_{\B}-1}_q}.
\end{align*}
The first term vanishes (uniformly in $n\in \N$) as $m\to \infty$, due to the uniform $L_p$ bound on $\B \bar{w}_{n,m}$, as the operator $W=\nabla^{k_{\B}} \circ \B^{-1}$ is a $0$-homogeneous, smooth Fourier multiplier. Moreover, for the second summand note that due to Lemma \ref{lemma:multiplier} \ref{lemma:multiplier2} $W$ is continuous from $L_q(\T_d;\R^m)$ to $L_q(\T_d;\R^h\otimes (\R^{d})^{\otimes k_\B})$ in the weak topology. Recall, that $\tilde{v}_n \rightharpoonup 0$, as $n \to \infty$ in $L_p(\T_d;\R^m)$, that $\bar{v}_{n,m}$ is uniformly bounded in $L_p(\T_d;\R^m)$ and for fixed $m \in \N$, $\bar{v}_{n,m} = \varphi_m \tilde{v}_n \rightharpoonup 0$. The weak topology of $L_p(\T_d;\R^m)$ is metrisable on bounded sets, whence we may apply Lemma \ref{lemma:diagonal} to get that the convergence 
\[
    \bar{v}_{n,m} \rightharpoonup 0
    \quad \text{in } L_p(\T_d;\R^m), 
    \quad \text{as } n \to \infty,
\]
is \textit{uniform } in $m \in \N$. Again, by the boundedness of $W$, it holds that
\begin{equation} \label{Step3:1}
    W \bar{v}_{n,m} \rightharpoonup 0 
    \quad
    \text{in } L_p(\T_d;\R^h\otimes (\R^{d})^{\otimes k_\B}) \text{ uniformly in } m.
\end{equation}
For $s<p^{\ast} = dp/(d-p)$, the embedding $W^{k_{\B}}_p(\T_d;\R^h) \hookrightarrow W^{k_{\B}-1}_s(\T_d;\R^h)$ is compact. Hence, uniform weak convergence of $\nabla^{k_{\B}} \bar{w}_{n,m}$, together with Poincar\'{e}'s inequality, imply that 
\begin{equation} \label{Step3:2}
    \lim_{n \to \infty} \sup_{m \in \N} \Vert \bar{w}_{n,m} \Vert_{W^{k_{\B}-1}_s} =0.
\end{equation}
This holds in particular for $s=p<p^\ast$. Therefore, choosing $M_2(n)$ as 
\[
    M_2(n) := \left( \sup_{m \in \N} \Vert \bar{w}_{n,m} \Vert_{W^{k_{\B}-1,p}} \right)^{\frac{-1}{2 k_{\B}}}
\]
implies for any sequence $m(n)$ with $m(n) \leq \min\{M_1(n),M_2(n)\}$ and $m(n) \to \infty$, the inequality
\[
    \Vert \B w_{n,m(n)} -v_n \Vert_{L_q} \longrightarrow 0, \quad \text{as } n \to \infty.
\]
\medskip

\noindent\textbf{Step 4: Equi-integrability of $w_{n,m}$. }
It remains to show that we may choose the diagonal sequence $w_{n,m(n)}$ in such a fashion, that $\nabla^j w_{n,m(n)}$ is still $p$-equi-integrable for all $1 \leq j \leq k_{\B}$.
Note that 
\[
    \nabla^j w_{n,m} = \psi_m\nabla^j \bar{w}_{n,m} + \sum_{i=0}^{j-1} \nabla^i \bar{w}_{n,m} \otimes \nabla^{j-i} \psi_m.
\]
The sequence $\bar{w}_{n,m}$ is uniformly bounded in $m$ and $n$ in $W^{k_{\B}}_p(\T_d;\R^m)$, as $\bar{v}_{n,m}$ is uniformly bounded in $L_p(\T_d;\R^m)$ and $\B^{-1}$ maps $L_p(\T_d;\R^m)$ to $W^{k_{\B}}_p(\T_d;\R^h)$. Hence, for $j<k_{\B}$, $\nabla^j \bar{w}_{n,m}$ is bounded in $L_r(\T_d;\R^h\otimes (\R^{d})^{\otimes j})$ for some $r>p$ and thus $\vert  \psi_m\nabla^j \bar{w}_{n,m}  \vert \leq \vert  \nabla^j \bar{w}_{n,m}\vert$ is $p$-equi-integrable. Furthermore, observe that we have the pointwise estimate 
\[
    \abs{\nabla^i \bar{w}_{n,m} \otimes \nabla^{j-i} \psi_m} \leq 
    m^{k_{\B}} \abs{\nabla^i \bar{w}_{n,m}} 1_{\Omega \setminus V_m}.
\]
Hence, for $p$-equi-integrability it suffices to show that there is $M_3(n) \to \infty$, as $n \to \infty$, such that for $i < k_{\B}$ the sets
\begin{align} \label{Step4:1}
    &
    \left\{ \nabla^{k_{\B}} \bar{w}_{n,m} \colon m \leq M_3(n) \right\}  
    \\ \label{Step4:2}
    &
    \left\{m^{k_{\B}} \nabla^i \bar{w}_{n,m} 1_{\Omega \setminus U_m} \colon m \leq M_3(n) \right\}
\end{align}
are $p$-equi-integrable.
Indeed, \eqref{Step4:1} is clear, even for $m \in \N$, instead of only $m \leq M_3(n)$, using again that $W = \nabla^{k_{\B}} \circ \B^{-1}$ is a smooth $0$-homogeneous Fourier multiplier. On the other hand, $\nabla^{k_{\B}} \bar{w}_{n,m} = W( \tilde{v}_{n,m})$ and  $W( \tilde{v}_{n,m})$ is $p$-equi-integrable for $m,n \in \N$ by Step 1. 
In \eqref{Step3:2} we have already established the convergence 
\[
    \lim_{n \to \infty} \sup_{m \in \N} \Vert \bar{w}_{n,m} \Vert _{W^{k_{\B}-1}_s} =0 
\]
for all $s < p^{\ast}$. Let now $s\in (p,p^{\ast})$ be fixed.
Then for all measurable sets $E$ we find that
\begin{align*}
    \int_{E} \vert \nabla^i \bar{w}_{n,m} m^{k_{\B}} 1_{\Omega \setminus V_m} \vert^p \dd x  
    & \leq 
    m^{k_{\B}p}\int_{E \cap (\Omega \setminus V_m)} \vert \nabla^i \bar{w}_{n,m} \vert^p 
    \\
    & \leq 
    m^{k_{\B}p}\vert E \cap (\Omega \setminus V_m) \vert \left(\fint_{E \cap (\Omega \setminus V_m)} \vert  \nabla^i \bar{w}_{n,m} \vert^s \dd x \right)^{p/s}
    \\
    & \leq \vert E \vert^{\frac{s-p}{p}} m^{k_{\B}p} \sup_{\tilde{m} \in \N} \Vert \bar{w}_{n,\tilde{m}} \Vert_{W^{k_{\B}-1}_s}^p.
\end{align*}
Note that $\vert E\vert^{\frac{s-p}{p}} \to 0$, as $\vert E \vert \to 0$. Hence we assume that $m \leq M_3(n)$, with $M_3$ defined as 
\begin{equation}\label{Step4:defM3}
    M_3(n) := \left(\sup_{m \in \N} \Vert \bar{w}_{n,m} \Vert_{W^{k_{\B}-1}_s} \right)^{\frac{-1}{2 k_{\B}}} \longrightarrow \infty, 
    \quad \text{as } n \to \infty.
\end{equation}
We conclude that for any $0 \leq j \leq k_{\B}$ the set 
\[
    \left \{ \nabla^j w_{n,m} \colon n \in \N, m \leq M_3(n) \right\}
\]
is $p$-equi-integrable. 
\medskip

Finally, choosing a sequence $m(n) \to \infty$, as $n \to \infty$, with $m(n) \leq \min\{M_1(n),M_2(n),M_3(n)\} \to \infty$ completes the proof.
\end{proof}
 
 \begin{corollary}[Preservation of boundary conditions] \label{coro:boundaryequiint}
 Let $\Omega \subset \R^d$ be an open and bounded set with Lipschitz boundary. Suppose that $\A \colon C^{\infty}(\R^d;\R^m) \to C^{\infty}(\R^d;\R^l)$ is a homogeneous differential operator of order $k_{\A}$, satisfying the constant rank property. Let $v \in L_p(\Omega;\R^m)$ and let $v_n \subset L_p(\Omega;\R^m)$, such that $v_n \rightharpoonup v$ in $L_p(\Omega;\R^m)$ and $ \A v_n \to \A v$ in $W^{-k_{\A}}_p(\Omega;\R^l)$. Suppose that $\B$ is a potential of $\A$. \begin{enumerate} [label=(\roman*)]
     \item\label{item:cor1} Suppose that $v$ can be written as $v= \B u$. There exists a sequence $u_n \subset W^{k_{\B}}_p(\Omega;\R^h)$, such that\begin{enumerate}
         \item $u_n - u$ is compactly supported in $\Omega$;
         \item $\B u_n$ is $p$-equi-integrable;
         \item $\Vert \B u_n - v_n \Vert_{L_{r}(\Omega)} \to 0$ for some $1<r<p$.
     \end{enumerate}
     \item\label{item:cor2} There is a sequence $\bar v_n \subset L_p(\Omega;\R^m)$, such that \begin{enumerate}
         \item $\A \bar v_n = \A v$;
         \item $\bar v_n-v$ is compactly supported in $\Omega$;
         \item $\bar v_n$ is $p$-equi-integrable;
         \item $\Vert \bar v_n-v_n \Vert_{L_{r}(\Omega)} \to 0$ for some $1<r<p$.
     \end{enumerate}
 \end{enumerate}
 \end{corollary}
 Corollary \ref{coro:boundaryequiint} is used to modify sequences of functions in the constraint set $\Ccal$ to obtain equi-integrable sequences while at the same time preserving differential constraints and boundary conditions. Note that in problems of fluid mechanics the boundary conditions are typically given for $u$, the potential of $\epsilon$, therefore part \ref{item:cor1} is suitable for boundary conditions on the fluid velocity $u$ being the potential of the strain. On the other hand, boundary conditions for $\sigma$ are directly given in terms of the stress. Hence part \ref{item:cor2} is suitable there.

\bigskip
\subsection{Relaxation}

If the function $\funct$ is not $\A$-quasiconvex, the functional $I$ in \eqref{def:G} fails to be weakly lower-semicontinuous. Hence, we cannot ensure existence of minimisers just by using the \emph{direct method in the calculus of variations}.
However, when studying the data-driven problem, it is enough to consider approximate minimisers, i.e. minimising sequences $v_n$ with $I(v_n)$ converging to the infimum of $I$, and their weak limits $v^{\ast}$. In the following, we define a functional $I^{\ast}$ such that it is the \emph{relaxation} of $I$. Thus, any weak limit $v^{\ast}$ of a minimising sequence is a minimiser of $I^{\ast}$ and, vice versa, any minimiser of $I^{\ast}$ is a weak limit of approximate minimisers.


\bigskip
\subsubsection{Relaxation under a linear differential constraint.}

We recall the definition of $I$ from \eqref{def:G}. For simplicity, we use for the quasiconvex envelope of a function $\funct \colon \Omega \times \R^{m_1} \times \R^{m_2} \to \R$ the short-hand notation
\[
    \QA \funct(x,v) = \QA (\funct(x,\cdot))(v).
\]
Note that by Proposition $\ref{proposition:wlsc}$ the functional $I^{\ast}$ given by \[
I^{\ast}(v) := \begin{cases} \int_{\Omega} \QA \funct(x,v(x)) \dd x, & \A v =0 \\ \infty, & \text{else,} \end{cases}
\]
is weakly lower-semicontinuous in $L_p(\Omega;\R^{m_1}) \times L_q(\Omega;\R^{m_2})$. That $I^\ast$ is indeed the relaxation of $I$ is a consequence of the following (linear) result \cite{BFL}. 

\begin{proposition} \label{prop:relaxationlinear}
Let $(v_1,v_2) \in L_p(\Omega;\R^{m_1}) \times L_q(\Omega;\R^{m_2})$. Furthermore, let $\funct\colon \Omega \times \R^{m_1} \times \R^{m_2} \to \R$ satisfy the following assumptions:  
\begin{enumerate}[label=(A\arabic*)]
    \item \label{hypo:funct1} $\funct \colon \Omega \times (\R^{m_1} \times \R^{m_2}) \to \R$ is a Carath\'eodory function;
    \item  \label{hypo:funct2} there is $C>0$ such that for almost every $x \in \Omega$ and $(v_1,v_2)\in  \R^{m_1}\times \R^{m_2}$ it holds that
    \[
        0 \leq \funct(x,v_1,v_2) \leq C(1+ \vert v_1 \vert^p+ \vert v_2 \vert^q).
    \]
\end{enumerate} 
Then, for any $\varepsilon >0$ there exists a bounded sequence $v^n=(v_{1,n}^{\varepsilon},v_{2,n}^{\varepsilon})$ in $L_p(\Omega;\R^{m_1}) \times L_q(\Omega;\R^{m_2})$, such that 
\begin{enumerate}[label=(\roman*)]
    \item $v_{1,n}^{\varepsilon} \rightharpoonup v_1$ in $L_p(\Omega;\R^{m_1})$ and $v_{2,n}^{\varepsilon} \rightharpoonup v_2$ in $L_q(\Omega;\R^{m_2})$ as $n \to \infty$;
    \item $\A_1 v_{1,n}^{\varepsilon} = \A_1 v_1$ and $\A_2 v_{2,n}^{\varepsilon} =  \A_2 v_2$;
    \item $v_n^{\varepsilon}$ is almost a recovery sequence, i.e. \[
     \int_{\Omega} \QA \funct(x,v) \dd x \geq \lim_{n \to \infty} \int_{\Omega} \funct(x,v^{n,\varepsilon}) \dd x - \varepsilon.
    \]
\end{enumerate}
\end{proposition}

\begin{remark} \label{remark:coerc} 
     The (almost) recovery sequence $v_n^\eps$ in Proposition \ref{prop:relaxationlinear} is bounded in $L_p(\Omega;\R^{m_1}) \times L_q(\Omega;\R^{m_2})$ with a bound that depends on $\varepsilon$. Consequently, a priori we might not be able to take a weakly convergent diagonal sequence $v_n^{\varepsilon(n)}$, such that 
     \[
     \int_{\Omega} \QA \funct(x,v) \dd x 
     \geq 
     \lim_{n \to \infty} \int_{\Omega} \funct\left(x,v_n^{\varepsilon(n)}\right)\dd x.
     \]
     However, for fixed $v=(v_1,v_2) \in L_p(\Omega;\R^{m_1}) \times L_q(\Omega;\R^{m_2})$, let us define the constraint set $\Ccal_v$ as the set of functions $(w_1,w_2) \in L_p(\Omega;\R^{m_1}) \times L_q(\Omega;\R^{m_2})$ satisfying 
     \[
        \begin{cases}
            \A_1 w_1 = \A_1 v_1 
            &
            \\
            \A_2 w_2 = \A_2 w_2. &
        \end{cases} 
     \]
     We say that a functional $J$ is \emph{coercive} on $\Ccal_v$, provided 
     \begin{equation}\label{eq:def_coercivity}
        v \in \Ccal_v \text{ and } \Vert v \Vert \to \infty 
        \quad \Longrightarrow \quad
        J(v) \to \infty.
     \end{equation}
     If $J \colon v \mapsto \int_{\Omega} f(x,v) \dd x$ is coercive, there is a uniform bound on the $L_p(\Omega;\R^{m_1}) \times L_q(\Omega;\R^{m_2})$-norm of $v_n^{\varepsilon}$. By taking a diagonal sequence of $v_n^{\varepsilon}$ we may conclude the existence of a recovery sequence $v_n$ satisfying 
      \[
     \int_{\Omega} \QA \funct(x,v) \dd x \geq \lim_{n \to \infty} \int_{\Omega} \funct(x,v_n) \dd x.
    \]
    \emph{Coercivity} as defined in \eqref{eq:def_coercivity} is classically obtained by assuming that 
    \begin{equation} \label{classcoerc}
        f(x,v) \geq C_1 (\vert v_1 \vert^p + \vert v_2 \vert^q) -C_2.
    \end{equation}
    This strong pointwise coercivity condition is however not suitable for our setting. The distance function to a set $K$ only satisfies \eqref{classcoerc} if the set $K$ is bounded. Instead, we use a weaker coercivity condition of the type 
    \begin{equation}
        f(x,v) \geq C_1 (\vert v_1 \vert^p + \vert v_2 \vert^q) -\gamma v_1 \cdot v_2 -C_2.
    \end{equation}
    In general, $v_1 \cdot v_2$ does not have a good pointwise bound. Nevertheless, in the fluid mechanical setting, appropriate boundary conditions allow us to bound the integral $\int_{\Omega} v_1 \cdot v_2 \dd x$, cf. Section  \ref{sec:diffconst}. 
\end{remark}

\subsubsection{Relaxation under a semi-linear differential constraint.} 

As above, let $\Omega \subset \R^d$ be an open and bounded domain with Lipschitz boundary. Instead of considering a linear differential constraint, e.g. 
\[
    \begin{cases}
        \A_1 v_1 = 0 & 
        \\
        \A_2 v_2 = f,
    \end{cases}
\]
we include a semilinear term. In the fluid mechanical setting this semilinear term is given by
\[
    \epsilon \longmapsto (u \cdot \nabla) u,
\]
where $u$ is uniquely determined by $\epsilon$ due to boundary conditions and the constraint $\epsilon = \tfrac{1}{2}(\nabla u +\nabla u^T)$.

We fix a suitable general setting. Let, as before  $\A_1 \colon L_p(\Omega;\R^{m_1}) \to W^{-k_1}_p(\Omega;\R^{l_1}) $ be a constant rank operator with a potential $\B_1 \colon W^{k_{\B_1}}_p(\Omega;\R^{h_1}) \to L_p(\Omega;\R^{m_1})$ and $\A_2 \colon L_q(\Omega;\R^{m_2}) \to W^{-k_2}_p(\Omega;\R^{l_2})$ be a constant rank operator. In addition, we require the semilinear term to satisfy the following:
\begin{enumerate} [label=(A\arabic*)] \setcounter{enumi}{2}
       \item  \label{eq:nl_3} $\theta \colon \Omega \times \R^{h_1} \times (\R^{h_1} \otimes \R^d) \ldots \times (\R^{h_1} \times \R^{h_1} \otimes (\R^d)^{\otimes k_{\B_1}} \to \R^{m_1}$ is a continuous map;
    \item \label{eq:nl_4} The map $\Theta$ defined on $W^{k_{\B_1}}_p(\Omega;\R^{h_1})$ via \[
    (\Theta u)(x) = \theta\bigl(x,u(x),\nabla u(x),\ldots,\nabla^{k_{\B_1}}u(x)\bigr)
    \]
    is continuous from the \emph{weak} topology of $W^{k_{\B_1}}_p(\Omega;\R^{h_1})$ to the strong topology of  $L_r(\Omega;\R^{l_2})$ for some $r>q$.
\end{enumerate}
    
We study the following set of constraints: 
\begin{equation} \label{const:semilin}
\begin{cases}
    \A_1 v_1 = 0 & 
    \\
    v_1 = \B_1 u_1 & 
    \\
    \A_2 v_2 = \A_2 \Theta(u_1).&
\end{cases}
\end{equation}

\begin{theorem} \label{thm:relaxationnonlinear}
Let $\funct \colon \Omega \times \R^{m_1} \times Y \to \R$ satisfy the assumptions \ref{hypo:funct1}--\ref{hypo:funct2} from Proposition \ref{prop:relaxationlinear} and let $\Theta \colon L_p(\Omega;\R^{m_1}) \to W^{-1}_r(\Omega;\R^{l_2})$ and $\A_1$, $\A_2$ satisfy the aforementioned hypotheses \ref{eq:nl_3}--\ref{eq:nl_4}. 
Suppose that $u_1 \in W^{k_1}_p(\Omega;\R^{h_1})$ and $v=(v_1,v_2) \in L_p(\Omega;\R^{m_1}) \times L_q(\Omega;Y \times \R)$, such that $u_1 = \B_1 v_1$ and $\A_2 v_2 = \Theta(u_1)$. 
Then, for all $\varepsilon > 0$, there exist bounded sequences $u_{1,n}^{\varepsilon} \subset W^{k_1}_p(\Omega;\R^{h_1})$ and $v_{n}^{\varepsilon}  \subset L_p(\Omega;\R^{m_1}) \times L_q(\Omega;Y \times \R)$ such that 
\begin{enumerate} [label=(\roman*)]
     \item \label{relaxnl:1} $\B_1 u_{1,n}^{\varepsilon} = v_{1,n}^{\varepsilon}$;
     \item \label{relaxnl:2}$u_{1,n}^{\varepsilon} - u_1$ is supported in $\Omega_n \subset \subset \Omega$;
     \item \label{relaxnl:3}$\A_2 v_{2,n}^{\varepsilon} = \A_2\Theta(u_{1,n}^{\varepsilon})$;
     \item \label{relaxnl:4} $v_{2,n}^{\varepsilon}-v_2$ is supported in $\Omega_n \subset \subset \Omega$;
     \item $v_{n}^{\varepsilon}$ is almost a recovery sequence, i.e. it satisfies
    \[
        \int_{\Omega} \QA \funct(x,v) \dd x \geq \lim_{n \to \infty} \int_{\Omega} \funct(x,v_{n}^{\varepsilon}) \dd x - \varepsilon.
     \]
 \end{enumerate}
\end{theorem}

\begin{remark}
     \begin{enumerate}[label=(\roman*)]
         \item The statement of Theorem \ref{thm:relaxationnonlinear} is quite strong concerning boundary conditions. Indeed, the recovery sequence consisting of $u_{1,n}^\varepsilon$ and $v_{2,n}^\varepsilon$ preserves both the boundary conditions of $u_1$ \emph{and} the boundary conditions of $v_2$. Thus, it is possible to use the statement independently of the particular boundary conditions (Dirichlet, Neumann, \dots) in Section~\ref{sec:diffconst}.
        \item Remark \ref{remark:coerc} is still valid in the setting of Theorem \ref{thm:relaxationnonlinear}. More precisely, if we have a coercivity condition on the functional restricted to functions obeying \ref{const:semilin} and some boundary conditions, then we may find a recovery sequence satisfying \ref{relaxnl:1}--\ref{relaxnl:4} and 
        \[
            \int_{\Omega} \QA \funct(x,v) \dd x \geq \lim_{n \to \infty} \int_{\Omega} \funct(x,v_n) \dd x.
        \]
     \end{enumerate}
\end{remark}

\begin{proof}[\textbf{Proof of Theorem \ref{thm:relaxationnonlinear}}]
By the linear relaxation result Proposition \ref{prop:relaxationlinear} there exists a sequence $(\bar{v}_{1,n},\bar{v}_{2,n}) \subset L_p(\Omega;\R^{m_1}) \times L_q(\Omega;Y \times \R)$ weakly converging to $v=(v_1,v_2)$ satisfying 
\[
\begin{cases}
    \A_1 \bar{v}_{1,n}^{\varepsilon} = 0 & \\
    \A_2 \bar{v}_{2,n}^{\varepsilon}  = \A_2 v_2 =\A_2 \Theta(u_1) & \\
    \int_{\Omega} \QA \funct(x,v) \dd x \geq \lim_{n \to \infty} \int_{\Omega} \funct(x,v_n) \dd x - \varepsilon &
\end{cases}
\]
By Proposition \ref{proposition:equiintegrablewlsc} and Corollary \ref{coro:boundaryequiint} we may take $\tilde{u}_{1,n}^{\varepsilon} \in W^{k_1}_p(\Omega;\R^h)$, and $\tilde{v}_n \in L_p(\Omega;\R^{m_1}) \times L_q(\Omega;Y \times \R)$, such that 
\begin{enumerate}[label=(\roman*)]
    \item $\tilde{v}_{1,n}^{\varepsilon}= \B_1 \tilde{u}_{1,n}^{\varepsilon}$;
    \item the first $k_1$-derivatives of $\tilde{u}_{1,n}^{\varepsilon}$ are $p$-equi-integrable;
    \item $\tilde{v}_{2,n}^{\varepsilon}$ is $q$-equi-integrable;
    \item $\A_2 \tilde{v}_{2,n}^{\varepsilon}= \A_2 \Theta(u_1)$;
    \item the functions $\tilde{u}_{1,n}^{\varepsilon}$ and $\tilde{v}_{2,n}^{\varepsilon}$ satisfy the boundary conditions 
    \[
        \begin{cases}
             \mathrm{spt}(\tilde{u}_{1,n}^{\varepsilon}-u_1) \subset \Omega_n & 
             \\
            \mathrm{spt}(\tilde{v}_{2,n}^{\varepsilon} -v_2) \subset \Omega_n
        \end{cases}
    \]
    for some $\Omega_n \subset \subset \Omega$;
    \item $\int_\Omega \QA \funct(x,v)\dd x \geq \lim_{n\to \infty} \int_\Omega \funct(x,\tilde v_n^{\varepsilon})\dd x - \varepsilon$.
\end{enumerate}
We set $v_1^n=\tilde{v}_{1,n}^{\varepsilon}$ and $u_{1,n}^{\varepsilon}=\tilde{u}_{1,n}^{\varepsilon}$ and modify $\tilde{v}_{2,n}^{\varepsilon}$ by 
\[
    v_{2,n}^{\varepsilon}= \tilde{v}_{2,n}^{\varepsilon} +w_{2,n}^{\varepsilon}
\]
such that $\A_2 v_{2,n}^{\varepsilon} = \Theta(u_{1,n}^{\varepsilon})$. In particular, we solve the following equation:
\begin{equation} \label{diffequation:w}
    \begin{cases}
        \A_2 w_{2,n}^{\varepsilon} = \A_2(\Theta(v_{1,n}^{\varepsilon}) - \Theta(v_1)), 
        & x\in \Omega \\
       \mathrm{spt}(\tilde{w}_{2,n}^{\varepsilon} -v_2) \subset \subset \Omega
 \end{cases}
\end{equation}
But we know that $w_{2,n}^{\varepsilon} = \Theta(u_{1,n}^{\varepsilon}) - \Theta(u_1)$ already is a solution to this system. As $u_{1,n}^{\varepsilon}-u_1$ is supported inside $\Omega_n \subset \subset \Omega$, so is $u_{1,n}^{\varepsilon}$ due to the definition of the map $\Theta$, cf. \ref{eq:nl_3} and \ref{eq:nl_4}. Due to weak-strong continuity we have 
\[
 \Vert w_{2,n}^{\varepsilon} \Vert_{L_r} 
   =
     \Vert \Theta(u_{1,n}^{\varepsilon}) - \Theta(v_1) \Vert_{L_r} \longrightarrow 0 \quad \text{as } n \to \infty.
\]
 Then $v_{2,n}^{\varepsilon} := \tilde{v}_{2,n}^{\varepsilon} + w_{2,n}^{\varepsilon}$ still is $q$-equi-integrable, as $\tilde{v}_{2,n}^{\varepsilon}$ is $q$-equi-integrable and $w_{2,n}^{\varepsilon}$ bounded in $L_r(\Omega;Y \times \R)$ for some $r>q$; hence also $p$-equi-integrable. Moreover, as $v_{1,n}^{\varepsilon} \rightharpoonup v_1$ in $L_p(\Omega;Y)$ and $\Theta$ is weak-strong continuous,
\[
\Vert \tilde{v}_{2,n}^{\varepsilon} - v_{2,n}^{\varepsilon} \Vert_{L_r} = \Vert w_{2,n}^{\varepsilon} \Vert_{L_r} \longrightarrow 0 \quad \text{as } n\to \infty,
\]
and we conclude by Proposition \ref{proposition:equiintegrablewlsc} that \[
\liminf_{n \to \infty} \int_{\Omega} \funct(x,v_{1,n}^{\varepsilon},v_{2,n}^{\varepsilon}) \dd x \leq \liminf_{n \to \infty} \int_{\Omega}\funct(x,\tilde{v}_{1,n}^{\varepsilon}, \tilde{v}_{2,n}^{\varepsilon}) \dd x.
\]
As $\tilde{v}_{2,n}^{\varepsilon} - v_2$ is compactly supported in $\Omega$, $v_{2,n}^{\varepsilon} -v_2$ satisfies the demanded boundary conditions and $\A v_{2,n}^{\varepsilon} =\A_2 \Theta(v_{1,n}^{\varepsilon})$. Hence, (up to a subsequence) $v_n^{\epsilon}$ is almost a recovery sequence.

 \end{proof}

\begin{remark}
     The statement of Theorem \ref{thm:relaxationnonlinear} is taylored towards its application for fluid dynamics, cf. Subsection \ref{subsec:semilin}. Observe that in the proof of Theorem \ref{thm:relaxationnonlinear}, a main step was to solve the differential equation \begin{equation} \label{solvediffeq}
             \A_2 w = \A_2 \bigl(\Theta(u_{1,n}^{\varepsilon})- \Theta(u_1)\bigr)
         \end{equation}
         together with suitable boundary conditions. This equation is solved by the observation, that $(\Theta(u_{1,n}^{\varepsilon})- \Theta(u_1))$ already satisfies the boundary conditions.
         
    If we generalise the setting to other non-linearities, we need more assumptions on the non-linearity. For example, consider a constraint like \[
    \begin{cases}
    \A_1 v_1 = 0 & 
    \\
    v_1 = \B_1 u_1 & 
    \\
    \A_2 v_2 =  \zeta(u_1).
\end{cases}
    \]
    for some map $\zeta \colon W^{k_{\B_1}}_p(\Omega;\R^{h_1}) \to W^{-k_{\A_1}}_q(\Omega;\R^{h_2})$. Then weak-strong continuity is not enough, as one also needs to solve the analogue of \eqref{diffequation:w} with suitable boundary conditions. If for example, $\A_2 = \diverg$, then a further condition is as follows: Whenever $u_1$ and $u'_1$ satisfy $\mathrm{spt}(u_1-u'_1) \subset \subset \Omega$, then $\int \zeta(u_1)-\zeta(u'_1) \dd x =0$ (such that the divergence-equation is solvable, cf. \cite{Bogovski}).

\end{remark}


\section{Convergence of data sets} \label{sec:dataconv} 
In this section, we define two different notions of \textit{data convergence}, i.e. we define a suitable topology on closed subsets of $Y\times Y$. We show that these notions are equivalent to convergence of the unconstrained functionals $J$ in \eqref{intro:J}. In particular, these notions of data convergence are independent of the underlying differential constraint. Recall that we assume that the data consist of pairs of strain $\epsilon$ and the viscous part $\tigma$ of the stress; the pressure $\pi$ is not part of the data.
\subsection{Data convergence on bounded sets} \label{sec:dataconvbd}
 
\begin{definition} \label{def:boundedconv}
Let $Y \times Y $ be equipped with the metric $d\colon Y \times Y \to \R$ and $(\D_n), \D$ be closed, nonempty subsets of $Y\times Y$. We say that $\D_n$ converges to $\D$ strongly in the topology $\topobd$, $\D_n \limdata \D$, if the following is satisfied:
\begin{enumerate} [label=(\roman*)]
    \item \textbf{Uniform approximation:} There exists a sequence $a_n \to 0$, such that for all $z=(\epsilon, \tigma) \in \D$ it holds
    \[
    \dist(z,\D_n) \leq a_n (1  + \vert \epsilon \vert^p + \vert \tigma \vert^q ).
    \]
    \item \textbf{Fine approximation:} There exists a sequence $b_n \to 0$, such that for all $z_n=(\epsilon_n,\tigma_n) \in \D_n$ it holds
    \[
    \dist(z_n,\D) \leq b_n (1 + \vert \epsilon_n \vert^p + \vert \tigma_n \vert^q).
    \]
\end{enumerate}
\end{definition}

We consider the functionals defined on $V$ by \[
    J(v) = \int_{\Omega} \dist(v,\D) \dd x
    \quad \text{and} \quad
    J_n(v) = \int_{\Omega} \dist(v,\D_n) \dd x.
\]

\begin{theorem} \label{thm:equivconvergence}
Let $\D_n,\D$ be closed, nonempty subsets of $Y\times Y$. The following statements are equivalent:
\begin{enumerate}[label=(\roman*)]
    \item $\D_n \limdata \D$;
    \item For all $v \in V$ it holds that
    \[
        \lim_{n \to \infty} J_n(v) = J(v)
    \]
    and this convergence is uniform on bounded subsets of $V$.
\end{enumerate}
\end{theorem}

\begin{proof}
\textbf{`(i) $\Rightarrow$ (ii)'. } Suppose without loss of generality that $0 \in \D$. Otherwise we translate the underlying space which at most changes $a_n,b_n$ by a bounded factor. Let $v \in V$, with $\int_{\Omega} \dist(v,0) \dd x \leq R$. We assume without loss of generality that $p\ge q$. Then for $n \in \N$ we may estimate
\begin{align*}
    \int_{\Omega} \dist(v,\D) \dd x 
    &= 
    \int_{\Omega} d(v,\D)^p \dd x 
    \leq \int_{\Omega} \left(d(v,w_n) + d(w_n,\D) \right)^p \dd x, 
\end{align*}
where $w_n(x) \in \D_n$ is a point in $\D_n$ such that  $d(v(x),w_n(x)) = d(v(x),\D_n)$. Note that, as $0 \in \D$ and due to the uniform approximation property, we obtain a pointwise bound on $w_n$, i.e. $d(w_n(x),0) \leq 2 d(v(x),0)$ for $n$ large enough. Therefore, for some $\varepsilon >0$ we get \begin{align*}
    \int_{\Omega} \dist(v,\D) \dd x 
    & \leq 
    \int_{\Omega} \left(d(v, \D_n) + b_n \bigl(1+ \dist(w_n,0)\bigr)^{1/p})\right)^p \dd x \\
    &\leq 
    \int_{\Omega} \left(d(v,\D_n) + 2b_n \bigl(1+ \dist(v,0)\bigr)^{1/p}\right)^p \dd x \\
    &\leq (1+\varepsilon) \int_{\Omega} d(v,\D_n)^p + C(\varepsilon,p) b_n^p \bigl(1+ \dist(v,0)\bigr) \dd x \\
    &\leq  \int_{\Omega} \dist(v,\D_n) \dd x + \left(\varepsilon  \int_{\Omega} \dist(v,\D_n) \dd x+ C(\varepsilon,p) b_n^p (1 +R) \right).
\end{align*}
Note that $\int_{\Omega} d(v,\D_n)^p \dd x $ is bounded from above (for $n$ large enough) by $2 \int_{\Omega} d(v,0)^p \dd x \leq 2R$ as $0 \in \D$ and $0$ is approximated uniformly by elements of $\D_n$. Therefore, for any $\delta >0$ we may choose $\varepsilon$ and $n_0 \in \N$ such that for all $n > n_0$ we have
\[
    \varepsilon  \int_{\Omega} \dist(v,\D_n) \dd x < \frac{\delta}{2} 
    \quad \text{and} \quad
    C(\varepsilon,p) b_n^p (1 +R) < \frac{\delta}{2}.
\]
Consequently, there exists $\delta(R,n) \to 0$, such that for all $v \in V$ with $\int_{\Omega} \dist(v,0) \dd x \leq R$ it holds that
\begin{equation}
  J (u) \leq J_n(v) + \delta(R,n).  
\end{equation}
For the lower bound on $J(v)$ we can do the same calculation using fine instead of uniform approximation and find that for any $v \in V$ with $\int_{\Omega} \dist(v,0) \dd x \leq R$ we have \[
    \int_{\Omega} \dist(v,\D_n) \dd x \leq  \int_{\Omega} \dist(v,\D) \dd x + \left(\varepsilon  \int_{\Omega} \dist(v,\D) \dd x+ C(\varepsilon,p) a_n^p (1 +R) \right).
\]
We argue as for the lower bound, to obtain $\tilde{\delta}(R,n) \to 0$, such that for all $v \in V$ with $\int \dist(v,0) \dd x \leq R$ \begin{equation}
    J_n(v) \leq J(v) + \tilde{\delta}(R,h).
\end{equation}
Therefore, the convergence $J_n(v) \to J(v)$ is uniform on bounded subsets of $V$.
\bigskip

\noindent\textbf{`(ii)$\Rightarrow$ (i)'. } We prove the statement by contradiction. Suppose first, that $\D$ is \textit{not} uniformly approximated, i.e. there exists $a >0$ and a subsequence $z_{n_k}=(\epsilon_{n_k},\tigma_{n_k}) \subset \D$, such that 
\[
    \dist(z_{n_k},\D_{n_k}) > a \bigl(1 + \vert \epsilon_{n_k} \vert^p + \vert \tigma_{n_k} \vert^q\bigr) 
    = a \bigl(1 + \dist(z_{n_k},0)\bigr).
\]
We assume without loss of generality that $0 \in \D$. 
Let $\Sigma_{n_k}$ be a subset of $\Omega$ with measure $\vert \Omega \vert(1+\dist(z_{n_k},0))^{-1}$. We define 
\[
v_{n_k}(x):= \begin{cases}
    0, & x \notin \Sigma_{n_k} \\
    z_{n_k}, & x \in \Sigma_{n_k}.
\end{cases}
\]
Then $\int_{\Omega} \dist(v_{n_k},0)$ is bounded uniformly from above by $\vert \Omega \vert$. Furthermore, \begin{align*}
    \int_{\Omega} \dist(v_{n_k},\D) = 0, 
    \quad 
    k \in \N.
\end{align*}
On the other hand,
\begin{align*}
    \int_{\Omega} \dist(v_{n_k},\D_{n_k}) 
    &\geq \int_{\Sigma_{n_k}} \dist(z_{n_k}, \D_{n_k}) 
    \geq \vert \Sigma_{n_k} \vert \cdot  a \bigl(1 + \dist(z_{n_k},0)\bigr) 
    \geq \vert \Omega \vert a.
\end{align*}
  Therefore, $J_n(v)$ does not converge to $J(v)$ uniformly on bounded sets of $V$.
  \medskip
  
If $\D_n$ is \emph{not} a fine approximation of $\D$, the argument is similar. Then there exists $b>0$ and a subsequence $z_{n_k} \in\D_{n_k}$, such that, 
\[
    \dist(z_{n_k},\D) > b\bigl(1 + \dist(z_{n_k},0)\bigr).
\]
Again, assume that $0 \in \D$. We may assume that there exists a sequence $z'_n \to 0$ with $z'_n \in \D_n$, otherwise for $v \equiv 0$, it holds that
\[
    \limsup_{h \to \infty} \int_{\Omega} \dist(v,\D_n) \dd x> 0 
    = 
    \int_{\Omega} \dist(v,\D) \dd x.
\]
Let $\Sigma_{n_k}$ be a subset of $\Omega$ with measure $\vert \Omega \vert(1+\dist(z_{n_k},0))^{-1}$ and define 
\[
    v_{n_k}(x) := \begin{cases}
    0, & x \notin \Sigma_{n_k} \\
    z_{n_k}, & x \in \Sigma_{n_k}.
\end{cases}
\]
As argued before, $\int_{\Omega} \dist(v_{n_k},\D) \dd x $ is bounded uniformly by $\vert \Omega \vert$ and for $k \in \N$ we find that
\[
    \int_{\Omega} \dist(v_{n_k},\D_{n_k}) \dd x
    = \int_{\Omega \setminus \Sigma_{n_k}} \dist(0,\D_{n_k}) \dd x \longrightarrow 0 
    \quad \text{ as } k \to \infty.
\]
But, for the distance to $\D$ we have \begin{align*}
    \int_{\Omega} \dist(v_{n_k},\D) = \int_{\Sigma_{n_k}} \dist(z_{n_k},\D) 
    \geq 
    \vert \Sigma_{n_k} \vert \cdot  b\bigl(1 + \dist(z_{n_k},0)\bigr) =
    b \vert \Omega \vert. 
\end{align*}
Therefore, the convergence $J_n(v) \to J(v)$ cannot be uniform on bounded subsets of $V$.
\end{proof}

The definition of this type of convergence is motivated by Lemma \ref{gammaconv:1}. In particular, we have as a consequence that if $\D_n \limdata \D$, then the sequential $\Gamma$-limit of $J_n$ and of the constant sequence $J$ coincide, i.e 
\begin{align*}
    \Gamma-\lim_{n \to \infty} J_n &= \Gamma-\lim_{n \to \infty} J.
\end{align*}

\bigskip

\subsection{Data convergence on equi-integrable sets} \label{sec:dataconveq}

\begin{definition} \label{def:equiconv}
We say that a sequence of closed sets $\D_n \subset Y \times Y$ converges to $\D$ in the $\topoeq$-topology, $\D_n \limeq \D$, if there are sequences $a_n,b_n \to 0$ and $R_n, S_n \to \infty$ such that the following is satisfied: \begin{enumerate}[label=(\roman*)]
    \item \textbf{Uniform approximation on bounded sets:} For all $z \in \D$ with $\dist(z,0) < R_n$ we have 
    \[
    \dist(z, \D_n) \leq a_n (1 + \vert \epsilon \vert^p + \vert \tigma \vert^q). 
    \]
    \item \textbf{Fine approximation on bounded sets:} For all $z_n \in \D_n$ with $\dist(z_n,0) < S_n$ we have \[
    \dist(z,\D_n) \leq b_n (1 + \vert \epsilon_n \vert^p + \vert \tigma_n \vert^q).
    \]
\end{enumerate}
\end{definition}

\begin{remark} The following statements are equivalent to the \emph{uniform approximation on bounded sets:} 
\begin{itemize}
    \item For all $R>0$ there is a sequence $a_n^R \to 0$ such that for all $z \in \D$ with $\dist(z,0)<R$ we have 
    \[
    \dist(z,D_n) \leq a_n^R (1 + \vert \epsilon \vert^p + \vert \tigma \vert^q). 
    \]
    \item For all $a>0$ and $R>0$, there is an $n(a,R)$ such that for all $z \in \D$ with $\dist(z,0) < R$ and $n >n(a,R)$ we have
    \[
    \dist(z,D_n) \leq a (1 + \vert \epsilon \vert^p + \vert \tigma \vert^q). 
    \]
\end{itemize}
Similar equivalent statements hold for the \emph{fine approximation on bounded sets}.
\end{remark}

\begin{theorem} \label{thm:equiequivalence}
Let $\D_n,\D$ be closed, nonempty subsets of $Y\times Y$. The following statements are equivalent: 
\begin{enumerate} [label=(\roman*)]
    \item \label{thm:equiequivalence1}   $\D_n \limeq \D$ in the $\topoeq$-topology.
    \item \label{thm:equiequivalence2} The functionals $J_n$ converge uniformly to $J$ on $(p,q)$-equi-integrable subsets of $V$. That is, if $X \subset V$ is $(p,q)$-equi-integrable, then 
    \[
    \lim_{n \to \infty} \sup_{v \in X} \vert J_n(v) - J(v) \vert =0.
    \]
\end{enumerate}
\end{theorem}
\begin{proof}
\textbf{`\ref{thm:equiequivalence1} $\Rightarrow$ \ref{thm:equiequivalence2}':} The proof is similar to the proof of Theorem \ref{thm:equivconvergence}. We only prove that fine and uniform approximation imply that, for a $(p,q)$-equi-integrable subset $X \subset V$, we have
\begin{equation} \label{proof:equiequi:Claim}
\liminf_{n \to \infty} \inf_{v \in X} J_n(u) - J(u) \geq 0.
\end{equation}
The converse inequality follows similarly. 
For simplicity assume that $0 \in \D$ and that $p \geq q$. For some fixed $R>0$ we estimate 
\begin{equation} \label{eq:I_n_minus_I}
\begin{split}
    I_n(v) - I(v) &= \int_{\Omega} \dist(v,\D_n) - \dist(v,\D) \dd x \\
   & = \int_{\{ \dist(v,0)  \leq R\}} \dist(v,\D_n) - \dist(v,\D) \dd x + \int_{\{ \dist(v,0)  > R\}} \dist(v,\D_n) - \dist(v,\D) \dd x \\
   & \geq  \int_{\{ \dist(v,0)  \leq R\}} \dist(v,\D_n) - \dist(v,\D) \dd x -  C \int_{\{ \dist(v,0)  > R\}} (1 + \vert \epsilon \vert^p + \vert \tigma \vert^q ) \dd x.
\end{split}
\end{equation}
We now estimate both integrals on the right-hand side from below and start with the second term.
The set $X \subset V$ is $(p,q)$-equi-integrable. Hence, there is an increasing function $\omega \colon \R_+ \to \R_+$ such that 
\[
    \int_{E} (1 + \vert \varepsilon \vert^p + \vert \tigma \vert^q ) \dd x 
    \leq 
    \omega(\vert E \vert).
\]
The set $X$ is bounded. Thus, defining
\[
    M:= \sup_{v \in X} \int_{\Omega} 1+ \vert \varepsilon \vert^p + \vert \tigma \vert^q \dd x, 
\]
we find that the measure of $\{ \dist(v,0)  > R\}$ is bounded by $MR^{-1}$. Consequently, we obtain \begin{equation} \label{proof:equiequi1}
    - C \int_{\{ \dist(v,0)  > R\}} 1 + \vert \epsilon \vert^p + \vert \tigma \vert^q  \dd x \geq -C \omega(MR^{-1}).
\end{equation}
We turn to the first term in \eqref{eq:I_n_minus_I}. If $\dist(v(x),0) \leq R$, we may find some $w(x)\in \D$ with $\dist(w(x),0) \leq (2^p+2^q)R$, and \[
\dist(v(x),\D) = \dist(v(x),w(x)).
\]
Due to uniform approximation for all $w(x)$, we can estimate for $n$ large enough
\begin{align*}
 \int_{\{ \dist(v,0)  \leq R\}} \dist(v,\D_n) - \dist(v,\D) \dd x &= \int_{\{ \dist(v,0)  \leq R\}} d(v,\D_n)^p - d(v,\D)^p \dd x
\\&= \int_{\{ \dist(v,0)  \leq R\}}  d(v,\D_n)^p - d(v,w)^p \dd x \\
&\geq \int_{\{ \dist(v,0)  \leq R\}}  d(v,\D_n)^p - \bigl(d(v,\D_n) + d(w,\D_n)\bigr)^p \dd x \\
&\geq \int_{\{ \dist(v,0)  \leq R\}} - \varepsilon d(v,\D_n)^p - C_{\varepsilon} d(w,\D_n) \dd x \\
&\geq - \varepsilon M  - C_{\varepsilon} a_n M.
\end{align*}
Together with \eqref{proof:equiequi1} this implies \[
 J_n(v) - J(v) \geq -C \omega(M/R)  - \varepsilon M  - C_{\varepsilon} a_n M.
\]
Choosing $R(\varepsilon)$ and $n$ large enough, then for any $\varepsilon$ there is $n_{\varepsilon}$, such  that \[
J_n(v) - J(v) \geq - 2M \varepsilon, 
\quad  
v \in X,\ n\geq n_{\varepsilon},
\]
which establishes \eqref{proof:equiequi:Claim}.

\noindent\textbf{`\ref{thm:equiequivalence2} $\Rightarrow$ \ref{thm:equiequivalence1}':} This implication is a consequence of the same counterexamples as in Theorem \ref{thm:equivconvergence}. Indeed, suppose that the sets $\D_n$ do not uniformly approximate $\D$ on bounded sets. Then there exist $R>0$, $a>0$ and a sequence $z_{n_k} \subset \D$, such that $\dist(z_n,0) \leq R$ and
\[
\dist(z_{n_k},\D_{n_k}) \geq a (1 + \vert \epsilon_{n_k} \vert^p + \vert \tigma_{n_k} \vert^q). 
\]
By the same construction as in the proof of Theorem \ref{thm:equivconvergence}, that is
\[
    v_{n_k} := 
    \begin{cases} 0, & x \notin \Sigma_{n_k} \\ z_{n_k}, & x \in \Sigma_{n_k}, 
    \end{cases} 
\]
we obtain a sequence, such that $J(v_{n_k}) =0$ and $J_n(v_{n_k})\geq a \vert \Omega \vert$ with $v_{n_k}$ uniformly bounded in $L_{\infty}(\Omega;Y\times Y)$ and hence $v_{n_k}$ is also $(p,q)$-equi-integrable. For fine approximation the argument is again very similar.
\end{proof}


\section{The data-driven problem in fluid mechanics} \label{sec:diffconst}
In this section we apply the theory developed in the previous sections to the setting of fluid mechanics. We thus specialise to an explicit set of constraints $\Ccal$ consisting of differential constraints and boundary conditions. In Subsection \ref{subsec:inertialess} we consider the case of inertialess fluids, leading to a set of linear differential constraints. In Subsection \ref{subsec:semilin} we consider nonlinear differential constraints. In both cases we work with the following boundary conditions defined on three mutually disjoint  and relatively open parts of the boundary  $\Gamma_D,\Gamma_R,\Gamma_N\subset \partial \Omega$ that satisfy
\begin{displaymath}
    \overline{\Gamma_D\cup \Gamma_R\cup \Gamma_N}= \partial \Omega 
    \quad \text{and}\quad \mathcal{H}^{d-1}(\bar{\Gamma}_D \setminus \Gamma_D) =  \mathcal{H}^{d-1}(\bar{\Gamma}_R \setminus \Gamma_R) = \mathcal{H}^{d-1}(\bar{\Gamma}_N \setminus \Gamma_N) =0
\end{displaymath}
and have $C^1$-boundary as subsets of the manifold $\partial\Omega$.
We consider $(\epsilon,\tigma) \in L_p(\Omega;Y) \times L_q(\Omega;Y )$ with an associated velocity field $u:\Omega\to \R^d$, where $\epsilon=\tfrac 12 \bra{\nabla u+\nabla u^T}$ and a pressure field $\pi:\Omega\to \R$, such that $u$ and $\sigma=-\pi \id+\tigma$ satisfy the following boundary conditions.
\begin{description}
     \item[\namedlabel{Dirichlet}{(D)}]\textbf{No-slip/Dirichlet boundary conditions}: 
     \begin{align*}
         u = g\quad \text{on } \Gamma_D
         \quad
         \text{for } g\in W^{1-1/p}_p(\Gamma_D;\R^d).
     \end{align*}
     \item[\namedlabel{Navier}{(R)}] \textbf{Navier-slip/Robin boundary conditions}: 
     \begin{align*}
         \begin{cases}
             u\cdot \nu=g_\nu\\
             P_{T\partial\Omega}\bra{(\tigma + \pi \id)\nu+\lambda u}= h_\tau
        \end{cases}\quad \text{on } \Gamma_R
     \end{align*}
     for $g_{\nu} \in W^{1-1/p}_p(\Gamma_R)$ and $h_\tau\in  W^{-1/q}_q(\Gamma_R;\R^d)$. Here, $\lambda\ge 0$ is the inverse slip-length and $P_{T\partial\Omega}$ is the orthogonal projection to the tangent space. Note that the second equation can equivalently be cast as
     \begin{align}\label{eq:Robin}
        P_{T\partial\Omega}\bra{\tigma \nu+\lambda u}= h_\tau \quad \text{on } \Gamma_R.
    \end{align} 
     \item[\namedlabel{Neumann}{(N)}]  \textbf{Neumann boundary conditions}: 
     \begin{align*}
         (\tigma + \pi \id)\nu=h \quad\text{on } \Gamma_N
         \quad \text{for } h\in W^{-1/q}_q(\Gamma_N;\R^d).
     \end{align*}
\end{description}

\begin{remark}
\begin{enumerate}[label=(\roman*)]
\item The boundary conditions for $u$ can be understood as conditions for $\epsilon$ in a suitable weak formulation. For instance, if $\Gamma_D = \partial \Omega$, then \ref{Dirichlet} is equivalent to the following condition on $\epsilon$. For any $\varphi \in W^1_q(\Omega;Y)$ with $\diverg \varphi=0$ we have 
\[
    \int_\Omega \epsilon \cdot \varphi \dd x
    = 
    \int_{\partial\Omega} g (\varphi \cdot \nu) \dd \mathcal{H}^{d-1}.
\]
However, since an $\epsilon$ that is contained in the constraint set $\Ccal$ automatically admits a corresponding $u$ (see \eqref{def:Ccal1} below and following explanation), we write the conditions directly for $u$. A similar remark applies to the appearance of $\pi$.

\item The Navier-slip boundary condition \ref{Navier} requires $P_{T\partial\Omega}u\in W_q^{-1/q}(\Gamma_R;\R^d)$ since the other two terms in \eqref{eq:Robin} are contained in this space. Since $\epsilon\in L_p(\Omega;Y)$, and by Lemma \ref{kopo} together with a trace estimate, we have $u\in W^{1-1/p}_p(\Gamma_R;\R^d)$. The space $W^{1-1/p}_p(\Gamma_R)$ embeds into $W^{-1/q}_q(\Gamma_R)$, whenever either $p\geq q$ or 
 \[
     1-\tfrac{1}{p} - \tfrac{d-1}{p} \geq -\tfrac{1}{q} - \tfrac{d-1}{q}.
\]
Thus, since $q= \tfrac{p}{p-1}$, we require 
\begin{align}\label{eq:cond_p_Navier}
    p \geq \tfrac{2d}{d+1}.
\end{align}
We can therefore treat the Navier-slip boundary condition in the physically relevant dimensions $d=2$ and $d=3$ for $p\geq 4/3$ and for $p\geq 3/2$, respectively.

\item The Navier boundary condition \ref{Navier} includes the so called free-slip boundary condition for $\lambda=0$.

\item For simplicity we assume in the following that either $\Gamma_N=\partial \Omega$ or $\Gamma_D\neq\emptyset$. This allows us to control $\norm{u}_{W^1_p}$ in terms of $\norm{\epsilon}_{L_p}$ and the boundary data via the Korn--Poincar{\'e} inequality, cf. Lemma \ref{kopo}. If $\Gamma_R\neq \emptyset$, while $\Gamma_D=\emptyset$, it becomes tedious to specify under which conditions this control can still be obtained. See Lemma \ref{l:Korn} and Remark \ref{r:Korn} below.
\end{enumerate}
\end{remark}

In order to obtain a Korn--Poincar\'e type inequality, $u$ has to be uniquely determined by the above boundary conditions 
\begin{equation} \label{babyboundary}
    \begin{cases}  
    u = g, & x \in \Gamma_D 
    \\
    u \cdot \nu = g_{\nu}, & x \in \Gamma_R \end{cases}
\end{equation}
and the constraint 
\[
    \epsilon = \tfrac 12\bra{\nabla u + \nabla u^T},
\]
or the conditions must be invariant under renormalisation by rigid body motions.

\begin{lemma}[Validity of the Korn-Poincar\'e ineqaulity under boundary conditions]\label{l:Korn}
Let $\Omega \subset \R^d$ be open and bounded with $C^1$-boundary and let $\partial\Omega = \bar{\Gamma}_D \cup \bar{\Gamma}_R\cup \bar{\Gamma}_N$ be as specified above. Moreover, suppose that $g\in W^{1-1/p}_p(\partial \Omega;\R^d)$, $g_{\nu}\in W^{1-1/p}_p(\partial \Omega)$ and that for all $A \in \R^{d \times d}_{\skw},~b \in \R^d$ we have
\begin{equation} \label{eq:failure}
    \begin{cases}
        Ax + b =0, & x \in \Gamma_D 
        \\ 
        (Ax+b) \cdot \nu(x) =0, & x \in \Gamma_R
    \end{cases}
    \quad \Longrightarrow\quad A=0,\ b=0.
\end{equation}
Then the following statements hold true:
\begin{enumerate} [label=(\roman*)]
    \item\label{it:korn1} If $u_1$ and $u_2$ satisfy \eqref{babyboundary} and \[
    \nabla u_1 + \nabla u_1^T = \nabla u_2 + \nabla u_2^T,
    \]
     then $u_1=u_2$.
     \item\label{it:korn2} For all $u \in W^{1,p}(\Omega;\R^d)$ obeying \eqref{babyboundary} with $\Gamma_D \neq \emptyset$, the Korn--Poincar\'e inequality 
     \begin{equation} \label{eq:kopo:inhom}
        \Vert u \Vert_{W^{1,p}} \leq C (1+ \Vert \nabla u + \nabla u^T \Vert_{L_p})
     \end{equation}
     holds for a constant $C=C(\Omega,\Gamma_D,\Gamma_R,g,g_{\nu},p)$.
\end{enumerate}
\end{lemma}
\begin{proof}
\textbf{\ref{it:korn1}:} The assertion follows from the fact that if $\nabla u_1 + \nabla u_1^T = \nabla u_2 + \nabla u_2^T$, then $u_1-u_2 = Ax +b$ for some $A \in \R^{d \times d}_{\skw}$ and $b \in \R^d$. Condition \eqref{eq:failure} then implies that $A=0$ and $b=0$.

\noindent\textbf{\ref{it:korn2}:} The vector space $X \subset W^1_p(\Omega;\R^d)$ of functions satisfying the homogeneous boundary conditions in \eqref{babyboundary} satisfies, due to \eqref{eq:failure},\[
X \cap \{ A x + b \colon A \in \R^{d \times d}_{\skw}, b \in \R^d\} = \{0\}.
\]
By transposition we get the inhomogeneous version \eqref{eq:kopo:inhom} for the affine space of functions satisfying \eqref{babyboundary}.
\end{proof}
\begin{remark}\label{r:Korn}
Indeed, \eqref{eq:failure} is a rather weak condition on the set $\Omega$. For example, in dimension $d=2$, the weakest boundary condition in the case $\Gamma_D=\emptyset$ would be 
\[
    (Ax + b) \cdot \nu(x) =0 \quad \text{on } \Gamma_R.
\]
Since $\R^{d\times d}_{\skw}$ is one-dimensional, we can explicitly set 
\begin{equation*}
    A=\left( \begin{array}{cc} 0 & 1 \\ -1 & 0 \end{array} \right).
\end{equation*}     
It follows that the only sets not satisfying \eqref{eq:failure} are such that $\Gamma_R$ is a subset of concentric circles. Moreover, if $\Gamma_D\neq \emptyset$, then \eqref{eq:failure} is automatically satisfied.
 
In dimension $d=3$, the situation is similar. Indeed, if $\Gamma_D\neq \emptyset$, then \eqref{eq:failure} is satisfied. If $\Gamma_D=\emptyset$, then, if $\Gamma_R$ is a subset of the boundary of a domain that is rotationally symmetric around a certain axis, \eqref{eq:failure} is not satisfied.
\end{remark}
\begin{remark}
    Uniqueness of $u$ is only important for fluids with inertia. For inertialess fluids, $u$ only appears in the constraints through boundary conditions. Therefore, even if $\epsilon = \tfrac{1}{2}(\nabla u_1 + \nabla u_1^T) = \tfrac{1}{2}(\nabla u_2 + \nabla u_2^T)$ for $u_1 \neq u_2$ enjoying the same boundary conditions, it \emph{does not} matter for the system of equations whether we take $u_1$ or $u_2$. In contrast, for fluids with inertia, the contribution $(u \cdot \nabla) u$ in the differential constraints causes the choice of $u$ to be important.
    Therefore, in the linear setting, even if the prescribed boundary conditions \ref{Dirichlet}, \ref{Navier} and \ref{Neumann} allow to choose different $u \in W^{1}_p(\Omega;\R^d)$, for example if $\Gamma_N=\partial \Omega$, we may project onto a subspace that does not allow multiple solutions to \[
    \epsilon = \tfrac{1}{2} \left(\nabla u + \nabla u^T\right).
    \]
    Consequently, we can apply Lemma \ref{kopo} in this situation.
\end{remark}

\subsection{Inertialess fluids}\label{subsec:inertialess}
In this section we study inertialess fluids leading to the set of \emph{linear} differential constraints from \eqref{intro:linear_C}. That is, we consider
\begin{equation}\tag{linD}\label{def:Ccal1}
    \begin{cases}
        \epsilon = \frac{1}{2}\left(\nabla u + \nabla u^T\right) 
        &
        \\
        \Div u = 0 &
         \\
        -\Div \tigma = f - \nabla \pi, &
    \end{cases}
\end{equation}
where $f\in W^{-1}_q(\Omega;\R^d)$ is given. Combining this with the boundary conditions, the constraint set is given by
\begin{align}\tag{linC}\label{eq:Ccal}
    \Ccal_{\lin}\coloneqq\{(\epsilon,\tigma)\in V\colon \eqref{def:Ccal1},\ref{Dirichlet},\ref{Navier},\text{ and }\ref{Neumann}\text{ are satisfied}\}.
\end{align}
Note that the statement `$(\epsilon,\tigma)$ satisfies \eqref{def:Ccal1}' means that there are $u\in W^1_p(\Omega;\R^d)$ and $\pi\in L_q(\Omega)$ such that \eqref{def:Ccal1} is satisfied. For data sets $\D_n, \D \subset Y \times Y$ we consider the functionals $I_n$ and $I$ as in \eqref{intro:I}. 

\subsubsection{Coercivity} 
In this subsection we verify coercivity of the functionals $I_n$ and $I$.

\begin{definition}
We call a function $\funct \colon Y \times Y \to \R$ $\mathbf{(p,q)}$\textbf{-coercive}, if there exist $C_1,C_2>0$ and $\gamma \in \R$ such that 
\begin{align}\label{eq:pq_coercive}
    \funct(\epsilon,\tigma) \geq C_1(\vert \epsilon \vert^p+\vert \tigma \vert^q) - C_2 - \gamma \epsilon\cdot\tigma.
\end{align}
We say that $\funct$ has $\mathbf{(p,q)}$\textbf{-growth}, if there is $C_0>0$ such that 
\[
\funct(\epsilon,\tigma) \leq C_0(1  + \vert \epsilon \vert^p + \vert \tigma \vert^q).
\]
\end{definition}
For $v \in V$ we define the functional
\begin{align}\label{eq:I_f}
    I(v) := 
    \begin{cases} 
        \int_{\Omega} \funct(v) \dd x, & v \in \Ccal_{\lin} \\ 
        \infty, & \text{else,}
    \end{cases}
\end{align}
in analogy to \eqref{intro:I}.
\begin{remark}
    In Section \ref{sec:dataconv} we examine data convergence without the differential constraints, in particular we study the unconstrained functional $J$. In general, we do not expect a coercivity statement of the type 
    \[
        \Vert v \Vert_{V} \to \infty \quad \Longrightarrow \quad J(v) \to \infty.
    \]
    In the following we prove that coercivity follows in the presence of the differential constraints together with suitable boundary conditions, i.e. it holds that
    \[
        \Vert v \Vert_{V} \to \infty,~v \in \Ccal_{\lin} \quad \Longrightarrow \quad I(v)=J(v) \to \infty.
    \]
    We can include the term $\epsilon\cdot \tigma$ on the right-hand side of \eqref{eq:pq_coercive} because it is a Null-Lagrangian. This becomes clear in Remark \ref{r:null_lag} and in the proof of Lemma \ref{l:coercivity_omega} below. In some sense we only require coercivity away from the collinearity set $\{(\epsilon,\tigma): \epsilon =\beta \tigma, \beta\in \R\}$. Because we expect $\epsilon$ and $\tigma$ to be colinear for classical fluids, this kind of transversal coercivity is a natural condition for the distance to the data sets which takes the role of $\funct$ later on.
\end{remark}
\begin{remark} \label{r:null_lag}
For the purpose of exposition, we prove a coercivity result for functions on the torus. Here, \emph{averages} of the functions $(\epsilon,\tigma)$ take over the role of \emph{boundary values} and the role of the differential constraints can be isolated more clearly.

Let $\funct$ be $(p,q)$-coercive. We claim that there are constants $C_1,C_2>0$, such that for any $(\epsilon_0,\tigma_0) \in Y \times Y$ and  all $(\epsilon,\tigma) \in L_p(\T_d;Y) \times L_q(\T_d;Y)$ satisfying  
\begin{equation} \label{lintorus} 
\begin{cases}
    \int_{\T_d} (\epsilon,\tigma) \dd x = 0& \\
    \epsilon=\tfrac{1}{2} \left(\nabla u +\nabla u^T\right)& \\
    \diverg \tigma = \nabla \pi, 
\end{cases}
\end{equation}
 for some  $\pi \in L_q(\T_d)$, we have the following coercivity: 
\begin{equation} \label{coerc:torus}
    \int \funct(\epsilon_0 +\epsilon,\tigma_0+\tigma) \dd x \geq c_1 \int_{\T_d} \vert \epsilon \vert^p +\vert \tigma \vert^q \dd x - c_2 (1+\vert \epsilon_0 \vert^p + \vert \tigma_0 \vert^q).
\end{equation}
We compute
\begin{align*}
    \int_{\T_d} (\epsilon_0+\epsilon) &\cdot (\tigma_0+\tigma) \dd x= \int_{\T_d} \epsilon \cdot \left((\tigma_0+\tigma) + (\pi_0 + \pi) \id\right) \dd x  + \varepsilon_0 \cdot \int_{\T_d}\left((\tigma_0+\tigma) + (\pi_0 + \pi) \id\right) \dd x \\
    &= \int_{\T_d} \frac 12\bra{\nabla u + \nabla u^T} \left((\tigma_0+\tigma) + (\pi_0 + \pi) \id\right) \dd x + \varepsilon_0 \cdot \int_{\T_d}\left(\tigma_0 + \pi_0\id\right) \dd x  \\
    &= \int_{\T_d} \nabla u\left((\tigma_0+\tigma) + (\pi_0 + \pi) \id\right) \dd x +  \varepsilon_0 \cdot \tigma_0  \\
    &= -\int_{\T_d} u \cdot \diverg(\tigma+\pi \id) \dd x + \varepsilon_0 \cdot \tigma_0
    = \varepsilon_0 \cdot \tigma_0.
\end{align*}
Therefore, \[
\left \vert \int_{\T_d} (\epsilon_0+\epsilon) \cdot (\tigma_0+\tigma) \dd x \right \vert \leq \vert \epsilon_0 \vert^p + \vert \tigma_0 \vert^q.
\]
We conclude that \begin{align*}
    \int \funct(\epsilon_0 +\epsilon,\tigma_0+\tigma)& \geq C_1 \int_{\T_d} \vert \varepsilon_0 + \varepsilon \vert^p + \vert \tigma_0 + \tigma |^q \dd x -C_2 - \gamma \int_{\T_d} \epsilon \cdot \tigma \dd x \\
    &\geq C_1\int_{\T_d} \vert \epsilon \vert^p + \vert \tigma \vert^q \dd x - C^\prime_2(1+ \vert \epsilon_0 \vert^p + \vert \tigma_0 \vert^q).
\end{align*}
\end{remark}

Using the boundary conditions instead of averages, we obtain coercivity of the functional also on bounded domains, as long as the integrand is $(p,q)$-coercive. 

\begin{lemma}[Coercivity in $\Omega$ with boundary values]\label{l:coercivity_omega}
Suppose that $f,g,g_\nu,h_\tau$, and $h$ are given as in  \eqref{def:Ccal1}, \ref{Dirichlet}, \ref{Navier}, and \ref{Neumann}. We assume that either $\Gamma_N=\partial \Omega$ or $\Gamma_D\neq \emptyset$. If $\Gamma_R\neq \emptyset$, then we additionally assume $p\ge 2d/(d+1)$. Suppose that $\funct \colon Y \times Y \to \R$ is $(p,q)$-coercive and has $(p,q)$-growth. Then there are $C_3,C_4>0$,such that for $I$ from \eqref{eq:I_f} and for all $v=(\epsilon,\tigma) \in V$
\[
    I(v) \geq C_3 \int_{\Omega}(\vert \epsilon \vert^p + \vert \tigma \vert^q) \dd x  -C_4.
\]
\end{lemma}

\begin{proof}
 We may assume that $v \in \Ccal_{\lin}$, otherwise there is nothing to show. By the coercivity of $\funct$ we have
\begin{align}\label{eq:omeg_coerc}
    I(v) 
    = \int_{\Omega} \funct(\epsilon,\tigma) \dd x 
    \geq 
    \int_{\Omega} C_1( \vert \epsilon \vert^p+\vert \tigma \vert^q) -C_2-\gamma \epsilon\cdot \tigma \dd x.
\end{align}
Since $v\in \Ccal_{\lin}$,
\[
    \epsilon = \tfrac{1}{2}\left(\nabla u + \nabla u^T\right),
\]
for some $u$ with
\[
\Vert u \Vert_{W^1_p} \leq C\bra{1+\Vert \epsilon \Vert_{L_p}} ,
\]
due to the Korn-Poincar\'e inequality from Lemma~ \ref{l:Korn}\ref{it:korn2} . Furthermore we have the following estimate
\begin{align}\label{eq:tigma_pi_est}
    \norm{\tigma \nu}_{W^{-1/q}_q(\partial\Omega)}+\norm{\pi \nu}_{W^{-1/q}_q(\partial\Omega)}&\le C\bra{\norm{\tigma}_{L_q}+\norm{f}_{W^{-1}_q}},
\end{align}
which is due to $- \Div \tigma+\nabla \pi=f$. Let us now estimate the last term in \eqref{eq:omeg_coerc}. The following computations will be done under the assumption that all functions are smooth. The statement follows by density. Observe that 
\begin{align}
    \int_{\Omega} \epsilon \cdot \tigma\dd x &= \int_{\Omega} \tfrac 12 \bra{\nabla u + \nabla u^T} \cdot (\tigma-\pi\id)\dd x = \int_{\Omega} \nabla u \cdot (\tigma-\pi\id) \dd x\nonumber\\
    &= -\int_{\Omega} u \cdot (\diverg \tigma-\nabla\pi) \dd x+ \int_{\partial \Omega} u\cdot (\tigma-\pi\id) \nu \dd \Haus^{d-1} \nonumber\\
    &=  \int_{\Omega} u\cdot f \dd x + \int_{\partial \Omega} u\cdot (\tigma-\pi\id)\nu \dd \Haus^{d-1}. \label{eq:eps_tig}
\end{align}
On the one hand, we have the following estimate for the bulk term 
\begin{align}\label{eq:f_est_lin}
    \left \vert \int_{\Omega} u\cdot f\dd x \right \vert \leq  \Vert  u \Vert_{L_p} \Vert f \Vert_{L_q} \leq  C\bra{1+\Vert \epsilon \Vert_{L_p}}\Vert f \Vert_{L_q}.
\end{align}
On the other hand, the boundary contribution can be estimated on the Dirichlet part by
\begin{align}
     \left \vert\int_{\Gamma_D} u\cdot(\tigma - \pi \id) \nu\dd \Haus^{d-1}\right\vert
     &= \left\vert\int_{\Gamma_D} g\cdot (\tigma - \pi \id) \nu\dd \Haus^{d-1}\right\vert\nonumber\\
     &\le  \Vert g \Vert_{W^{1-1/p}_p(\Gamma_D)} \bra{\norm{(\tigma - \pi \id) \nu}_{W^{-1/q}_q(\Gamma_D)}}\nonumber\\
     & \le \Vert g \Vert_{W^{1-1/p}_p(\Gamma_D)} \bra{\norm{\tigma - \pi \id \nu}_{W^{-1/q}_q(\Gamma_D)}} \nonumber\\
    &\leq C\bra{ \Vert \epsilon \Vert_{L_p}+\norm{\tigma}_{L_q}+\norm{f}_{W^{-1}_q}},\label{eq:dir}
\end{align}
on the Navier part by first isolating the term with sign
\begin{align}
\int_{\Gamma_R} u\cdot(\tigma - \pi \id) \nu\dd \Haus^{d-1}
     &= \int_{\Gamma_R} g_\nu \nu\cdot (\tigma - \pi \id) \nu-\lambda \abs{P_{T_x\partial\Omega}u}^2+P_{T_x\partial\Omega}u\cdot h_\tau\dd \Haus^{d-1}, \label{eq:rob}
\end{align}
and then estimating
\begin{align}
&\left \vert\int_{\Gamma_R} g_\nu \nu\cdot (\tigma - \pi \id) \nu+P_{T_x\partial\Omega}u\cdot h_\tau\dd \Haus^{d-1}\right\vert\nonumber\\
     &\quad\le  \Vert g_\nu \Vert_{W^{1-1/p}_p(\Gamma_R)} \Vert (\tigma - \pi \id) \nu \Vert_{W^{-1/q}_q(\Gamma_R)}+\norm{u}_{W^{1-1/p}_p(\Gamma_R)}\Vert h_\tau \Vert_{W^{-1/q}_q(\Gamma_R)} \nonumber\\
    &\quad\leq C\bra{1+\Vert \epsilon \Vert_{L_p}+\norm{\tigma}_{L_q}+\norm{f}_{W^{-1}_q}}\label{eq:rob2},
\end{align}
and on the Neumann part by
\begin{align}
\left \vert\int_{\Gamma_N} u\cdot(\tigma - \pi \id) \nu\dd \Haus^{d-1}\right\vert
    = \left\vert\int_{\Gamma_N}u\cdot h\dd \Haus^{d-1}\right\vert
    \le  \Vert u \Vert_{W^{1-1/p}_p(\Gamma_N)} \Vert h \Vert_{W^{-1/q}_q(\Gamma_N)} 
    \leq C_h \Vert \epsilon \Vert_{L_p}. \label{eq:neu}
\end{align}
Inserting \eqref{eq:rob} into \eqref{eq:eps_tig} and using the result together with \eqref{eq:f_est_lin}, \eqref{eq:dir}, \eqref{eq:rob2}, and \eqref{eq:neu} in \eqref{eq:omeg_coerc} yields
\begin{align}
    I(v) &\geq C_1 \left(\Vert \epsilon \Vert_{L_p}^p+\Vert \tigma \Vert_{L_q}^q  \right)-C_2 -\gamma \int_{\Omega} \epsilon \cdot \tigma \dd x\nonumber \\
    & \geq  C_1 \left(\Vert \epsilon \Vert_{L_p}^p+\Vert \tigma \Vert_{L_q}^q\right)- C\bra{\Vert \epsilon \Vert_{L_p}+\Vert \tigma \Vert_{L_q}+1} \nonumber\\
    &\geq \frac{C_1}{2} \left(\Vert \epsilon \Vert_{L_p}^p+\Vert \tigma \Vert_{L_q}^q \right) -C,\label{eq:coerc1}
\end{align}
where we used Young's inequality in the last step and the constants depend on $d,\Omega,f,g,g_\nu,h,h_\tau$.
\end{proof}

Lastly we check, that indeed the function $\dist(\cdot,\D)$ is $(p,q)$-coercive if $\D$ contains data for which `$\epsilon$ and $\tigma$ are aligned well enough'.

\begin{lemma}
The distance function $\dist(\cdot,\D)$ to a set $\D \subset Y \times Y$ is $(p,q)$-coercive if and only if there are $c_1 \in \R$ and $c_2>0$, such that 
\begin{equation} \label{eq:Dcond}
 \D \subset \{ (\epsilon,\tigma) \in Y\times Y \colon c_1  \epsilon\cdot \tigma +c_2 > \vert \epsilon \vert^p + \vert \tigma \vert^q\}.
\end{equation}
\end{lemma}

\begin{remark}
Condition \eqref{eq:Dcond} means that the data very roughly behaves like a power law for data points with large strain, i.e. $\sigma \sim \beta \vert \varepsilon \vert^{\alpha-1} \varepsilon$ whenever $(\sigma,\epsilon) \in \D$ for $\alpha = p-1$. The factor $\beta$ however might depend on the strain $\epsilon$.
\end{remark}

\begin{proof}
\textbf{`$\Longrightarrow$':}
Suppose first that the distance function to $\D$ is $(p,q)$-coercive, i.e. \[
\dist((\epsilon,\tigma),\D) \geq C_1 (\vert \epsilon \vert^p + \vert \tigma \vert^q) - C_2 - \gamma \epsilon \cdot \tigma.
\]
Then, for all $(\epsilon,\tigma) \in \D$ we have  \[
0 \geq C_1 (\vert \epsilon \vert^p + \vert \tigma \vert^q) - C_2 - \gamma  \epsilon \cdot \tigma
\]
and therefore, \[
(\epsilon,\tigma) \in \D \quad \Longrightarrow  \quad  \vert \epsilon \vert^p + \vert \tigma \vert^q < c_2 + c_1 \epsilon \cdot \tigma.
\]

\noindent\textbf{`$\Longleftarrow$':} For the converse direction we need to prove that the distance function to the set \[
 \D= \{ (\epsilon,\tigma) \in Y \times Y \colon c_1 \epsilon\cdot \tigma +c_2 > \vert \epsilon \vert^p + \vert \tigma \vert^q\}
\]
is $(p,q)$-coercive. The constant $c_2$ only makes $\D$ thicker by a finite amount. To see this, for $(\epsilon,\tigma)\in \D$, write $\tigma=\alpha \epsilon+\tigma^\perp$ with $\epsilon\cdot \tigma^\perp=0$ and define $\tigma_\beta=\alpha \epsilon+\beta\tigma^\perp$. Since $\epsilon\cdot \tigma=\alpha\abs{\epsilon}^2$ we must have $\vert\tigma^\perp\vert^q\le c_2+c_\alpha\abs{\epsilon}$ because of $(\epsilon,\tigma)\in\D$. Then $\abs{\tigma_\beta}^q\le c_{q} \abs{\alpha \epsilon}^q+\beta^q\abs{\tigma^\perp}^q$ while $\epsilon\cdot \tigma=\epsilon\cdot \tigma_\beta$. Decreasing $\beta$, we find a $\tigma_\beta$ such that $c_1 \epsilon \cdot \tigma > \vert \epsilon \vert^p + \vert \tigma \vert^q$and such that $\dist((\epsilon,\tigma),(\epsilon,\tigma_\beta))$ is bounded independently of $(\epsilon,\tigma)$.

Thus, we may assume that $c_2=0$ since this only shifts $C_2$ in \eqref{eq:pq_coercive}. Then $\D$ is $(p,q)$-homogeneous, i.e. $(\epsilon,\tigma) \in \D \Rightarrow (\lambda \epsilon, \lambda^{p/q} \tigma)\in \D$ for all $\lambda>0$. This in turn implies that the distance function is $(p,q)$-homogeneous, i.e. \begin{equation}\label{eq:homdist} 
    \dist\left((\lambda \epsilon, \lambda^{p/q} \tigma),\D\right) = \lambda^p \dist\left((\epsilon,\tigma),\D\right).
\end{equation}
for all $\lambda >0$. Let $S= \{\vert \epsilon \vert^p + \vert \tigma \vert^q=1\}$ be the unit sphere. Then the set 
\[
E\coloneqq S \cap \{2c_1 \epsilon \cdot \tigma  \le \vert \epsilon \vert^p + \vert \tigma \vert^q\}
\]
is compact and has positive distance to $\D$, i.e. there exists $a>0$ such that
\[
    (\epsilon,\tigma) \in E \quad \Longrightarrow \quad \dist((\epsilon,\tigma),\D)  >a.
\]
Hence, setting 
\begin{equation*}
    c=\max_{(\epsilon,\tigma)\in E}(\vert \epsilon \vert^p + \vert \tigma \vert^q-2c_1 \epsilon \cdot \tigma),
\end{equation*}
we have
\[
(\epsilon,\tigma) \in S \quad  \Longrightarrow \quad \dist((\epsilon,\tigma),\D) \geq \frac ac ( \vert \epsilon \vert^p + \vert \tigma \vert^q - 2 c_1 \epsilon \cdot \tigma),
\]
where we use that the right-hand side is smaller than $0$ on in the complement of $E$, while it is smaller than $a$ in $E$.
This and \eqref{eq:homdist} show that the distance function $\dist$ is $(p,q)$-coercive.

\end{proof}

\subsubsection{$\Gamma$-convergence} 

\begin{theorem}[$\Gamma$-convergence in the linear setting] \label{thm:gammaconvlin}
Let $\D_n,\D \subset Y\times Y$ be closed, nonempty sets, and let $\Ccal_{\lin}$ be given by \eqref{eq:Ccal}. Moreover, suppose that \begin{enumerate} [label=(\roman*)]
    \item The distance functions to $\D_n$ and $\D$ are \emph{uniformly} $(p,q)$-coercive, i.e. there are $c_1,c_2$, such that
    \[
    \D_n,\D \subset  \{ (\epsilon,\tigma) \in V \times V \colon c_1 \epsilon \cdot \tigma +c_2 > \vert \epsilon \vert^p + \vert \tigma \vert^q\};
    \]
    \item $\D_n \limeq \D$;
    \item if $\Gamma_R\neq \emptyset$, let $p\ge \frac{2d}{d+1}$.
\end{enumerate}
Then the functional $I_n$ $\Gamma$-converges to $I^{\ast}$, where \[
I^{\ast}(v) = \begin{cases}
       \int_{\Omega} \QA \dist(v,\D) \dd x, & v \in \Ccal_{\lin} \\
       \infty, & \text{else.}
\end{cases}
\]
\end{theorem}

\begin{proof}The hypotheses of Theorem \ref{theorem:equiint} are all satisfied with $\funct_n=\dist(\cdot,\D_n)$, $\funct=\dist(\cdot,\D)$ and $X=\Ccal_{\lin}$. Indeed, \ref{theorem:equiint:hypo1} is Corollary \ref{coro:boundaryequiint}, \ref{theorem:equiint:hypo4} is the assumption $\D_n \limeq \D$ and \ref{theorem:equiint:hypo2} is satisfied by distance functions of sets, such that $\D, \D_n \cap B(0,R) \neq \emptyset$ for some $R>0$. This in turn follows from nonemptyness and $\D_n \limeq \D$. Condition \ref{theorem:equiint:hypo3} follows from the fact that the functions $\funct$ in our setting are distance functions, hence even locally Lipschitz continuous. Finally, the set $X=\Ccal_{\lin}$ is weakly closed because for a bounded sequence $v_n=(\epsilon_n,\tigma_n)\subset V$ the pressure $\pi_n$ satisfies, after suitable renormalisation,
\begin{align*}
  \norm{\pi_n}_{L_q}\le C\bra{\norm{\tigma_n}_{L_q}+\norm{f}_{W^{-1}_q}}    
\end{align*}
and is thus also bounded. Since the differential constraints \eqref{def:Ccal1} are linear, it is possible to take the limit for a subsequence.
Therefore, Theorem \ref{theorem:equiint} implies that $I_n$ $\Gamma$-converges to the $\Gamma$-limit of $I$, which is given by $I^\ast$ due to Proposition \ref{prop:relaxationlinear}.
\end{proof}

\begin{remark}
    Theorem \ref{thm:equiequivalence} establishes equivalence between data convergence and uniform convergence of $J_n$ towards $J$ if there is \textit{no} differential constraint $\A v=0$. It is not clear whether such an equivalence holds for the constrained functionals $I_n$ and $I$. Indeed, in an abstract degenerate setting, e.g. $\ker \A [\xi] =\{0\}$ for all $\xi \in \R^d \setminus \{0\}$, so that only constant functions are in $\ker \A$, it is easy to see that the equivalence does not hold. Indeed, uniform approximation for bounded/equi-integrable functions in the constraint set $\Ccal$ is equivalent to \emph{pointwise} uniform approximation on bounded sets. That is, there are $R_n \to \infty$ and $\tilde{a}_n \to 0$, such that for all $z \in \D$ with $\dist(z,0) \leq R_n$
    \[
    \dist(z,\D_n) \leq \tilde{a}_n.
    \]
    This is considerably weaker than the notions of convergence introduced in Definition \ref{def:boundedconv} and Definition \ref{def:equiconv}. A similar notion holds for fine approximation.
    Nevertheless, from a physical viewpoint, the pointwise data convergence $\D_n \limeq \D$ is a reasonable assumption and we are thus not interested in a complete characterisation of convergence for the constrained functionals.
\end{remark}

\subsection{Fluids with Inertia} \label{subsec:semilin}

In this subsection we consider the system of differential constraints, corresponding to a fluid with inertia 
\begin{equation}\tag{nlD} \label{eq:semilinear}
    \begin{cases}
     \epsilon = \tfrac 12\left(\nabla u +\nabla u^T\right)& \\
     \diverg u = 0& \\
     -\diverg \tilde{\sigma}= f-\nabla \pi -(u \cdot \nabla) u.&
    \end{cases}
\end{equation}
Regarding the boundary conditions, we make the following assumptions throughout this subsection:
\begin{enumerate}[label=(B\arabic*)]
    \item \label{B1} $\Gamma_N=\emptyset$, i.e. there are only no-slip and Navier-type boundary conditions;
    \item \label{B2} $\Gamma_D\neq \emptyset$;
    \item \label{B3} One of the following two statements is true
    \begin{enumerate}[label=(B3\alph*)]
    \item\label{it1} $p>2$;
    \item\label{it3} $g=0$ and $g_\nu=0$.  
    \end{enumerate}
\end{enumerate}
Note that assumption \ref{it3} represents the important case of a non-permeable boundary.
In comparison to the linear problem  \eqref{def:Ccal1}, the set \eqref{eq:semilinear} of differential constraints admits a \emph{direct} coupling between $\epsilon$ and $\tigma$ through the inertial term $(u \cdot \nabla) u$. For this set of constraints to still be meaningful, the inertial term $(u \cdot \nabla) u$ needs to be in the same space as $f$, $\diverg \tigma$, and $\nabla \pi$. Since $u\in W^{1}_p(\Omega;\R^d)$, for $p<d$ (otherwise we use $u\in W^{1}_r(\Omega;\R^d)$ for all $r<d$), we have by embedding $u\in L_{dp/(d-p)}(\Omega;\R^d)$ and thus $u\otimes u\in L_{dp/(2d-2p)}(\Omega;\R^{d\times d})$, which implies $(u \cdot \nabla) u=\diverg(u\otimes u)\in W^{-1}_{dp/(2d-2p)}(\Omega;\R^d)$. In order for this space to be contained in $W^{-1}_q(\Omega;\R^d)$, we must have
\begin{align}
    q=\frac{p}{p-1}\le \frac{dp}{2d-2p},
\end{align}
which implies 
\begin{equation} \label{exponentbound}
p \geq \frac{3d}{d+2}.
\end{equation}
Throughout this section we assume that \eqref{exponentbound} holds. This includes the Newtonian case $p=2$ in the physical dimensions $d=2,3$. Since we have
\begin{align*}
    p \geq \frac{3d}{d+2} \geq \frac{2d}{d+1},
\end{align*}
condition \eqref{eq:cond_p_Navier} is always satisfied. Hence, the Navier boundary condition \ref{Navier} is well-defined.

In this subsection we consider the constraint set
\begin{align}\tag{nlC}\label{eq:Ccal_semi}
    \Ccal_{\nl}\coloneqq\{(\epsilon,\tigma)\in V\colon \eqref{eq:semilinear},\ref{Dirichlet},\text{ and }\ref{Navier}\text{ are satisfied.}\}
\end{align}

\subsubsection{Coercivity in the semilinear case}
In this subsection we check that functionals of the form \eqref{eq:I_f}, with $\Ccal_{\nl}$ given by \eqref{eq:Ccal_semi}, are still coercive.

\begin{lemma}[Coercivity in the semi-linear setting]\label{lemma:coercsemi}
    Let $p \geq 3d/(d+2)$ and assume that the assumptions \ref{B1}--\ref{B3} hold. Let $\funct$ be $(p,q)$-coercive and let $\Ccal_{\nl}$ be given by \eqref{eq:Ccal_semi}. Then there are constants $C_3,C_4>0$, such that 
    \begin{equation}
        I(v) = \int_{\Omega} \funct(\epsilon,\tigma) \dd x \geq C_3 \left( \Vert \epsilon \Vert_{L_p}^p + \Vert \tigma \Vert_{L_q}^q\right)-C_4.
    \end{equation}
\end{lemma}
\begin{proof}
 Similarly to the proof of Lemma \ref{l:coercivity_omega}, we need to  estimate $\int \epsilon \cdot \tigma \dd x$, as for any $(\epsilon,\tigma) \in Y \times Y$
\begin{align}
    \funct(\epsilon,\tigma) \geq C_1  (\vert \epsilon \vert^p + \vert \tigma \vert^q)-C_2 - \gamma \epsilon \cdot \tigma.\label{eq:coerc_nonl}
\end{align}
Since $v\in \Ccal_{\nl}$, there is a $u$ such that 
\[
    \epsilon = \tfrac{1}{2}\left(\nabla u + \nabla u^T\right),
\]
for some $u$, where 
\begin{align}\label{eq:u_Korn_}
    \Vert u \Vert_{W^1_p} \leq C\bra{1+\Vert \epsilon \Vert_{L_p}}
\end{align}
due to the Korn--Poincar\'e inequality, Lemma \ref{kopo} and Lemma \ref{l:Korn}. Furthermore, we have the estimate
\begin{align}\label{eq:tigma_pi_est}
    \norm{\tigma \nu}_{W^{-1/q}_q(\partial\Omega)}+\norm{\pi \nu}_{W^{-1/q}_q(\partial\Omega)}&\le C\bra{\norm{\tigma}_{L_q}+\norm{f}_{W^{-1}_q}+\norm{u}^2_{W_p^1}},
\end{align}
which is due to $- \Div \tigma+\nabla \pi=f-\bra{u\cdot \nabla}u$.\\
   
   Indeed, repeating the calculation from the proof of Lemma \ref{l:coercivity_omega} and then using the nonlinear force balance, we obtain 
   \begin{align}
       \int_{\Omega} \epsilon \cdot \tigma \dd x&= -\int_{\Omega} u \cdot (\diverg \tigma - \nabla\pi) \dd x +\int_{\partial \Omega} u \cdot (\tigma - \pi \id)\nu\dd \mathcal{H}^{d-1} \nonumber\\
       &=\int_{\Omega} u \cdot (u \cdot \nabla) u + u\cdot f \dd x + \int_{\partial  \Omega} u \cdot (\tigma - \pi \id) \nu \dd \mathcal{H}^{d-1} \nonumber\\
       &=\int_{\Omega} \diverg\left( \frac 12 u  \vert u \vert^2\right) + u\cdot f \dd x + \int_{\partial \Omega}   u \cdot (\tigma - \pi \id) \nu \dd \mathcal{H}^{d-1}  \nonumber\\
       &= \int_{\Omega} u\cdot f \dd x+\int_{\partial \Omega}  \frac 12 (u \cdot \nu)  \vert u \vert^2 + u \cdot (\tigma - \pi \id) \nu \dd \mathcal{H}^{d-1} \label{eq:eps_tigma}.
   \end{align}
    For the first term  we use \eqref{eq:u_Korn_} to bound 
\begin{align}
   \left\vert\int_{\Omega} u \cdot f \dd x\right\vert &\le \Vert u \Vert_{W^{1}_p} \Vert f \Vert_{W^{-1}_q} 
   \le C\bra{1+\Vert \epsilon \Vert_{L_p}}  \Vert f \Vert_{W^{-1}_q}.\label{eq:f_est_nonl}
\end{align}
 For the boundary term we consider the cases \ref{it1} and\ref{it3} separately.\\
  \textbf{Case \ref{it1}:} We split $\partial\Omega=\overline{\Gamma_D\cup \Gamma_R}$ and start with
\begin{align}
    \int_{\Gamma_D} & \frac 12(u \cdot \nu) \vert u \vert^2 -  u \cdot (\tigma - \pi \id)\nu \dd \mathcal{H}^{d-1}
    =
    \int_{\Gamma_D} \frac 12(g \cdot \nu) \vert g \vert^2 -  g \cdot (\tigma - \pi \id)\nu \dd \mathcal{H}^{d-1}\nonumber\\
    &\leq \Vert g \Vert^3_{L_3(\Gamma_D)} + \Vert g \Vert_{W^{1-1/p}_p(\Gamma_D)} \bra{\Vert \tigma \nu \Vert_{W^{-1/q}_q(\Gamma_D)}+\norm{\pi\nu}_{W^{-1/q}_q(\Gamma_D)}}\nonumber\\
    &\leq C\left(1+\norm{u}^2_{W^1_p}+\norm{\tigma}_{L_q}\right)\nonumber\\
    &\le C\left(1+\norm{\epsilon}^2_{L_p}+\norm{\tigma}_{L_q}\right).\label{eq:1D}
\end{align}
 Note that $W^{1-1/p}_p(\Gamma_D)$ embeds into $L_3(\partial \Omega)$, whenever \[
       \frac{1}{3} \ge \frac{1}{p} + \frac{1-1/p}{d-1}.
       \]
This holds in view of assumption \eqref{exponentbound}.
For the other part of the boundary we estimate
\begin{align}
        \int_{\Gamma_R} & \frac 12(u \cdot \nu) \vert u \vert^2 -  u \cdot (\tigma - \pi \id)\nu \dd \mathcal{H}^{d-1} \nonumber \\
        & =\int_{\Gamma_R} \frac 12 g_\nu \vert u \vert^2 -  g_\nu \nu\cdot (\tigma - \pi \id)\nu +\lambda\abs{P_{T_x\partial\Omega}u}^2
        -P_{T_x\partial\Omega}u \cdot h_\tau\dd \mathcal{H}^{d-1}.\label{eq:1R}
\end{align}
For the terms without sign we obtain
\begin{align}
        &\left|\int_{\Gamma_R} \frac 12 g_\nu \vert u \vert^2 -  g_\nu \nu\cdot (\tigma - \pi \id)\nu -P_{T_x\partial\Omega}u \cdot h_\tau\dd \mathcal{H}^{d-1} \right| \nonumber \\
        &\quad \leq 
        \norm{g_\nu}_{L_3(\Gamma_R)}\norm{u}_{L_3(\Gamma_R)}^2+\norm{g_\nu}_{W^{1-1/p}_p(\gamma_R)}\bra{\norm{\tigma\nu}_{W^{-1/q}_q(\Gamma_R)}+\norm{\pi\nu}_{W^{-1/q}_q(\Gamma_R)}}
        \nonumber \\
        & \quad\quad
        +\norm{h_\tau}_{W^{-1/q}_q(\Gamma_R)}\norm{u}_{W^{1-1/p}_p(\Gamma_R)}\nonumber\\
        &\quad \leq C\left(1+\norm{u}^2_{W^1_p}+\norm{\tigma}_{L_q}\right) \nonumber\\
        &\quad \leq C\left(1+\norm{\epsilon}^2_{L_p}+\norm{\tigma}_{L_q}\right).\label{eq:1R2}
\end{align}
   Inserting \eqref{eq:1R} into \eqref{eq:eps_tigma} and using the result together with \eqref{eq:f_est_nonl}, \eqref{eq:1D}, \eqref{eq:1R2}, and the $(p,q)$-coercivity of $\funct$, yields
\begin{align*}
    I(v) 
    &\geq C_1 \left(\Vert \epsilon \Vert_{L_p}^p+\Vert \tigma \Vert_{L_q}^q  \right) - C_2 - \gamma \int_{\Omega} \epsilon \cdot \tigma \dd x\nonumber \\
    & \geq  C_1 \left(\Vert \epsilon \Vert_{L_p}^p+\Vert \tigma \Vert_{L_q}^q\right)- C \bra{1+\Vert \epsilon \Vert^2_{L_p}+\Vert \tigma \Vert_{L_q}} \nonumber\\
    &\geq \frac{C_1}{2} \left(\Vert \epsilon \Vert_{L_p}^p+\Vert \tigma \Vert_{L_q}^q \right) -C,
\end{align*}
where we use Young's inequality and the fact that $p>2$.\\

\noindent\textbf{Case \ref{it3}:} Since $g=0$ and $g_\nu=0$, the boundary term simplifies to
 \begin{align}
         \int_{\partial \Omega} \frac 12(u \cdot \nu) \vert u \vert^2 -  u \cdot (\tigma - \pi \id)\nu \dd \mathcal{H}^{d-1} 
         & =-\int_{\Gamma_R} P_{T_x \partial\Omega}u \cdot P_{T_x \partial\Omega}(\tigma \nu) \dd \mathcal{H}^{d-1} \nonumber \\
         &= \int_{\Gamma_R} \lambda\abs{P_{T_x \partial\Omega}u}^2 -P_{T_x \partial\Omega}u\cdot h_\tau \dd \mathcal{H}^{d-1}\label{eq:3R}.
\end{align}
For the term without sign we obtain
\begin{align}
    \left\vert\int_{\Gamma_R} P_{T_x \partial\Omega}u\cdot h_\tau \dd \mathcal{H}^{d-1}\right\vert
    \le \Vert u \Vert_{W^{1-1/p}_p(\Gamma_R)}\norm{h_\tau}_{W^{-1/q}_q(\Gamma_R)}
    \leq C \left(1+\norm{\epsilon}_{L_p}\right)\label{eq:3R2}
\end{align}
By inserting \eqref{eq:3R} into \eqref{eq:eps_tigma} and using \eqref{eq:f_est_nonl}, \eqref{eq:3R2} and the $(p,q)$-coercivity of $\funct$, we obtain
\begin{align*}
    I(v) 
    &\geq 
    C_1 \left(\Vert \epsilon \Vert_{L_p}^p+\Vert \tigma \Vert_{L_q}^q  \right)-C_2 - \gamma\int_{\Omega} \epsilon \cdot \tigma \dd x\nonumber \\
    & \geq  
    C_1 \left(\Vert \epsilon \Vert_{L_p}^p+\Vert \tigma \Vert_{L_q}^q\right)- C\bra{1+\Vert \epsilon \Vert_{L_p}} \nonumber\\
    &\geq 
    \frac{C_1}{2} \left(\Vert \epsilon \Vert_{L_p}^p+\Vert \tigma \Vert_{L_q}^q \right) -C,
\end{align*}
where we use again Young's inequality.

    \end{proof} 
  \bigskip
\subsubsection{Continuity of $\,\Theta(u) = u \otimes u $}
  To verify the assumptions of Theorem \ref{thm:relaxationnonlinear}, in particular the weak closedness of $\Ccal_{\ln}$, we show that the map 
  \[
  u \longmapsto u \otimes u
  \]
   is continuous from the weak topology of $W^{1}_p(\Omega;\R^d)$ to the strong topology of $L_r(\Omega;Y)$ for some $r>q$.
   
\begin{lemma}\label{l:theta}
   Let $p>3d/(d+2)$. Then there is an $r > q = p/(p-1)$, such that $\Theta$ is continuous from $W^{1}_p(\Omega;\R^d)$, equipped with the weak topology, into to $L_r(\Omega;Y)$.
\end{lemma}
In view of Korn's inequality (Lemma \ref{kopo}) bounded sets in $L_p(\Omega;Y)$ are mapped to bounded sets in $W^1_p(\Omega;\R^d)$ by the map $\epsilon\mapsto u$. Hence, the map $\Theta$ might also be seen as a map $\epsilon \mapsto u \otimes u$. 
 
\begin{proof} For $p\geq d$ the result immediately follows from the case $p<d$ by first embedding into $W^{1}_\tau(\Omega;\R^d)$ for some $\tau<d$. Thus, let $p<d$. Then $W^{1}_p(\Omega;\R^d)$ embeds compactly into $L_s(\Omega;\R^d)$ for all $s < dp/(d-p)$. In particular, for every weakly convergent sequence $u_n \subset W^{1}_p(\Omega;\R^d)$, the sequence 
\[
   \Theta(u_n)= u_n \otimes u_n
\]
converges strongly in $L_r(\Omega;\R^d)$ for $r<dp/(2d-2p)$. This can be satisfied at the same time as $r>q=p/(p-1)$ if and only if $p>3d/(d+2)$.

\end{proof}
   
\subsubsection{$\Gamma$-convergence with semilinear constraint.}

\begin{theorem}[$\Gamma$-convergence in the semilinear setting] \label{thm:gammaconvsemilin}
Let $\D_n,\D \subset Y\times Y$ be closed, nonempty sets and let $\Ccal_{\nl}$ be given by \eqref{eq:Ccal_semi}. Moreover, suppose that: 
\begin{enumerate} [label=(\roman*)]
    \item The distance functions to $\D_n$ and $\D$ are uniformly $(p,q)$-coercive, i.e. there are $c_1,c_2$, such that 
    \[
    \D_n,\D \subset  \{ (\epsilon,\tigma) \in V \times V \colon c_1 \epsilon \cdot \tigma +c_2 > \vert \epsilon \vert^p + \vert \tigma \vert^q\};
    \]
    \item $\D_n \limeq \D$;
    \item $p> \frac{3d}{d+2}$;
    \item assumptions \ref{B1}--\ref{B3} hold.
\end{enumerate}
Then the functional $I_n$ $\Gamma$-converges to $I^{\ast}$, where \[
I^{\ast}(v) = \begin{cases}
       \int_{\Omega} \QA \dist(v,\D) \dd x, & v \in \Ccal_{\nl} \\
       \infty, & \text{else.}
\end{cases}
\]
\end{theorem}

\begin{proof}
The proof is very similar to the proof of Theorem \ref{thm:gammaconvlin}. Indeed, as the constraint set $\Ccal_{\nl}$ is weakly closed by Lemma \ref{l:theta}, the only difficulty, given $v \in \Ccal_{\nl}$, is to find a recovery sequence lying in $\Ccal_{\nl}$. This is achieved in Theorem \ref{thm:relaxationnonlinear}.
\end{proof}


\section{Consistency of data-driven solutions and PDE solutions in the case of material law data} \label{sec:exact}

In this section we consider data that are given by a \emph{constitutive law}, i.e. 
\[
    \tigma = 2\mu(\vert \epsilon \vert) \epsilon, 
    \quad \epsilon \in Y,
\]
for a viscosity $\mu \colon \R \to \R$. We compare the solutions obtained by the \emph{classical PDE} approach to minimisers of the \emph{data-driven} functional. As before, we assume $\Gamma_N=\emptyset$ and call a pair $(\epsilon,\tigma) \in L_p(\Omega;Y) \times L_q(\Omega;Y)$ a \emph{weak solution} to the stationary Navier--Stokes equation, if there is $u \in W^{1}_p(\Omega;\R^d)$ and a pressure $\pi \in L_q(\Omega)$, such that
\begin{equation} \label{weakPDE}
 \begin{cases}
        \epsilon = \tfrac{1}{2} \bigl(\nabla u + \nabla u^T\bigr), 
        & x \in \Omega
        \\
        \diverg u = 0, 
        & x \in \Omega
        \\
        (u \cdot \nabla)u-\diverg(2\mu(|\epsilon|)\epsilon) +\nabla\pi=f, & x \in \Omega
        \\
        \ref{Dirichlet}, \ref{Navier}, & x\in \partial \Omega,
    \end{cases}
\end{equation}
where $\eqref{weakPDE}_3$ has to be satisfied in $W^{-1}_q(\Omega;\R^d)$.
Note that the system \eqref{weakPDE} is equivalent to
\begin{equation} \label{weakPDE2}
 \begin{cases}
        \epsilon = \tfrac{1}{2} \bigl(\nabla u + \nabla u^T\bigr), 
        & x \in \Omega
        \\
        \diverg u = 0, 
        & x \in \Omega\\
        -\diverg \tilde{\sigma} 
        = f - \nabla \pi -(u \cdot \nabla) u, 
        & 
        x \in \Omega
        \\ 
        \tigma = 2\mu(\vert \epsilon \vert) \epsilon, & x \in \Omega \\
        \ref{Dirichlet}, \ref{Navier}, & x\in \partial \Omega.
    \end{cases}
\end{equation}
We may interpret the convergence of data sets discussed in Section \ref{sec:dataconv} as an increase of the accuracy of measurement. If a constitutive law exists, then the limit $\D$ of data sets $\D_n$ should represent this law. 
Since we assume that the set $\D$ is given 
by a constitutive law $\epsilon \mapsto \tigma_c(\epsilon)$, we consider data sets
\begin{equation} \label{datasetgraph}
    \D = \{ (\epsilon,\tigma) \colon \tigma= \tigma_c(\epsilon)\}.
\end{equation}
For typical constitutive laws, a solution to the induced partial differential equation \eqref{weakPDE2} exists and it is natural to ask whether (approximate) solutions to the data-driven problem with $\D_n$ converge to a solution of \eqref{weakPDE2}. It turns out that this is true if the constitutive relation is monotone.
Indeed, assume that $(\epsilon,\tigma) \in \Ccal_{\nl}$, i.e. that the differential constraints
\begin{equation*}
    \begin{cases}
        \epsilon = \tfrac{1}{2} \bigl(\nabla u + \nabla u^T\bigr), 
        & x \in \Omega
        \\
        \diverg u = 0, 
        & x \in \Omega\\
        -\diverg \tilde{\sigma} 
        = f- \nabla \pi - (u \cdot \nabla) u, 
        & 
        x \in \Omega
    \end{cases}
\end{equation*}
are satisfied. If in addition $I(u)=0$, and thus $u$ is a minimiser, then we have
\begin{equation*}
    (\epsilon,\tigma) \in \D = \{ (\epsilon,\tigma) \colon \tigma= \tigma_c(\epsilon)\}
    \quad 
    \text{almost everywhere}.
\end{equation*}
Consequently, a minimiser of $I$ satisfying $I(u)=0$ is a solution to the partial differential equation. Conversely, given a constitutive law $\tigma_c$ and a weak solution to the partial differential equation \eqref{weakPDE2}, we may construct the  set $\D$ as in $\eqref{datasetgraph}$ and observe that any solution to the partial differential equation \eqref{weakPDE2} is also a minimiser of $I$.

 
If the data set $\D$ is a limit of measurement data sets $\D_n$, it is not clear a priori whether a sequence of (approximate) minimisers $u_n$ of $I_n$ converges weakly to a solution $u$ to the partial differential equation because we can only infer $I^{\ast}(u)=0$ and \emph{not} $I(u)=0$. This is addressed in the following proposition, which directly follows from the relaxation statement Theorem \ref{thm:gammaconvsemilin}.


\begin{proposition}
Let $p > 3d/(d+2)$ and let $\epsilon \mapsto \tigma_c(\epsilon)$ be a given constitutive law. Moreover, assume that the corresponding data set $\D$ is given by \eqref{datasetgraph}, such that the distance function $\dist(\cdot,\cdot)$ is $(p,q)$-coercive.
If the partial differential equation \eqref{weakPDE2} admits a weak solution $v$, i.e. $\min_{v \in \Ccal} I(v)=0$, then a function $v^{\ast}$ is a minimiser of $I^{\ast}$ if and only if 
\begin{equation*}
    v^{\ast} \in \{\QA \dist((\epsilon,\tigma),\D)= 0\}
\end{equation*}
almost everywhere.
Moreover, if 
\begin{equation} \label{qcset}
    \{\QA \dist((\epsilon,\tigma),\D)= 0\} = \D,
\end{equation}
then any such approximate solution $v^{\ast}$ is already a solution to the partial differential equation \eqref{weakPDE2}.
\end{proposition}


In the following we characterise some constitutive laws satisfying \eqref{qcset}. To this end, we study the set
\[
    \{ \QA \dist((\epsilon,\tigma),\D\} =0 \}.
\]

\begin{definition}
Let $1<p<\infty$ and $q=p/(p-1)$. For a set $\D \subset Y \times Y$ we define the \textbf{$\A$-$(p,q)$-quasiconvex hull} of $\D$ as 
\[
    \D^{(p,q)} = \left\{ (\epsilon,\tigma) \in Y \times Y \colon \QA \dist((\epsilon,\tigma),\D)=0 \right\}.
\]
We call a set $\D \subset Y \times Y$ \textbf{$\A$-$(p,q)$-quasiconvex} if $\D=\D^{(p,q)}$.
\end{definition}


\bigskip
\subsection{Newtonian fluids}
In the Newtonian setting the fluid's viscosity is constant, i.e. $\mu(|\epsilon|) \equiv \mu_0 > 0$ and hence the relation between the local strain $\epsilon$ and the viscous stress $\tigma$ is linear with $\tigma = 2\mu_0 \epsilon$. In the following, we assume without loss of generality that $\mu_0=1/2$. That is, we have $p=q=2$ and the constitutive law is given by the data set
\[
    \D_{\Newt} = \{(\epsilon,\epsilon)\colon \epsilon \in Y\} \subset Y \times Y.
\]
Note that, in terms of $\epsilon$ and $\tigma$, the Newtonian data set $\D_{\Newt}$ and the distance function $\dist(\cdot,\cdot)$ can be written as 
\begin{equation*}
   \D_{\Newt}
   =
   \left\{(\epsilon,\tigma) \colon \epsilon \cdot \tigma= \tfrac{1}{2} \left(\vert \epsilon \vert^2 + \vert  \tigma \vert^2\right)\right\} 
   \quad 
   \text{and} 
   \quad
   \dist((\epsilon,\tigma), \D_{\Newt}) = \tfrac{1}{2} \vert \epsilon - \tigma \vert^2.
\end{equation*}
Since in this case $\dist((\cdot,\cdot),\D_{\Newt})$ is already a convex function, it is also $\A$-quasiconvex and we have
\[
    \QA \dist((\epsilon,\tigma),\D_{\Newt}) 
    = 
    \dist((\epsilon,\tigma),\D_{\Newt}).
\]
Consequently, we observe that the $\A$-$(2,2)$-quasiconvex hull $\D^{(2,2)}_{\Newt}$ of $\D_{\Newt}$ is given by
\begin{equation*}
  \D^{(2,2)}_{\Newt}
    = 
    \left\{(\epsilon,\tigma) \colon \dist((\epsilon,\tigma),\D_N) = 0\right\} 
    = 
    \D_{\Newt}.
\end{equation*}
Therefore, any solution to the data-driven problem for Newtonian fluids is also a weak solution to the partial differential equation, in the sense that $u \in W^{1,p}(\Omega;\R^d)$ satisfies
\begin{equation*}
    \begin{cases}
        (u \cdot \nabla) u 
        = 
        -\nabla \pi + \frac 12\Delta u, & x \in \Omega
        \\
        \diverg u = 0, & x \in \Omega
    \end{cases}
\end{equation*}
and the boundary conditions \ref{Dirichlet}, \ref{Navier}.

\bigskip

\subsection{Power-law fluids}
In the case of power-law fluids, the constitutive law for the fluid's viscosity is $\mu(|\epsilon|) = \mu_0 |\eps|^{\alpha-1}\eps$ with given flow-consistency index $\mu_0 > 0$ and  flow-behaviour exponent $\alpha > 0$. Consequently, we have $\tigma = 2\mu_0 |\epsilon|^{\alpha-1}$. As above, we set without loss of generality $\mu_0 =1/2$. In the previously used notation, we thus
consider $1 < p < \infty$, $q=p/(p-1)$ and $\alpha = p/q = 1/(p-1)$ and suppose that the material law is given by the data set 
\[
    \D_{\Pow}
    =
    \left \{(\epsilon,\vert \epsilon \vert^{\alpha-1} \epsilon) \colon \epsilon \in Y \right \} \subset Y \times Y.
\]
Observe that, for $\alpha\neq 1$, the set $\D_\Pow$ is \textit{not} convex. Consequently, also the corresponding distance function is not convex. However, 
\begin{equation*}
    (\epsilon,\tigma) \in \D_{\Pow} \Longleftrightarrow 
    \epsilon \cdot \tigma 
    = 
    \tfrac{1}{p} \vert \epsilon \vert^p 
    + 
    \tfrac{1}{q} 
    \vert \tigma \vert^q.
\end{equation*}
It turns out that the $\A$-$(p,q)$-quasiconvex hull $\D^{(p,q)}_\Pow$ of $\D_\Pow$ in fact coincides with the data set $\D_\Pow$. In order to verify this, we rely on the following observation (see also \cite{Yan3}).


\begin{lemma} \label{lem:distpower}
Let $\dist(\cdot,\D)$ be $(p,q)$-coercive. Then 
\[
    \D^{(p,q)} = \bigcap_{\funct \in T_{p,q}} \{\funct(z) \leq 0\},
\]
where $T_{p,q}$ is the set of all continuous functions $\funct \in C(Y \times Y)$ satisfying  \begin{itemize}
    \item $\funct$ is $\A$-quasiconvex;
    \item $\funct(z) \leq 0$ for all $z \in \D$;
    \item $\vert \funct(\epsilon,\tigma) \vert \leq C(1+\vert \epsilon \vert^p + \vert \tigma \vert^q)$.
\end{itemize}
\end{lemma}


\begin{proof}
\noindent \textbf{`$\mathbf{\supseteq}$':} Since $\QA \dist(\cdot,\D)$ is contained in $T_{p,q}$, it is clear that  $\bigcap_{\funct \in T_{p,q}} \{\funct(z) \leq 0\}$ is a subset of $\D^{(p,q)}$. 

\noindent\textbf{`$\mathbf{\subseteq}$':}  Suppose now that $(\epsilon_0,\tigma_0) \in \D^{(p,q)}$. Then there exists a sequence $(\epsilon_n,\tigma_n) \in L_p(\T_d;Y) \times  L_q(\T_d;Y)$ with zero average, satisfying the differential constraint such that \begin{equation} \label{eq:distconv}
    \int_{\T_d} \dist\bigl(\bigl(\epsilon_0+\epsilon_n(x),\tigma_0 +\tigma_n(x)\bigr),\D\bigr) \dd x < \frac{1}{n},
    \quad n \in \N.
\end{equation}
Due to the coercivity of the distance function we can bound 
\[
    \Vert \epsilon_n \Vert_{L_p} + \Vert \tigma_n \Vert_{L_q} \leq C( 1 + \vert \epsilon_0 \vert^p + \vert \tigma_0 \vert^q),
    \quad n \in \N.
\]
Take now $\funct \in T_{p,q}$. Then $\funct$ is locally Lipschitz continuous thanks to  Proposition \ref{prop:FM} \ref{loc:Lipschitz}. Define $w_n = (\epsilon'_n,\tigma_n')$ as the projection of $(\epsilon_0+\epsilon_n,\tigma_0+\tigma_n)$ onto $\D$. Then, in view of  \eqref{eq:distconv} we find that,
\[
    \Vert \epsilon_0+\epsilon_n - \epsilon_n' \Vert_{L_p} 
    \longrightarrow 
    0
    \quad \text{and} \quad 
    \Vert \tigma_0+\tigma_n - \tigma_n' \Vert_{L_q}
    \longrightarrow 
    0.
\]
The local Lipschitz continuity of $\funct$ and the boundedness of $(\epsilon_n,\tigma_n)$ now imply 
\begin{equation} \label{eq:aux_6}
    \left \vert \int_{\T_d} \funct(\epsilon_0+\epsilon_n,\tigma_0+\tigma_n) - \funct(\epsilon_n',\tigma_n') \dd x \right \vert 
    \longrightarrow 
    0 
    \quad \text{as } n \to \infty. 
\end{equation}
Using $\A$-quasiconvexity of $\funct$, \eqref{eq:aux_6}, and the non-positivity of $\funct$ this implies 
\[
    \funct(\epsilon_0,\tigma_0)
    \le 
    \liminf_{n\to \infty}\int_{\T_d} \funct(\epsilon_0+\epsilon_n,\tigma_0+\tigma_n) \dd x
    \leq
    \liminf_{n\to \infty}\int_{\T_d} \funct(\epsilon_n',\tigma_n')\dd x 
    \le 0. 
\]
Eventually, we find that $(\epsilon_0,\tigma_0) \in\bigcap_{\funct \in T_{p,q}} \{\funct(z) \leq 0\}$ and the proof is complete. 
\end{proof}


\begin{corollary} \label{cor:powerlaw}
Let $p,q,\alpha$ and $\D_{\Pow}$ be as before. Then \[
    \D_\Pow^{(p,q)} = \D_\Pow.
\]
\end{corollary}


\begin{proof}
Lemma \ref{lem:distpower} implies that we only need to find a function $\funct$, which is $\A$-quasiconvex, is non-positive in $(\epsilon,\tigma)$ if and only if $(\epsilon,\tigma) \in  \D_{\Pow}$ and has $(p,q)$-growth. The function \[
\funct(\epsilon,\tigma) := \tfrac{1}{p} \vert \epsilon \vert^p + \tfrac{1}{q} \vert \tigma \vert^q - \epsilon \cdot \tigma \]
exactly satisfies these assertions. Therefore, $\D_{\Pow}^{(p,q)} = \D_{\Pow}$.
\end{proof}

\bigskip

\subsection{Monotone material laws}

Again, consider $1 < p < \infty$, $q=p/(p-1)$ and $\alpha=p/q$. We consider a constitutive law 
\begin{equation}\label{eq:tigma_with_2}
    \tigma( \epsilon) = 2\mu (\vert \epsilon \vert) \epsilon
\end{equation}
for a viscosity $\mu \in C\bigl(\R_+;\R_+\bigr)$. For better readability we omit the factor $2$ in \eqref{eq:tigma_with_2} in the following calculations.
Furthermore, throughout this subsection we assume that the material law $\tigma(\cdot)$ is \emph{monotone}, i.e. for all $\epsilon_1,\epsilon_2 \in Y$ we have
     \[
     (\epsilon_1 - \epsilon_2) \cdot (\tigma(\epsilon_1)-\tigma(\epsilon_2)) \geq 0;
     \]
and we denote $a \coloneqq \lim_{s \to 0} \mu(s) s$. 
The \emph{data set} $\D_{\Mon}$ corresponding to the constitutive law $\epsilon \mapsto \tigma(\epsilon)$ is given as follows (cf. Figure \ref{fig:5:1}):
\begin{equation} \label{def:DMon}
     \D_{\Mon} = \overline{\D}_{\epsilon} \cup  \D_0, \quad \D_{\epsilon} = \bigl \{ (\epsilon,\tigma(\epsilon)) \colon \epsilon \in Y \setminus \{0\} \bigr \}, \quad \D_0 = \bigl\{(0,\tigma) \colon \vert \tigma \vert \leq a \bigr\}.
 \end{equation}
 
 \begin{remark} \begin{enumerate} [label=(\roman*)]
 \item Monotonicity of such a radial-symmetric function $\tigma(\epsilon)$ is equivalent to monotonicity of its one-dimensional counterpart \[
 s \longmapsto \mu(s) s.
 \]
 Therefore, the limit $a= \lim_{s\to 0} \mu(s)s$ is well-defined.
 \item The setting includes the previously discussed cases of Newtonian and power-law fluids, as well as Ellis-law fluids \cite{WS}. Furthermore, it allows the strain-stress graph to have a discontinuity at zero, so-called Herschel-Bulkley fluids, cf. \cite{MRS05}.
 \end{enumerate}
 \end{remark}


\begin{center}
\begin{figure}[h!]
\begin{tikzpicture}[yscale=1.25, xscale=1.25]
\draw[->, very thick] (-4,0)--(4,0) node[right] {$\epsilon$};
\draw[->, very thick] (0,-2.5)-- (0,2.5) node[above] {$\tigma(\epsilon)$};
\draw[-,magenta,ultra thick,] (0,-1)-- (0,1);
\draw[domain=0:3,variable=\x,magenta, ultra thick] plot ({\x},{1+0.13*(\x+0.3) * (\x+0.3) });
\draw[domain=-3:0,variable=\x,magenta, ultra thick] plot ({\x},{-1-0.13*(-\x+0.3) * (-\x+0.3) });
\draw[-, dashed, gray, thick] (-4,1.52)--(4,1.52);
\draw[-, dashed, gray, thick] (1.7,-2.5)--(1.7,2.5); 
\fill[luh-dark-blue!10] (0.02,1.53) rectangle (1.69,2.4);
\fill[luh-dark-blue!10] (-4,1.53) rectangle (-0.02,2.4);
\node at (0.85,2.0) {$\funct_0 \geq 0$};
\fill[luh-dark-blue!10] (1.71,0.02) rectangle (3.89,1.51);
\fill[luh-dark-blue!10] (1.71,-0.02) rectangle (3.89,-2.5);
\node at (2.85,0.75) {$\funct_0 \geq 0$};
\filldraw[black] (1.7,1.52) circle (1.5pt) node[anchor=north west]{$(\epsilon_0,\tigma_0)$};
\end{tikzpicture}
\caption{A monotone material set $\D_{\Mon}$ and the separating function $\funct_0$ for a given $(\epsilon_0,\tigma_0) \in \D_{\Mon}$.} \label{fig:5:1}
\end{figure}
\end{center}

\begin{theorem}
Let $p,q,\alpha$ and $\D_{\Mon}$ be as above. Then we have \[
\D_{\Mon}^{(p,q)} = \D_{\Mon}.
\]
\end{theorem}
 \begin{proof}
 As for the proof of Corollary \ref{cor:powerlaw} for the power-law case, it suffices to find $\A$-quasiconvex separating functions (Lemma \ref{lem:distpower}). For $(\epsilon_0,\tigma_0) \in \D_{\Mon}$ we define the function (cf. Figure \ref{fig:5:1}).
\[
 \funct_0(\epsilon,\tigma)= -(\epsilon-\epsilon_0)\cdot(\tigma-\tigma_0).
\] 
This function is $\A$-quasiconvex (even $\A$-quasiaffine, i.e. $\funct$ and $-\funct$ are $\A$-quasiconvex) and has $(p,q)$-growth, as \[
\vert \funct_0(\epsilon,\tigma) \vert\leq \tfrac{1}{p} \vert \epsilon-\epsilon_0 \vert^p + \tfrac{1}{q} \vert \tigma-\tigma_0 \vert^q.
\]
To conclude that $\D_{\Mon}^{(p,q)} = \D_{\Mon}$ we still need to show that \begin{enumerate}[label=(\roman*)]
    \item \label{lastthm:proof1} $\funct_0$ is non-positive on $\D_{\Mon}$;
    \item \label{lastthm:proof2} for all $(\epsilon,\tigma) \notin \D_{\Mon}$ there is $(\epsilon_0,\tigma_0) \in \D_{\Mon}$, such that $\funct_0(\epsilon,\tigma) >0$.
\end{enumerate}
\textbf{\ref{lastthm:proof1}:} Take $(\varepsilon,\tigma) \in \D$. Suppose that $\vert \varepsilon \vert \geq \vert \varepsilon_0 \vert$ (the other case is rather similar). Then \begin{align*}
    -\funct_0(\epsilon,\tigma) &= (\epsilon-\epsilon_0) \cdot (\tigma - \tigma_0) \\
    &=(\epsilon-\epsilon_0) \cdot \left(\mu(\vert \epsilon \vert) \epsilon - \mu(\vert \epsilon_0 \vert) \epsilon_0 \right)
    \\
    &= \mu(\vert \epsilon_0 \vert) (\epsilon -\epsilon_0) \cdot (\epsilon -\epsilon_0) + (\epsilon -\epsilon_0) \cdot \left(\bigl(\mu(\vert \epsilon_0 \vert) - \mu(\vert \epsilon_0 \vert)\bigr) \epsilon\right)\\
    &\geq 0 + \bigl(\mu(\vert \epsilon_0 \vert) - \mu(\vert \epsilon_0 \vert)\bigr) \bigl(\vert \epsilon \vert^2 - \vert \epsilon \vert \vert \epsilon_0 \vert\bigr)  \geq 0
\end{align*}
\textbf{\ref{lastthm:proof2}:} Suppose that $(\epsilon,\tigma) \notin \D_{\Mon}$. If $\epsilon \neq 0$, this means that $\tigma \neq \mu(|\varepsilon|) \varepsilon$. In that case, consider \[
\epsilon_t = \epsilon +t(\tigma - \mu(|\epsilon|)\epsilon) 
\]
and $\tigma_t= \mu(|\epsilon_t|) \epsilon_t$. If $\varepsilon=0$, simply take $\epsilon_t= t e_{11}$. For now, take $\epsilon \neq 0$, the other case is quite similar. Then for $t<0$ small enough
\[
-\funct_t(\epsilon,\tigma) =(\epsilon-\epsilon_0) \cdot (\tigma - \tigma_t) 
= t (\tigma - \mu(\vert \epsilon \vert)\epsilon) \cdot (\tigma - \mu(|\epsilon_t|) \epsilon_t)
 <0
\]
as the map \[
t \mapsto  (\tigma - \mu(|\epsilon_t|) \epsilon_t)
\]
is continuous. Hence, there is $t<0$, such that 
\[
 (\tigma - \mu(\vert \epsilon \vert)\epsilon) \cdot (\tigma - \mu(|\epsilon_t|) \epsilon_t) >0.
\]
To summarise, there is a function $\funct_t \in T_{p,q}$, such that $\funct_t(\epsilon,\tigma)>0$, whenever $(\epsilon,\tigma) \notin \D_{\Mon}$.
\end{proof}
\begin{remark}
Starting from the constitutive law $\epsilon \mapsto \tigma_c(\varepsilon)$, there are two choices for $\D_{\Mon}$. We may define $\D_{\Mon}$ as in \eqref{def:DMon} or only take the set $\overline{\D}_{\varepsilon}$ introduced in \eqref{def:DMon}. For the $\A$-quasiconvex hull this does not make a difference, i.e. 
\begin{equation} \label{Dmonalt}
    \overline{\D}_\varepsilon^{(p,q)} = \D_{\Mon}^{(p,q)}= \D_{\Mon}.
\end{equation}
Indeed, \eqref{Dmonalt} can be verified by calculating the \emph{$\Lambda_{\A}$-convex hull} of the set $\overline{\D}_{\varepsilon}$ (that is, we successively take convex combinations along $\Lambda_{\A}$). The $\Lambda_{\A}$-convex hull is a subset of the $\A$-quasiconvex hull. Therefore, it suffices to show that the $\Lambda_{\A}$-convex hull of $\overline{\D}_{\varepsilon}$ contains $\D_{\Mon}$. This in turn follows from the fact that \[
\ker \A_2[\xi] = \{\tigma \in Y \colon \tigma  \xi =0 \} + \R (\xi \otimes \xi) \quad \Longrightarrow  \quad \Lambda_{\A_2} =Y.
\]
Using this observation, the $\Lambda_{\A}$-convex hull of $\{(0,\tigma) \colon \vert \tigma \vert = a\} \subset \overline{\D}_{\varepsilon}$ is the convex hull $\D_0$. Consequently, the $\Lambda_{\A}$-convex hull and therefore also the $\A$-quasiconvex hull of $\overline{\D}_{\varepsilon}$ contain $\D_{\Mon}$.
\end{remark}

\bigskip

\bibliographystyle{alpha}
\bibliography{DDFD} 
\end{document}